\documentclass{amsart}
\usepackage{amsmath} 
\usepackage{amssymb}
\usepackage{mathrsfs}

\newtheorem{theorem}{Theorem}[section]

\newtheorem{claim}[theorem]{Claim}

\newtheorem{trichotomy}[theorem]{The BB Trichotomy Theorem}
\newtheorem{smain claim}[theorem]{The Section Main Claim}

\theoremstyle{definition}
\newtheorem{definition}[theorem]{Definition}

\newtheorem{fact}[theorem]{Fact}

\newtheorem{question}[theorem]{Question}
\newtheorem{observation}[theorem]{Observation}
\newtheorem{obs}[theorem]{Observation}

\theoremstyle{remark}
\newtheorem{remark}[theorem]{Remark}
\newtheorem{discussion}[theorem]{Discussion}
\newtheorem{conclusion}[theorem]{Conclusion}

\newtheorem{conjecture}[theorem]{Conjecture}

\newtheorem{notation}[theorem]{Notation}

\newcount\skewfactor
\def\mathunderaccent#1#2 {\let\theaccent#1\skewfactor#2
\mathpalette\putaccentunder}
\def\putaccentunder#1#2{\oalign{$#1#2$\crcr\hidewidth
\vbox to.2ex{\hbox{$#1\skew\skewfactor\theaccent{}$}\vss}\hidewidth}}

\def\smallbox#1{\leavevmode\thinspace\hbox{\vrule\vtop{\vbox
   {\hrule\kern1pt\hbox{\vphantom{\tt/}\thinspace{\tt#1}\thinspace}}
   \kern1pt\hrule}\vrule}\thinspace}

\newcommand{\rest}{{\restriction}}
\newcommand{\dom}{{\rm dom}} 
 
\newcommand{\nst}{{\rm nst}} 
\newcommand{\Hom}{{\rm Hom}} 
\newcommand{\Sep}{{\rm Sep}} 
\newcommand{\spp}{{\rm sp}} 
\newcommand{\Sp}{{\rm Sp}} 
\newcommand{\SP}{{\rm SP}} 
\newcommand{\RK}{{\rm RK}} 
\newcommand{\TDU}{{\rm TDU}} 
\newcommand{\Sol}{{\rm Sol}} 
\newcommand{\Reg}{{\rm Reg}} 
\newcommand{\otp}{{\rm otp}} 
\newcommand{\cov}{{\rm cov}} 
\newcommand{\Ord}{{\rm Ord}} 
\newcommand{\BB}{{\rm BB}} 
\newcommand{\tr}{{\rm tr}} 
\newcommand{\tcf}{{\rm tcf}} 
\newcommand{\bd}{{\rm bd}} 
\newcommand{\ged}{{\rm gd}} 
\newcommand{\cd}{{\rm cd}} 
\newcommand{\pp}{{\rm pp}} 
\newcommand{\issp}{{\rm issp}} 
 
\newcommand{\ussp}{{\rm ussp}} 
\newcommand{\cf}{{\rm cf}}

\newcommand{\Rang}{{\rm Rang}}
\newcommand{\comp}{{\rm comp}}

\newcommand{\wilog}{{\rm without loss of generality}}
\newcommand{\Wilog}{{\rm Without loss of generality}}

\newcommand{\then}{{\underline{then}}}
\newcommand{\when}{{\underline{when}}}
\newcommand{\Then}{{\underline{Then}}}

\newcommand{\Iff}{{\underline{Iff}}}
\newcommand{\mn}{{\medskip\noindent}}
\newcommand{\sn}{{\smallskip\noindent}}
\newcommand{\bn}{{\bigskip\noindent}}

\newcommand{\cW}{{\mathcal W}}

\newcommand{\gd}{{\mathfrak d\/}} 

\newcommand{\cH}{{\mathcal H}}

\newcommand{\cF}{{\mathcal F}}
\newcommand{\cG}{{\mathcal G}}

\newcommand{\cP}{{\mathcal P}}

\newcommand{\bbQ}{{\mathbb Q}}

\newcommand{\cT}{{\mathcal T}}
 
\newcommand{\cU}{{\mathcal U}}

\newcommand{\pr}{{\rm pr}}

\newcount\skewfactor
\def\mathunderaccent#1#2 {\let\theaccent#1\skewfactor#2
\mathpalette\putaccentunder}
\def\putaccentunder#1#2{\oalign{$#1#2$\crcr\hidewidth
\vbox to.2ex{\hbox{$#1\skew\skewfactor\theaccent{}$}\vss}\hidewidth}}

\newcommand{\Dom}{{\rm Dom}}

\newcommand{\wolog}{{without loss of generality}}

\newcommand{\bbZ}{\Bbb Z}

\newcommand{\mat}{\mathcal}

\newenvironment{PROOF}[2][\proofname.]
   {\begin{proof}[#1]\renewcommand{\qedsymbol}{\textsquare$_{\rm #2}$}}
   {\end{proof}}

\begin{document}

\title{pcf and abelian groups}
\author{Saharon Shelah}
\address{Institute of Mathematics
The Hebrew University of Jerusalem
Einstein Institute of Mathematics \\
Edmond J. Safra Campus, Givat Ram \\
 Jerusalem 91904, Israel,
 and:  Department of Mathematics
 Hill Center-Busch Campus \\
 Rutgers, The State University of New Jersey
 110 Frelinghuysen Road, Piscataway,
 New Brunswick, NJ 08854, USA}
\email{shelah@math.huji.ac.il}
\urladdr{http://www.math.rutgeb rs.edu/\char`\~shelah}
\thanks{First typed: 06.Dec.06 - The author would like to thank the
  Israel Science Foundation for partial support of this research
  (Grant No. 1053/11, 710/07).  Research also supported by 
German-Israeli Foundation for Scientific Research and Development. Paper 898}


\date{December 1, 2013}



\let\labeloriginal\label
\let\reforiginal\ref

\begin{abstract}
We deal with some pcf (possible cofinality theory) investigations
mostly motivated by questions in abelian group theory.  We concentrate
on applications to test problems but we expect the combinatorics will
have reasonably wide applications.  The main test problem is the
``trivial dual conjecture" which says that there is a quite free abelian
group with trivial dual.  The ``quite free" stands
for ``$\mu$-free" for a suitable cardinal $\mu$, the first open case is
$\mu = \aleph_\omega$.
We almost always answer it positively, that is, prove 
the existence of $\aleph_\omega$-free Abelian groups 
with trivial dual, i.e., with no non-trivial homomorphisms to the integers.  
Combinatorially, we prove that ``almost
always" there are ${\mat F}\subseteq {}^\kappa  \lambda$ which are 
quite free and have a relevant black box.  
The qualification ``almost always" 
means except when we have strong restrictions on cardinal
arithmetic, in fact restrictions which hold ``everywhere".
The nicest combinatorial result is probably the so called 
``Black Box Trichotomy Theorem" proved in ZFC.
Also we may replace abelian groups by $R$-modules.  Part of our
motivation (in dealing with modules) is that in some sense the improvement
 over earlier results becomes clearer in this context.
\end{abstract}

\maketitle
\numberwithin{equation}{section}
\setcounter{section}{-1}

\newpage


\bigskip

\centerline {Annotated Content} 
\medskip
\par \noindent 
\S0 \quad Introduction, pg. \pageref{Introduction}
\begin{enumerate}
\item[${{}}$]  [We formulate the trivial dual conjecture for $\mu$, 
TDU$_\mu$, and relate it to pcf statements and black box principles.  
Similarly we state the trivial endomorphism conjecture for $\mu$, 
TED$_\mu$, but postpone its treatment.]
\end{enumerate}
\medskip

\par \noindent 
\S1 \quad Preliminaries, pg. \pageref{Preliminaries}
\begin{enumerate}
\item[${{}}$]  [We quote some definitions and results we shall use
and state a major conclusion of this work: the Black Box Trichotomy Theorem.]
\end{enumerate}
\medskip

\par \noindent 
\S2 \quad Cases of weak G.C.H., pg. \pageref{Cases}
\begin{enumerate}
\item[${{}}$]  [Assume $\mu \in {\bold C}_\kappa,\mu < \lambda <
2^\mu < 2^\lambda$ moreover $\lambda = \min\{\chi:2^\chi > 2^\mu\}$.
\Then \, for any 
$\theta < \mu$, a black box called BB$(\lambda,\mu^+,\theta,\kappa)$
holds, which for our purpose is very satisfactory.]
\end{enumerate}
\medskip

\par \noindent 
\S3 \quad Getting large $\mu^+$-free ${\cF} \subseteq 
{}^\kappa \mu$, pg. \pageref{Getting}
\begin{enumerate}
\item[${{}}$]  [The point is to give sufficient conditions
for BB: see \ref{2b.98}(2).
Let $\mu \in {\bold C}_\kappa$ and $\lambda = 2^\mu$. We
give sufficient conditions for the existence of $\mu^+$-free 
${\mat F}\subseteq {}^\kappa \mu$ of cardinality $\lambda$, which is quite
helpful for our purposes, as it implies the existence of suitable black
boxes.  One such condition is (see
\ref{1f.21}): the existence of $\theta < \kappa$ and $\chi <
\lambda$ such that $\chi^\theta = \lambda$.  
Recall that by \S2 assuming $\lambda < \lambda^{< \lambda}$ suffices (for the
black box).  Now assuming there is no $\theta$ as above so
$\lambda = \lambda^{< \lambda}$, by older results if $\theta =
\cf(\theta) < \kappa \wedge \chi < \lambda \Rightarrow 
\chi^{<\kappa>_{\tr}} < \lambda$ then
$(D\ell)^*_{S^\lambda_\theta}$, hence
$(D \ell)_S$ for every stationary $S\subseteq S^\lambda_\kappa$. 

In \ref{1f.7} we consider $\theta \in (\kappa,\mu) \cap
\text{ Reg}$ and $\chi \in (\mu,\lambda)$ such that
$\chi^{<\theta>_{\tr}}=\lambda$.  Here the results are less
sharp.  Also if 
$\lambda = \chi^+$, where $\chi$ is regular, then this holds; see
\ref{2b.91}.  We finish by indicating some obvious connections.
\end{enumerate}
\medskip

\noindent 
\S4 \quad On the $\mu$-free trivial dual conjecture for 
$R$-modules, pg. \pageref{On}
\begin{enumerate}
\item[${{}}$] [We deduce what we can on the conjecture TDU$_\mu$.]
\end{enumerate}
\newpage

\section{Introduction} \label{Introduction}

\subsection {Background} \label{Background}\
\bigskip

We prove some black boxes, most notably the Black Box Trichotomy
Theorem.  Our original question is whether provably in ZFC the conjecture
TDU$_{\aleph_\omega}$ holds and even whether TED$_{\aleph_\omega}$
holds where:

\begin{definition}
\label{0p.7}
1) Let TDU$_\mu$, the trivial dual
conjecture for $\mu > \aleph_0$, mean: 
\newline
there is a $\mu$-free abelian group $G$, necessarily of cardinality
$\ge \mu$, such that $G$ has a trivial
dual (i.e., Hom$(G,\Bbb Z) = \{0\}$).
\newline
2) Let TED$_\mu$, the trivial endomorphism conjecture for $\mu$ mean:
there is a $\mu$-free abelian group with no non-trivial endomorphism,
i.e., End$(G)$ is trivial (that is, End$(G)$ $\cong \Bbb Z$).
\end{definition}
\smallskip

Much is known for $\mu = \aleph_1$ (see, e.g., \cite{GbTl06}).  
Note that each of the cases of \ref{0p.7} implies that $G$ is
$\aleph_1$-free, not free, and much is known on the existence of 
$\mu$-free, non-free abelian groups of cardinality $\mu$ 
(see , e.g., \cite{EM02}).
Also, positive answers are known for arbitrary $\mu$ under, 
e.g., ${\bold V} = {\bold L}$, 
see pg. 461 of \cite{GbTl06}.

Note that by singular compactness, for singular $\mu$ there are no
counterexamples of cardinality $\mu$.

By \cite{Sh:883}, if $\mu = \aleph_n$, then the answer to TDU$_\mu$
is yes, for the cardinality $\lambda = \beth_n$.  It was hoped that 
the method would apply to many other
related problems and to some extent this has been vindicated by
G\"obel-Shelah \cite{GbSh:920}; G\"obel-Shelah-Str\"ungman
\cite{GbShSm:981} and (on TED$_\mu,\mu = \aleph_n$) by 
G\"obel-Herden-Shelah \cite{GbHeSh:970}.
But we do not know the answer for $\mu = \aleph_\omega$.  Note that
even if we succeed this will not cover the results of \cite{Sh:883},
\cite{GbSh:920}, \cite{GbShSm:981}, 
\cite{GbHeSh:970}; e.g. because there the cardinality
of $G$ is $< \beth_\omega$ when $\mu < \aleph_\omega$ 
and probably even more so when we deal with larger cardinals.

A natural approach is to prove in ZFC appropriate set-theoretic
principles, and this is the method we try here.  This raises
combinatorial questions which seem interesting in their own right; our
main result in this direction is the Black Box Trichotomy Theorem
\ref{h.7}.  But the original algebraic question has 
bothered me and the results are irritating: it is ``very hard" not 
to answer yes in the following sense (later we say more on the set
theory involved):
\mn
\begin{enumerate}
\item[$(a)$]  Failure implies strong demands on cardinal arithmetic
in many $\beth_\delta$, (e.g. if cf$(\delta) = \aleph_1$ then
$\beth_{\delta +1} = \text{ cf}(\beth_{\delta +1}) = (\beth_{\delta
+1})^{< \beth_{\delta +1}}$ and $\chi < \beth_{\delta +1} \Rightarrow
\chi^{\aleph_0} < \beth_{\delta +1}$ - see details below),
\sn
\item[$(b)$]  If we weaken ``$\aleph_\omega$-freeness" to 
(so called ``stability" or ``softness" and even) ``$\aleph_1$-free
and constructible from a ladder system $\langle C_\delta:\delta \in S
\subseteq S^\lambda_{\aleph_0}\rangle$", then we can prove existence,
\sn
\item[$(c)$]  Replacing abelian groups by $R$-modules, the parallel
question depends on a set of regular cardinals related to the ring, 
sp$(R)$, see 
Definition \ref{0f.13} (so the case of abelian groups is $R = \bbZ$).  
If sp$(R)$ is empty, there is nothing to be done.
By \cite{Sh:775}, if sp$(R)$ is unbounded below some strong limit 
singular cardinal $\mu$ of cofinality $\aleph_0$ then TDU$_{\mu^+}$, 
see \ref{5e.47}.  Moreover, by \cite{Sh:829}, if sp$(R)$ is infinite, say 
$\kappa_n < \kappa_{n+1} \in \text{ sp}(R)$ then by \ref{5e.47} again
$\TDU_\mu$ for every $\mu$ 
(by the quotation \ref{1.3.25}).  Furthermore: (see \ref{1f.53}), we
prove that: 
if $\aleph_0,\aleph_1,\aleph_2 \in \text{ sp}(R)$ then the answer for
$R$-modules is positive.
\sn
\item[$(d)$]  Even if the negation of TDU$_{\aleph_\omega}$ is consistent
with ZFC its consistency strength is large, to some extent this
follows by clause (a) above but by \S2 we have more.
\end{enumerate}

Obviously, e.g. clause (c) clearly seems informative 
for abelian groups;  at first sight it seems helpful
that for every $n$ there is an
$\aleph_n$-free non-free abelian group of cardinality $\aleph_n$, but
this is not enough.  More specifically this method does not at present
resolve the problem  because for $R = \bbZ$ we only know 
that sp$(R)$ includes $\{\aleph_0,\aleph_1\}$, (and under MA it has no
other member $< 2^{\aleph_0}$).

Still we get some information: a reasonably striking set-theoretic
result is the Black Box Trichotomy Theorem \ref{h.7} below; some abelian group 
theory consequences are given in \S4.

A sufficient condition (see \ref{5e.32}) for a positive answer to
TDU$_\mu$ is :
\mn
\begin{enumerate}
\item[$\circledast_0$]   TDU$_\mu$ if $\BB(\lambda,\mu,\theta,J)$ 
when $J$ is $J^{\bd}_{\aleph_0}$  or $J^{\bd}_{\aleph_1 \times
\aleph_0}$, see \ref{0p.6}, $\cf(\lambda) > \aleph_0$ and $\theta = \beth_4$.
\end{enumerate}
\mn
This work will be continued in \cite{Sh:F1200} and also in
\cite{Sh:1017} which originally was part of the present paper.

Before we state the results we give some basic definitions.
\bigskip

\subsection {Basic Definitions} \label{Basic}\
\bigskip

\noindent
Recall that
\begin{definition}
\label{0p.4}
$\chi^{<\partial>_{\text{tr}}}$ is the
$\partial$-tree power of $\chi$, i.e., the supremum of the number of
$\partial$-branches of a tree with $\le \chi$ nodes and $\partial$
levels.
\end{definition}

\begin{notation}
\label{0p.6}
1) For a set $S$ of ordinals with no greatest member (e.g. a limit ordinal
   $\delta$) let $J^{\bd}_S$ be the ideal $\{u:u$ is a
bounded subset of $S\}$. 

\noindent
2) For limit ordinals $\delta_1,\delta_2$ let $J^{\bd}_{\delta_1
   \times \delta_2} = \{u \subseteq \delta_1 \times \delta_2:\{\alpha
   < \delta_1:\{\beta < \delta_2:(\alpha,\beta) \in u\} \notin
   J^{\bd}_{\delta_2}\} \in J^{\bd}_{\delta_1}\}$.

\noindent
3) For limit ordinals $\delta_1,\delta_2$ let $\delta_3 = \delta_2
   \cdot \delta_1$ and $J^{\bd}_{\delta_1 * \delta_2}$ be the
   following ideal on $\delta_3:\{u \subseteq
\delta_3:\{(\alpha,\beta) \in \delta_1 \times \delta_2:\delta_2 \cdot
\alpha + \beta \in u\} \in J^{\bd}_{\delta_1 \times \delta_2}\}$.
\end{notation}

\begin{definition}
\label{0p.9}
1) A sequence of non-empty 
sets $\bar C = \langle C_\alpha:\alpha \in S\rangle$ is
$\mu$-free if for every $u \in [S]^{< \mu}$ there exists $\bar A =
\langle A_\alpha \subseteq C_\alpha:\alpha \in u\rangle$ 
so that the sets $\langle
C_\alpha \backslash A_\alpha:\alpha \in u\rangle$ are pairwise
disjoint and each $A_\alpha$ is bounded in $C_\alpha$ with respect to
a given order on $C_\alpha$; in the default case ``every $C_\alpha$ is a
set of ordinals with the natural order".

\noindent
2) We may replace $\mu$ by a pair $(\mu,\bar J)$, where $\bar J =
   \langle J_\alpha:\alpha \in S\rangle$ and $J_\alpha$ is an ideal on
$\otp(C_\alpha)$ so now ``$A_\alpha$ bounded" is replaced by
``$\{\otp(\varepsilon \cap C_\alpha):\varepsilon \in A_\alpha\} \in
   J_\alpha$".  If $C_\alpha$ is a set of ordinals of a fixed 
order type $\gamma(*)$ and $J_\alpha = J$ for every
   $\alpha \in S$ where $J$ is an ideal on $\gamma(*)$ \then \, we may
replace the pair $(\mu,\bar J)$ by the pair $(\mu,J)$.  In other words,
instead of the demand ``$A_\alpha$ is 
bounded in $C_\alpha$" we require $A'_\alpha :=
 \{\otp(C_\alpha \cap \gamma):\gamma \in A_\alpha\} \in J$.
\end{definition}

\noindent
The definition of the assertion {\rm BB}$(\lambda,\mu,\theta,J)$ is as
follows. (BB stands for black box).
The following is a relative of \cite{Sh:775} (and see on the history there).
\begin{definition}
\label{0p.14}
Assume we are given a quadruple
$(\lambda,\mu,\theta,\kappa)$ of cardinals [but we may replace $\lambda$ by an
ideal $I$ on $S \subseteq \lambda = \sup(S)$ so writing $\lambda$
means $S = \lambda$; also we may replace
$\kappa$ by an ideal $J$ on $\kappa$ and writing $\kappa$ means 
$J = J^{\bd}_\kappa$]. Let
BB$^-(\lambda,\mu,\theta,\kappa)$ mean that some pair
$(\bar C,\bar{\bold c})$ satisfies the clauses (A) and (B) below;
we call the pair $(\bar C,\bar{\bold c})$ a witness for
BB$^-(\lambda,\mu,\theta,\kappa)$.  
Let BB$(\lambda,\mu,\theta,\kappa)$ mean that some witness $(\bar
C,\bar{\bold c})$ satisfies clause (A) below and for some sequence
$\langle S_i:i < \lambda\rangle$ of pairwise disjoint subsets of
$\lambda$ (\underline{or} of $S$), each $(\bar C \restriction S_i,
{\bar{\bold c} \restriction S_i})$ satisfies clause (B) below, (thus replacing
$S,\bar{\bold c}$ by $S_i,\bar{\bold c} \rest S_i$) where:
\mn
\begin{enumerate}
\item[$(A)$]  $(a) \quad \bar C = \langle C_\alpha:\alpha \in
S\rangle$ and $S = S(\bar C) \subseteq \lambda = \sup(S)$
\sn
\item[${{}}$]  $(b) \quad C_\alpha \subseteq \alpha$ has order
type $\kappa$
\sn
\item[${{}}$]  $(c) \quad \bar C$ is $\mu$-free (see \ref{0p.9}): 

\hskip25pt [\underline{but} when we replace $\kappa$ by $J$ then we
say ``$\bar C$ is $(\mu,J)$-free"]
\sn
\item[$(B)$]  $(d) \quad \bar{\bold c} = \langle {\bold c}_\alpha:\alpha \in
S \rangle$
\sn
 \item[${{}}$]  $(e) \quad {\bold c}_\alpha$ is a function from
$C_\alpha$ to $\theta$
\sn
 \item[${{}}$]  $(f) \quad$ if ${\bold c}:\bigcup\limits_{\alpha \in S}
C_\alpha \rightarrow \theta$, then ${\bold c}_\alpha
= {\bold c} \restriction C_\alpha$ for some $\alpha  \in S$ 

\hskip25pt [\underline{but} when we replace
$\lambda$ by $I$ an ideal on $S$, then we demand that 

\hskip25pt  the set $\{\alpha \in S:\bold c_\alpha = \bold c \rest C_\alpha\}$
is not in $I$].
\end{enumerate}
\end{definition}

\begin{remark}  The reader may recall that if $S$ is a stationary subset
of $\{\delta < \lambda:\text{cf}(\delta) = \kappa\}$ for a regular
cardinal $\lambda$ and $S$ is non-reflecting and $\bar C = \langle
C_\alpha:\alpha \in S\rangle$ satisfies $C_\delta \subseteq \delta =
\sup(C_\delta)$, otp$(C_\delta) = \kappa$, \underline{then}
$\diamondsuit_S$ implies BB$(\lambda,\lambda,\lambda,\kappa)$.  So if
$\bold V = \bold L$ then for every regular $\kappa < \lambda,\lambda$ a
non-weakly compact cardinal we have BB$(\lambda,\lambda,\lambda,\kappa)$.

So the consistency of (more than) having many cases of BB is known, 
\underline{but} we prefer to get results in ZFC, when possible.
\end{remark}  

Variants are:

\begin{definition}
\label{0p.15}
In Definition \ref{0p.14}:

\noindent
1) We may replace $\theta$ by $(\chi,\theta)$ which means there are
$S,\bar C$ satisfying clause (A) of Definition \ref{0p.14} and
\mn
\begin{enumerate}
\item[$(B)'$]   if $\bar{\bold F} = \langle {\bold F}_\alpha:\alpha \in
S\rangle$ and ${\bold F}_\alpha$ is a function from
${}^{(C_\alpha)}\chi$ to $\theta$, \then \, for some $\bar{\bold c}$ we
have:
\sn
\begin{enumerate}
\item[$(d)$]  $\bar{\bold c} = \langle {\bold c}_\alpha:\alpha \in
S\rangle$
\sn
\item[$(e)$]  ${\bold c}_\alpha < \theta$
\sn
\item[$(f)$]   if ${\bold c}:\lambda \rightarrow \chi$, \underline{then}
${\bold c}_\alpha = {\bold F}_\alpha({\bold c} \restriction C_\alpha)$ for
some $\alpha \in S$ [\underline{or} if we replace $\lambda$ by $I$ then the set
$\{\alpha \in S:{\bold c}_\alpha = {\bold F}_\alpha({\bold c} \restriction
C_\alpha)\}$ does not belong to the ideal $I$].
\end{enumerate}
\end{enumerate}
\mn
2) Replacing $(\chi,\theta)$ by $(\chi,1/\theta)$ abusing notation or
$\langle \chi,\theta\rangle$,
means that in clause (f) we replace ``${\bold c}_\alpha = {\bold
F}_\alpha({\bold c} \restriction C_\alpha)$" by ``${\bold c}_\alpha \ne
{\bold F}_\alpha({\bold c} \restriction C_\alpha)"$.

\noindent
3) We may replace $\mu$ by $\bar C$ and thus waive the freeness demand,
   i.e. $\bar C$ is not necessarily $\mu$-free.  Alternatively, we may
   replace $\mu$ by a set $\cF$ of one-to-one functions from $\kappa$
   to $\lambda$ when $\bar C$ lists $\{\text{Rang}(f):f \in \cF\}$.

\noindent
4) Replacing $\kappa$ by ``$< \kappa_1$" means that in (A)(b) we require just
$C_\alpha \subseteq \alpha \wedge |C_\alpha| < \kappa_1$ 
(and not necessarily otp$(C_\alpha) =\kappa$).  Replacing $\kappa$ by
$*$ means ``$< \lambda$".

\noindent
5) We may replace $\theta$ by ``$< \theta_1$" meaning ``for every
   $\theta < \theta_1$".
\end{definition}

\begin{remark}  
1) Note that BB$(\lambda,\mu,\theta,\kappa)$ is somewhat
related to NPT$(\lambda,\kappa)$ from \cite[Ch.II]{Sh:g},
i.e. BB$(\lambda,\lambda,\theta,\kappa) \Rightarrow 
\text{ NPT}(\lambda,\kappa)$, but NPT has no ``predictive" part.

\noindent
2) We shall use freely the obvious implications concerning the black
   boxes, e.g.
\mn
\begin{enumerate}
\item[$(*)$]  $BB^-(\lambda_1,\mu_1,\theta_1,\kappa_1)$ implies
  $BB^-(\lambda_2,\mu_2,\theta_2,\kappa_2)$ when $\lambda_2 =
  \lambda_1,\mu_2 \le \mu_1, \theta_2 \le \theta_1,\kappa_2 = \kappa_1$.
\end{enumerate}
\end{remark}

\noindent
Of course
\begin{obs}
\label{2b.98}
1) If $\bar C = \langle C_\alpha:\alpha \in [\mu,\lambda)\rangle,
C_\alpha \subseteq \mu$ non-empty and $2^\mu = \lambda$
(e.g. $\lambda = \mu^\kappa \wedge \mu \in {\bold C}_\kappa)$, \underline{then}
BB$(\lambda,\bar C,\lambda,*)$, see \ref{0p.15}(4).

\noindent
2) If in addition {\rm otp}$(C_\alpha) = \kappa$ and $\bar C$ is
$\mu_1$-free, then BB$(\lambda,\mu_1,\lambda,\kappa)$.
\end{obs}
\bigskip

\begin{PROOF}{\ref{2b.98}}
The proof is easy, but we shall give details.

\noindent
1) Let $S = [\mu,\lambda)$ and let $\langle S_\varepsilon:\varepsilon <
\lambda\rangle$ be a partition of $S$ into sets each of cardinality
$\lambda$, each stationary if $\lambda$ is regular.  
Recalling Definitions \ref{0p.14}, \ref{0p.15} it 
suffices to prove BB$(\lambda,\bar C \rest S_\varepsilon,\lambda,*)$ 
for each $\varepsilon < \lambda$; fix $\varepsilon$ now.  Clause
(A) in Definition \ref{0p.14} is obvious, so we shall 
prove clause (B)$'$, so let $\langle \bold F_\alpha:\alpha \in
S_\varepsilon\rangle$ and $\bold F_\alpha:{}^{(C_\alpha)}\lambda \rightarrow
\lambda$ be given and we should choose $\bar c \in {}^{(S_\varepsilon)}\theta$.

Let $\bar f = \langle f_\alpha:\alpha \in S_\varepsilon\rangle$ list
${}^\mu \lambda$, each appearing unboundedly often (and even
stationarily often if $\lambda$ is regular), and choose
$c_\alpha := \bold F_\alpha(f_\alpha \restriction C_\alpha)$.  Now
check. 

\noindent
2) Look at the definitions.
\end{PROOF}  

\begin{discussion}
\label{2b.99}
We use \ref{2b.98}, e.g. in \ref{h.9d}.
\end{discussion}

\noindent
Recall:
\begin{definition}
\label{0p.17}
1) If $\le_*$ is a partial order on a set $I$ let $\lambda =
\tcf(I,<_*)$ mean that $\lambda$ is a regular cardinal and there is an
$<_*$-increasing sequence $\langle t_\alpha:\alpha <
\lambda\rangle$ which is cofinal, that is 
$(\forall s \in I)(\exists i < \lambda)[s \le_* t]$.

\noindent
2) For $I,<_*$ as above let 
$\cf(I,<_*) = \min\{|\cP|:\cP \subseteq I$ is cofinal$\}$.
\end{definition}

\begin{definition}
\label{0p.19}
Assume $\mu > \theta \ge \sigma = \cf(\sigma) \ge \cf(\mu)$.

For $J$ an ideal on $\theta$ (or just on a set $A_*$ of cardinality
$\theta$) such that there is a $\subseteq$-increasing sequence of
members of $J$ of length $\cf(\mu)$ with union $\theta$ (or $A_*$). 

\noindent
1) We define $\pp_J(\mu) = \sup\{\tcf(\prod\limits_{i < \theta}
   \lambda_i,<_J):\lambda_i = \cf(\lambda_i) \in (\theta,\mu)$ for $i
   < \theta$ and $\mu = \lim_J\langle \lambda_i:i < \theta\rangle\}$
   \underline{where} $\mu = \lim_J\langle \lambda_i:i < \theta\rangle$
   means that $\mu_i < \mu \Rightarrow \{i < \theta:\lambda_i \notin
   [\mu_i,\mu]\} \in J$.

\noindent
2) We define $\pp_{\theta,\sigma}(\mu) = \sup\{\tcf(\prod\limits_{i <
   \theta} \lambda_i,<_J):J$ a $\sigma$-complete ideal on $\theta$ with
   $\lambda_i = \cf(\lambda_i) \in (\theta,\mu)$ such that $\mu =
   \lim_J\langle \lambda_i:i < \theta\rangle$.

\noindent
3) Let $\pp_J(\mu) =^+ \chi$ mean that $\pp_J(\mu) =
   \chi$ and $\chi$ is regular and in the supremum in part (1) is
attained; similarly in parts (2),(3).

\noindent
4) Let $\pp^+_J(\mu)$ be $(\pp_J(\mu))^+$ if $\pp_J(\mu)$ is regular
   and the supremum in part (1) is obtained and be $\pp_J(\mu)$ otherwise.
\end{definition}

\begin{definition}
\label{0p.21}
For cardinals $\lambda \ge \mu \ge \theta \ge \sigma$ let
$\cov(\lambda,\mu,\theta,\sigma) = \min\{|\cP|:\cP \subseteq
[\lambda]^{< \mu}$ and every $u \in [\lambda]^{< \theta}$ is included
in the  union of $< \sigma$ members of $\cP\}$.
\end{definition}
\bigskip

\subsection {What is Done} \
\bigskip

In this work we shall show that it is ``hard" for ${\bold V}$ not to
give a positive answer (i.e. existence) for \ref{0p.7}
 via a case of \ref{0p.14} or variants; we
review below the ``evidence" for this assertion.  By \ref{5e.32}(1) 
we know that (actually $2^{(2^{\aleph_1})}$ can be weakened):
\mn
\begin{enumerate}
\item[$\odot_0$]   a sufficient condition for
TDU$_\mu$ is, e.g., BB$(\lambda,\mu,2^{(2^{\aleph_1})^+},J)$,
where $\cf(\lambda) > \aleph_0$ and $J$ is $J^{\bd}_{\aleph_0}$ 
or $J^{\bd}_{\aleph_1 \times \aleph_0}$ 
(hence also $J = J^{\bd}_{\aleph_1}$ suffices; noting that here
$\kappa$ is $\aleph_0$ or $\aleph_1$ together
$BB(\lambda,\mu,2^{(2^{\aleph_1})^+},\kappa)$ suffice).
\end{enumerate}
\mn
Recall that ${\bold C}_\kappa$ is the class of strong limit singular
cardinals of cofinality $\kappa$ when
$\kappa > \aleph_0$, and ``most" of them
when $\kappa = \aleph_0$ (see Definition \ref{1.3.1} and Claim \ref{1.3.3}).

Now the first piece of the evidence given here that a
failure of G.C.H. near $\mu \in {\bold C}_\kappa$ helps is the following fact:
\mn
\begin{enumerate}
\item[$\circledast_1$]  BB$(\lambda,\mu^+,\theta,\kappa)$ if
$\theta < \mu \in {\bold C}_\kappa$ and $\mu < \lambda < 2^\mu <
2^\lambda$.
\end{enumerate}
\mn
[Why?  By Conclusion \ref{d.11}(1); it is a consequence of the Black
Box Trichotomy Theorem \ref{h.7}.]

Note: another formulation is
\mn
\begin{enumerate}
\item[$\boxdot_1$]  if $\theta < \mu \in {\bold C}_\kappa$ but
BB$(\lambda,\mu^+,\theta,\kappa)$ fails \then \,
$(2^\mu)^{< 2^\mu} = 2^\mu$.
\end{enumerate}
\mn
[Why?  Let $\lambda_1 = \text{ min}\{\chi:2^\chi > 2^\mu\}$, so
necessarily $\mu < \lambda_1$; if $\lambda_1 < 2^\mu$ then
BB$(\lambda_1,\mu^+,\theta,\kappa)$ holds by $\circledast_1$, so by
our assumption $\lambda_1 = 2^\mu$, so $\mu \le \chi < 2^\mu
\Rightarrow 2^\chi = 2^\mu \Rightarrow (2^\mu)^\chi = 2^{\mu \cdot \chi}
= 2^\chi = 2^\mu$, but this means $(2^\mu)^{< 2^\mu} = 2^\mu$, as
stated.] 

So by $\odot_0 + \boxdot_1$
\begin{enumerate}
\item[$\odot_1$]  if TDU$_{\aleph_\omega}$ fails, \underline{then}
\begin{enumerate}
\item[$(a)$]   a large class of cardinals satisfies
a weak form of G.C.H.
 \item[$(b)$]  more specifically, $(\mu \in {\bold C}_{\aleph_0}
\cup {\bold C}_{\aleph_1}) \wedge \lambda = 2^\mu \Rightarrow \lambda
= \lambda^{< \lambda}$.
\end{enumerate}
\end{enumerate}
\mn
For $\cT \subseteq {}^{\sigma >}\chi$ a tree with $\le \chi$
nodes and $\le \sigma$ levels we let $\lim_\sigma(\cT) = \{\eta \in
{}^\sigma \chi:(\forall \varepsilon < \sigma)(\eta \rest \varepsilon
\in \cT)\}$, and recall that the tree power $\chi^{<\sigma>_{\tr}}$ is
$\sup\{|\lim_\sigma(\cT)|:\cT \subseteq {}^{\sigma >}\chi$ is a tree
with $\le \chi$ nodes and $\le \sigma$ levels$\}$.

We have:
\mn
\begin{enumerate}
 \item[$\circledast_2$]   BB$(2^\mu,\kappa^{+ \omega+1},
\theta,J^{\bd}_{\kappa^+ \times \kappa})$ if
$\theta < \mu \in {\bold C}_\kappa$ and $(\forall \chi)(\chi < 2^\mu
\Rightarrow \chi^{<\kappa^+>_{\text{tr}}} < 2^\mu)$.
\end{enumerate}
\mn
[Why?  See \ref{h16}.]

So we have
\begin{enumerate}
\item[$\odot_2$]  if TDU$_{\aleph_\omega}$ fails, \then \, for every
$\mu \in {\bold C}_{\aleph_0}$ there is $\chi$ such that $\mu < \chi <
\chi^{<\aleph_1>_{\tr}} = 2^\mu$ (see Definition \ref{0p.4}), 
hence $\mu < \chi < 2^\mu$
and \wolog \, cf$(\chi) = \aleph_1$, hence $\mu^{+\omega_1} \le \chi <
2^\mu$, and so G.C.H. fails quite strongly (putting us in some sense in the
opposite direction to $\odot_1$)
\end{enumerate}
\mn
and also
\mn
\begin{enumerate}
 \item[$\circledast_3$]  if $\mu \in {\bold C}_\kappa,\theta < \mu,
\lambda = 2^\mu$ and some set ${\mat F} \subseteq {}^\kappa \mu$ is
$\mu_1$-free of cardinality $2^\mu \, (= \mu^\kappa)$, \underline{then}
BB$(\lambda,\mu_1,\theta,\kappa)$.
\end{enumerate}
\mn
[Why?  See \ref{2b.98}(2).]

In \S3 we shall give various sufficient conditions for the
satisfaction of the hypotheses of $\circledast_3$.  
Another piece of evidence is
\mn
\begin{enumerate}
\item[$\circledast_4$]  BB$(\lambda,\mu_1,\theta,J)$ \underline{when}:
\sn
\begin{enumerate}
\item[$(a)$]  $\theta < \mu \in {\bold C}_\kappa$ and $\lambda = 2^\mu
= \lambda^{< \lambda}$ and $\partial < \mu$,
\sn
 \item[$(b)$]  $J$ is an ideal on $\partial = \text{ cf}(\partial)$
extending $J^{\text{bd}}_\partial$, and $S \subseteq S^\lambda_\partial$
(see \ref{0p.31}(3)),
$\bar C = \langle C_\delta:\delta \in S\rangle$ are such that $\delta \in S
\Rightarrow C_\delta \subseteq \delta = \sup(C_\delta) \wedge
\kappa = \otp(C_\delta)$,
\sn
 \item[$(c)$]  $\bar C$ is $\mu_1$-free, $\mu_1 < \lambda$, see Definition
\ref{1.3.14}(1A),(2), it is closed to \ref{0p.9},
\sn
\item[$(d)$]  $\quad \bullet \quad (\forall \alpha < \lambda)(\lambda >
|\{C_\delta \cap \alpha:\delta \in S \wedge \alpha \in C_\delta\}|)
\wedge$

\hskip40pt $(\forall \chi < \lambda)(\chi^{<\partial>_{\text{\rm tr}}} <
\lambda)$ \underline{or}
\sn
\item[${{}}$]  $\quad \bullet \quad$ (D$\ell$)$_S$ 
(see Definition \ref{a26}).
\end{enumerate}
\end{enumerate}
\mn
[Why?  This follows from \cite{Sh:775}.]

A consequence for the present work is:
\begin{enumerate}
\item[$\circledast_5$]  BB$(\lambda,\kappa^{+ \omega},\theta,
J^{\text{bd}}_{\kappa^+ \times \kappa})$ \underline{when}:
\mn
\begin{enumerate}
\item[$(a)$]  $\theta < \mu \in {\bold C}_\kappa,\lambda = 
2^\mu = \lambda^{< \lambda}$,
\sn
 \item[$(b)$]  $S \subseteq S^\lambda_{\kappa^+},\delta \in S
\Rightarrow C_\delta \subseteq \delta = \sup(C_\delta) \wedge 
\otp(C_\delta) = \kappa^+$,
\sn
\item[$(c)$]  $\langle C_\delta:\delta \in S\rangle$ is
$\kappa^{+ \omega}$-free and $\kappa^{+ \omega} < \lambda$ which
actually follows,
\sn
\item[$(d)$]  (D$\ell$)$_S$ or the first possibility of
  $\circledast_4(d)$ for $\partial = \kappa$.
\end{enumerate}
\end{enumerate}
\mn
[Why?  By $\circledast_4$.]

The point of $\circledast_5$ is that we can find $\bar C$ as in clause (b) of
$\circledast_5$ with $S \subseteq S^\lambda_{\kappa^+}$ ``quite large"
so we ignore the difference (in the introduction) - see \ref{2b.111}. 
In particular
\mn
\begin{enumerate}
\item[$\boxdot_2$]  if $\lambda = \mu^+ = 2^\mu$ and 
$\mu$ $> \aleph_0$ is a strong
limit cardinal of cofinality $\kappa = \aleph_0$, \underline{then} for
some $\bar C,S$ clauses (a)-(d) of $\circledast_5$ hold.
\end{enumerate}
\mn
[Why?  As in $\circledast_2$.]

Moreover
\mn
\begin{enumerate}
 \item[$\boxdot_3$]  if $\kappa < \chi,\kappa$ is a regular cardinal, 
$\lambda = \chi^+ = 2^\chi$ and $\kappa \ne \text{ cf}(\chi)$,
\underline{then} $\diamondsuit_S$ for every stationary
$S \subseteq S^\lambda_\kappa = \{\delta <
\lambda:\text{cf}(\delta) = \kappa\}$.
\end{enumerate}
\mn
[Why?  By \cite{Sh:922} - see \ref{a35}.]

We can conclude
\begin{enumerate}
\item[$\odot_3$]  if TDU$_{\aleph_\omega}$ fails and $\mu \in
{\bold C}_{\aleph_0}$, \underline{then} $2^\mu$ is not $\mu^+$,
moreover, $2^\mu$ is not of the form $\chi^+$, cf$(\chi) \ne
\aleph_1$.
\end{enumerate}
\mn
[Why?  Note that $(D \ell)^*_{S^\lambda_{\aleph_1}}$ holds by $\boxdot_3$.]
\mn
\begin{enumerate}
\item[$\circledast_6$]  BB$(2^\mu,\mu^+,\theta,\kappa)$ if $\theta
< \mu \in {\bold C}_\kappa$ and $\chi^\sigma = 2^\mu$ for some $\sigma
= \cf(\sigma) < \kappa,\chi < 2^\mu$.
\end{enumerate}
\mn
[Why?  The assumptions (a) - (f) of claim \ref{1f.21} hold for $J =
J^{\text{\rm bd}}_\kappa$ and $\sigma$ here standing for $\theta$
there.  E.g. clause (d) there, ``$\alpha < \mu \Rightarrow
|\alpha|^\theta < \mu$" holds as $\mu$ is a strong limit.  So the first
assumption of conclusion \ref{1f.22} holds, and the second $(\mu^\kappa
= 2^\mu)$ holds as $\mu \in \bold C_\kappa$.  
So the conclusion of \ref{1f.22} holds which implies by \ref{2b.98} 
that $\circledast_6$ holds.]
\mn
\begin{enumerate}
\item[$\circledast_7$]  BB$(2^\mu,\partial,\theta,\kappa)$ if
$\theta < \mu \in {\bold C}_\kappa$ and $\partial = 
\sup\{\text{cf}(\chi):\text{ cf}(\chi) < \mu < \chi < 2^\mu$ and
pp$_{\text{cf}(\chi)\text{-comp}}(\chi) =^+ 2^\mu\}$.
\end{enumerate}
\mn
[Why?  By \ref{1f.7} and \ref{2b.98}.]

So (by $\odot_0,\circledast_6,\circledast_7$)
\mn
\begin{enumerate}
\item[$\odot_4$]  if TDU$_{\aleph_\omega}$
fails, \underline{then} for every $\mu \in {\bold C}_{\aleph_1}$ we have
\sn
\begin{enumerate}
\item[$(a)$]  $\alpha < 2^\mu \Rightarrow |\alpha|^{\aleph_0} < 2^\mu$
\sn
\item[$(b)$]  for some $n,\aleph_n \le \cf(\chi) < \mu \wedge \chi <
2^\mu \Rightarrow \pp_{\cf(\chi)\comp}(\chi) \ne^+ 2^\mu$. 
\end{enumerate}
\end{enumerate}
\mn
By the end of \S4
\mn
\begin{enumerate}
\item[$\odot_5$]  if TDU$_{\aleph_\omega}$ fails 
and $n \ge 3$, \then
\sn
\begin{enumerate}
\item[$(A)$]  no $\aleph_n$-free (abelian) group $G$ of
cardinality $\aleph_n$ is Whitehead
\sn
 \item[$(B)$]  if $\mu \in {\bold C}_{\aleph_0} \cup {\bold
C}_{\aleph_1}$ and $\lambda = 2^\mu$ then $(D \ell)_{S^\lambda_{\aleph_n}}$.
\end{enumerate}
\end{enumerate}
\bigskip

Generally in \cite{Sh:g} we 
suggest cardinal arithmetic assumptions as good ``semi-axioms".

We have used cases of WGCH (the Weak Generalized Continuum Hypothesis, 
i.e., $2^\lambda < 2^{\lambda^+}$ for every $\lambda$); in 
\cite{Sh:87a}, \cite{Sh:87b},
\cite{Sh:88}, also in \cite{Sh:192} and see \cite{Sh:h},
\cite{Sh:838}.
Influenced also by this, Baldwin suggested adopting WGCH as an extra
axiom (to ZFC) giving
arguments parallel to the ones for large cardinals (but with no problem of
consistency).  So it seems reasonable to see what we can say in our
context.

Note that above we get:
\begin{claim}
\label{0p.23}
Assume $\mu \in \bold C_\kappa$ or just $\mu$ is a 
cardinal of cofinality $\kappa$ (e.g. $\mu \ge \kappa = \cf(\mu))$.

\noindent
1) If $\mu^+ < 2^\mu < 2^{\mu^+}$ and $\kappa 
\in \{\aleph_0,\aleph_1\}$,
\underline{then} there is a $\mu^+$-free
abelian group of cardinality $\mu^+$ with $\Hom(G,\bbZ) = 0$;
note that this is iterable, i.e., if $\mu_{\ell +1} \in {\bold
C}_{\mu^+_\ell}$ for $\ell < n,2^{\mu_\ell} > \mu^+_\ell$ for $\ell <
n$ and $\mu_0$ is like $\mu$ above, then the conclusion applies for $\mu_n$.

\noindent
2) If $\mu^+ = 2^\mu$ and $\kappa \in \{\aleph_0,\aleph_1\}$,
\then \, there is an $\aleph_{\omega +1}$-free
abelian group of cardinality $\mu^+$ such that $\Hom(G,\bbZ) = 0$.
\end{claim}

\begin{PROOF}{\ref{0p.23}}
\label{5e.35}
1) First assume $\mu \in \bold C_\kappa$.

By \ref{h.7} there is a $\mu^+$-free $\cF \subseteq {}^\kappa \mu$ of
cardinality $\mu^+$ 
(yes! not $2^\mu$) hence BB$(\lambda,\mu,\lambda,\kappa)$ by
Conclusion \ref{0p.14}(1).  
By \ref{5e.14}, \ref{5e.28} there is $G$ as required.

Similarly for iterations.

Second, assume $\cf(\mu)=\kappa$.  We can find $\cF$ as above if $\mu$
is singular, use again \ref{5e.35} if $\mu = \kappa$ it is easy.  Then
we get $\BB(\lambda,\mu,2,\kappa)$ by \ref{0p.14}(3).  Check.

\noindent
2) The proof is similar.
\end{PROOF}

\noindent
Note that we can prove TDU$_{\aleph_{\omega +1}}$ if the answer to the
following is positive:
\begin{conjecture}
\label{0p.26}
If $\lambda = \lambda^{< \lambda} > \kappa^+$ and
$\kappa = \cf(\kappa)$ and $\lambda \ne \aleph_1$ 
(or at least $\lambda \ge \beth_\omega$
replacing the assumption $\lambda \ne \aleph_1$) then $(D
\ell)_{S^\lambda_\kappa}$. 
\end{conjecture}

\noindent
Related works are \cite{Sh:922} and G\"obel-Herden-Shelah 
(\cite{GbHeSh:970}).

\noindent
\underline{Acknowledgements}:
We thank the referee for doing much more than duty dictates and for
pointing out much needed corrections and clarifications and Maryanthe
Malliaris and another helper (found by the editor) for pointing  
out many English corrections and misprints.
\bigskip

\noindent
\centerline {$* \qquad * \qquad *$}
\bigskip

\begin{notation}
\label{0p.31}
0) For sets let $u_1 \backslash u_2 \backslash u_2$ mean 
$(u_1 \backslash u_2) \backslash u_3$.

\noindent
1) Usually $\bar C = \langle C_\delta:\delta \in S\rangle$ with $S =
   S(\bar C)$.

\noindent
2) A club of a limit ordinal $\delta$ (e.g. usually a regular cardinal)
   is a closed unbounded subset.

\noindent
3) $S^\lambda_\kappa := \{\delta < \lambda$ : cf$(\delta) = \kappa\}$.
\end{notation}

\begin{definition}
\label{0p.35}
Let $\bar C = \langle C_\delta:\delta \in S\rangle$ and $\lambda$ a
regular cardinal.

\noindent
1) $\bar C$ is a weak $\lambda$-ladder system \when \, $S$ is a
   stationary subset of (the regular cardinal) $\lambda$ and $\delta
   \in S \Rightarrow C_\delta \subseteq \delta$.

\noindent
2) $\bar C$ is a $\lambda$-ladder system \when \, $\lambda$ is
   regular, $S$ is a stationary subset of $\lambda$ and 
$C_\delta \subseteq \delta = \sup(C_\delta)$ for $\delta \in S$.

\noindent
3) $\bar C$ is a strict $\lambda$-ladder system \when \, in addition
$\otp(C_\delta) = \cf(\delta)$.

\noindent
4) $\bar C$ is a strict $(\lambda,\kappa)$-ladder system \when \, in addition
$S \subseteq S^\lambda_\kappa$.

\noindent
5) $\bar C$ is shallow \when \, $\alpha \in \bigcup\limits_{\delta \in
   S} C_\delta \Rightarrow \sup(S) > |\{C_\delta \cap \alpha:\delta
   \in S$ and $\alpha \in C_\delta\}|$.

\noindent
6) In parts (1),(2),(3) we may omit the ``$\lambda$" when clear from
   the content or replace $\lambda$ by $S$.
\end{definition}
\newpage

\section{Preliminaries} \label{Preliminaries}

Most of our results involve $\mu \in {\bold C}$ where

\begin{definition}
\label{1.3.1}
Let ${\bold C} = \{\mu:\mu$ is a strong
limit singular cardinal and pp$(\mu) =^+ 2^\mu\}$, recalling
Definition \ref{0p.19} for $=^+$.
\newline
2) ${\bold C}_\kappa = \{\mu \in {\bold C}:\text{\rm cf}(\mu) = \kappa\}$.
\end{definition}

\noindent
Note that \ref{1.3.15}(2) below which relies on \ref{1.3.14}(1),(1A) repeats
\ref{0p.9}.
\begin{definition}
\label{1.3.14}
1) The set $\cF \subseteq {}^\kappa \mu$ is
called $(\theta,\sigma,J)$-free where $J$ is an 
ideal on $\kappa$ \when \, $[f_1 \ne f_2 \in \cF \Rightarrow
\{i < \kappa:f_1(i) = f_2(i)\} \in J]$ and every $\cF' \subseteq \cF$ 
of cardinality $< \theta$ is $[J,\sigma]$-free which means that:
\mn
\begin{enumerate}
\item[$\bullet$]  there is a sequence $\langle u_f:f \in 
\cF'\rangle$ of members of
$J$ such that for every pair $(\gamma,i) \in \mu \times \kappa$ the
set $\{f \in \cF':f(i) = \gamma \wedge i \notin u_f\}$
has cardinality $< 1 + \sigma$.
\end{enumerate}
\mn
1A) We may replace ``$\cF \subseteq {}^\kappa \mu$" by a sequence
$\bar C = \langle
C_\delta:\delta \in S\rangle,C_\delta$ a set of order type $\kappa$,
or even just such 
a set $\{C_\delta:\delta \in S\}$; meaning that the definition
applies to $\{f_\delta:\delta \in S\}$ 
where for $\delta \in S,f_\delta$ is an increasing function
from $\kappa$ onto $C_\delta$.  Similarly for the other parts.

\noindent
2) If $\sigma = 1$ we may omit it.  If $J = J^{\bd}_\kappa$ we
may omit it so we may say ``${\cF} \subseteq {}^\kappa \mu$ is
$\theta$-free".  Lastly, ``$\cF$ is free" means $\cF$ is $|\cF|^+$-free.

\noindent
3) If $J$ is not an ideal on $\kappa$ but is a subset of
${\cP}(\kappa)$, \then \, we replace ``$u_f \in J$" by 
``$(u_f \in J) \Leftrightarrow (\emptyset \in J)$" and $u_f \subseteq
\kappa$, of course.

\noindent
4) We say a \underline{sequence} $\langle f_\alpha:\alpha < \alpha^*\rangle$ of
members of ${}^\kappa \mu$ is $(\theta,J)$-free \when: $J
\subseteq {\cP}(\kappa)$ and for every $w \subseteq \alpha^*$ of 
cardinality $< \theta$ the sequence $\bar f \rest w$ is $J$-free which
means that there is a sequence
$\langle u_{f_\alpha}:\alpha \in w\rangle$ of subsets of $\kappa$ such that:
$(u_f \in J) \Leftrightarrow (\emptyset \in J)$ and $\alpha \in w \wedge \beta
\in w \wedge \alpha < \beta \wedge i \in \kappa \backslash
u_{f_\alpha} \wedge i \in \kappa \backslash u_{f_\beta} \Rightarrow
f_\alpha(i) < f_\beta(i)$.  Again if $J = J^{\bd}_\kappa$
then we may omit it.

\noindent
5) We say ${\mat F} \subseteq {}^\kappa \mu$ is normal
\when \, $f_1,f_2 \in {\mat F} \wedge f_1(i_1) = f_2(i_2) \Rightarrow i_1
= i_2$.  We say ${\mat F} \subseteq {}^\kappa \mu$ is tree-like \when \,
it is normal and moreover $f_1 \in {\mat F} \wedge f_2 \in {\mat F}_1 \wedge
i < \kappa \wedge f_1(i) = f_2(i) \Rightarrow f_1 \restriction i = f_2
\rest i$.

\noindent
6) For ${\mat F} \subseteq {}^\kappa \mu$ and an ideal $J$ on $\kappa$
let (issp stands for instability spectrum)

\[\begin{array}{ll}
\text{issp}_J({\mat F}) = \{(\theta_1,\theta_2):&\kappa \le
\theta_1 < \theta_2 \text{ and for some } u \subseteq \mu \text{ of
cardinality } \le \theta_1 \\
  &\text{ we have } \theta_2 \le |\{\eta \in {\mat F}:\{i <
  \kappa:\eta(i) \in u\} \in J^+\}|\}.
\end{array}\]

\mn
7) Let $\theta \in \text{ issp}_J({\mat F})$ means $(<\theta,\theta) \in
\text{ issp}_J({\mat F})$ where $(< \theta_1,\theta_2) \in \text{ issp}_J(\cF)$
means that $(\theta'_1,\theta_2) \in \text{ issp}_J({\mat F})$ for 
some $\theta'_1 < \theta_1$.  For $J = J^{\bd}_\kappa$ we may omit
$J$.

\noindent
8) If we write $\text{issp}_J(\langle \eta_s:s \in I\rangle)$ we mean 
$\text{issp}_J(\{\eta_s:s \in I\})$ \underline{but} demand $s_1 \ne
s_2 \in I \Rightarrow \eta_{s_1} \ne \eta_{s_2}$.
\end{definition}

\noindent
Recall
\begin{claim}
\label{1.3.3}
\mn
\begin{enumerate}
\item[$(a)$]  we have $\mu \in {\bold C}$ and moreover, 
$\pp_{J^{\bd}_{\cf(\mu)}}(\mu) =^+ 2^\mu$ \when \,
$\mu$ is a strong limit singular cardinal of uncountable cofinality
\sn
\item[$(b)$]  if $\mu = \beth_\delta > \cf(\mu)$ and
$\delta = \omega_1$ or just $\cf(\delta) > \aleph_0$, \then \,
 $\mu \in {\bold C}_{\cf(\mu)}$ and for a club of
$\alpha < \delta$ we have $\beth_\alpha \in {\bold C}$
\sn
\item[$(c)$]  if $\mu \in \bold C_\kappa$ and $\chi \in (\mu,2^\mu)$
or just $\kappa = \cf(\mu) < \mu$ and $\chi \in 
(\mu,\pp^+_{J^{\bd}_\kappa}(\mu))$, see \ref{0p.19}(5), 
\then \, there is a $\mu^+$-free 
$\cF \subseteq {}^\kappa \mu$ of cardinality
$\chi$, even $<_{J^{\bd}_\kappa}$-increasing $\mu^+$-free
sequence of length $\chi$; moreover if
$(\prod\limits_{i < \kappa} \lambda_i,<_{J^{\bd}_\kappa})$
is $\chi^+$-directed and $\cF_* \subseteq \prod\limits_{i < \kappa}
\lambda_i$ is such that $\cF_*$ is cofinal \underline{or}
$(\cF_*,<_{J^{\bd}_\kappa})$ is well ordered of cardinality $> \chi$
 \then \, we can demand 
$\cF \subseteq \cF_*$ (and there is such
sequence $\langle \lambda_i:i < \kappa\rangle$).
\end{enumerate}
\end{claim}
\bigskip

\begin{PROOF}{\ref{1.3.3}}
Clause (a) holds by \cite[ChII,\S5]{Sh:g}, \cite[ChVII,\S1]{Sh:g}
and clause (b) by \cite[ChIX,\S5]{Sh:g} and clause (c) holds by 
\cite[ChII,2.3,pg.53 + 1.5A,pg.51]{Sh:g}.  
\end{PROOF}

\begin{obs}
\label{1.3.15}
1) If $J$ is a $\sigma$-complete ideal on $\kappa$ and
${\mat F} \subseteq {}^\kappa \mu$ and $\theta_0 < \theta_1 <
\theta_2,(\theta_1,\theta_2) \in \text{\rm issp}_J({\mat F})$ and
{\rm cov}$(\theta_1,\theta_0,\kappa^+,\sigma) < \text{\rm
cf}(\theta_2)$ recalling Definition \ref{0p.21}
(e.g. $\theta_1 < \theta^{+ \omega}_0,\theta_1 < 
\text{\rm cf}(\theta_2))$, \underline{then} $(<\theta_0,\theta_2) 
\in \text{\rm issp}_J({\mat F})$.

\noindent
2) If in addition ${\mat F}$ is tree-like, $J^{\bd}_\kappa \subseteq
J$ and $\kappa$ is regular,
\then \, {\rm cov}$(\theta_1,\theta_0,\kappa^+,\kappa) < 
\text{\rm cf}(\theta_2)$ suffices.

\noindent
3) Assume $J$ is an ideal on $\kappa$ and $\cF \subseteq {}^\kappa
   \mu$ is $(\theta,\sigma,J)$-free.  If $\sigma = 
\text{\rm cf}(\sigma)$ and $\kappa < \sigma$ 
\then \, for every $\cF' \subseteq \cF$ of cardinality $<
   \theta$ we can find $\langle u_f:f \in \cF'\rangle$ as in
   Definition \ref{1.3.14}(1) and a partition $\bar{\cF}' = \langle
 \cF'_\varepsilon:\varepsilon < \varepsilon(*) \le |\cF'|\rangle$
 of $\cF'$ into sets each of cardinality $< \sigma$ 
such that $\langle\{f(i)$: for some $i$ 
we have $f \in \cF'_\varepsilon,i \in \kappa \backslash u_f\}:
\varepsilon < \varepsilon(*)\rangle$ is a sequence
   of pairwise disjoint subsets of $\mu$.  If we waive ``$\kappa <
\sigma$" still for each $i <\kappa$ there is such an $\bar{\cF}^i$
   which can serve for this $i$.

\noindent
4) If $J$ is a $\kappa$-complete ideal on $\kappa$ and $\cF \subseteq
{}^\kappa \mu$ is $(\theta,\kappa^+,J)$-free hence $f_1 \ne f_2 \in
   \cF \Rightarrow \{i < \kappa:f_1(i) = f_2(i)\} \in J$ \then \, $\cF$ is
   $(\theta,J)$-free. 
\end{obs}
\bigskip

\begin{PROOF}{\ref{1.3.15}}
1) This should  be clear as in \cite[ChII,\S6]{Sh:g}, but we give
details.

Let $\cP$ exemplify cov$(\theta_1,\theta_0,\kappa^+,\sigma)$,
i.e. $\cP \subseteq [\theta_1]^{< \theta_0}$ has cardinality
cov$(\theta_1,\theta_0,\kappa^+,\sigma)$ and every $u \in
[\theta_1]^{\le \kappa}$ is included in the union of $< \sigma$
members of $\cP$.

By the assumption ``$(\theta_1,\theta_2) \in \text{\rm issp}_J(\cF)$"
there is $\cU \subseteq \mu$ which has cardinality $\le \theta_1$ such
that $\cF' = \cF'_{\cU} := \{\eta \in \cF:\{i < \kappa:\eta(i) \in
\cU\} \in J^+\}$ has cardinality $\ge \theta_2$.

Let $g$ be a one to one function from $\cU$ into $\theta_1$ and fix
for a while $\eta \in \cF'$.  Let 
$v_\eta := \{g(\eta(i)):i < \kappa$ and $\eta(i)
\in \cU\}$, clearly it is 
$\in [\theta_1]^{\le \kappa}$ hence there is $\cP_\eta
\subseteq \cP$ of cardinality $< \sigma$ such that $v_\eta \subseteq
\cup\{u:u \in \cP_\eta\}$.  So $\{\{i < \kappa:\eta(i) \in \cU$ and
$g(\eta(i)) \in u\}:u \in \cP_\eta\}$ is a family of $< \sigma$
subsets of $\kappa$ whose union belongs to $J^+$.  But $J$ is a
$\sigma$-complete ideal on $\kappa$ hence there is
\begin{enumerate}
\item[$\circledast$]  $u_\eta \in \cP_\eta$ such that $\{i <
\kappa:\eta(i) \in \cU$ and $g(\eta(i)) \in u_\eta\} \in J^+$.
\end{enumerate}
So $\langle u_\eta:\eta \in \cF'\rangle$ is well defined and $\eta \in
\cF' \Rightarrow u_\eta \in \cP$ but $|\cP| = \text{
cov}(\theta_1,\theta_0,\kappa^+,\sigma) < \text{ cf}(\theta_2)$ and
$\cF'$ was chosen such that $|\cF'| \ge \theta_2$, hence for some $u_2
\in \cP$ the family $\cF'' := \{\eta \in \cF':u_\eta = u_2\}$ has
cardinality $\ge \theta_2$.  But then letting $u_1 = \{\alpha \in
\cU:g(\alpha) \in u_2\}$ we have $\cF_* := \{\eta \in \cF:\{i < \kappa:\eta(i)
\in u_1\} \in J^+\} = \{\eta \in \cF:\{i < \kappa:g(\eta(i)) \in u_2\}
\in J^+\} \supseteq \cF''$ hence the subfamily $\cF''$ of $\cF$ has
cardinality $\ge |\cF''| \ge \theta_2$.  Also $|u_1| = |u_2| <
\theta_0$ by the choice of $(g,u_1)$ and as $u_2 \in \cP \subseteq
      [\theta_1]^{< \theta_1}$.

So $u_1$ exemplifies that $(< \theta_0,\theta_2) \in \text{\rm
issp}_J(\cF)$, the desired conclusion.

\noindent
2) As \wilog \, $J = J^{\text{\rm bd}}_\kappa$ and this ideal is
   $\kappa$-complete.

\noindent
3) Easy, too.  

\noindent
4) By part (3) and \ref{a10}(1).
\end{PROOF}

\begin{claim}
\label{a10}
Let $\cF \subseteq {}^\kappa \mu$ and $J$ an ideal on $\kappa$ be
such that $f_1 \ne f_2 \in \cF \Rightarrow \{i < \kappa:f_1(i) =
f_2(i)\} \in J$.

\noindent
1) $\cF$ is $(\theta^+,J)$-free if $J$ is $\theta$-complete.

\noindent
2) If $\kappa < \sigma < \lambda$ \then \,: $\cF$ is
$(\lambda,\sigma,J)$-free iff there are no regular $\partial \in
   [\sigma,\lambda)$ and pairwise distinct $f_\alpha \in \cF$ for
   $\alpha < \partial$ such that $S = \{\delta < \partial$: for some
   $\zeta \in [\delta,\partial)$ the set $\{i < \kappa:f_\zeta(i) \in
\{f_\varepsilon(i):\varepsilon < \delta\}$ belongs to $J^+\}$ is a
stationary subset of $\partial$.

\noindent
2A) In part (2), the two equivalent statements imply that for no 
$\theta \in [\sigma,\lambda),\theta \in \text{\rm issp}_J(\cF)$.

\noindent
3) Assume we are given a sequence $\bar f =
\langle f_\alpha:\alpha < \alpha_*\rangle$ of members of
${}^\kappa${\rm Ord} with no repetitions, 
and $\lambda = \text{\rm cf}(\lambda) > \kappa$
and $J$ is an ideal on $\kappa$.  

\noindent
\Then \, $\bar f$ is not $(\lambda,\lambda,J)$-free as a set iff there is an
increasing sequence $\langle \alpha_\varepsilon:\varepsilon <
\lambda\rangle$ of ordinals $<  \alpha_*$ such that 
the set $S = \{\varepsilon < \lambda$: {\rm cf}$(\varepsilon) \le \kappa$ and
$\{i < \kappa:(\exists \zeta < \varepsilon)(f_{\alpha_\varepsilon}(i) 
= f_{\alpha_\zeta}(i))\} \in J^+\}$ is
 a stationary subset of $\lambda$.

\noindent
4) In part (3) if in addition $\bar f$ is 
tree-like, i.e., $f_\alpha(\varepsilon)  =
f_\beta(\varepsilon) \Rightarrow f_\alpha \restriction \varepsilon
= f_\beta \restriction \varepsilon$ and $J^{\bd}_\kappa
\subseteq J$ then $S \subseteq S^\lambda_\kappa$.
\end{claim}

\begin{PROOF}{\ref{a10}}
1) Easy and more is proved in the proof of \ref{1f.13} below.

\noindent
2) Proved in proving $\boxplus$ suffice in the proof of \ref{1f.10}.

\noindent
2A) Easy, see Definition \ref{1.3.14}(6).

\noindent
3) By \ref{1.3.15}.

\noindent
4)  Like part (2), see more in \ref{1f.41}.
\end{PROOF}

\begin{claim}
\label{1f.41}
Assume $\lambda > \mu \ge \kappa_2 \ge \kappa_1 = \theta = \cf(\theta)$.

\noindent
1) ${\cF} \subseteq {}^\theta${\Ord}
is $(\kappa_2,\kappa_1)$-free \underline{iff} ${\mat F}$ is 
$(\kappa^+,\kappa)$-free for every regular $\kappa \in
[\kappa_1,\kappa_2)$.

\noindent
2) There is a $(\kappa^{+ \omega+1},\kappa)$-free set ${\cF} 
\subseteq {}^\omega \mu$ of cardinality $\lambda$ 
\underline{iff} for every $n < \omega$ there is a 
$(\kappa^{+n},\kappa)$-free set ${\cF} 
\subseteq {}^\omega \mu$ of cardinality $\lambda$.

\noindent
3) Assume $\lambda > \mu \ge \kappa^{+ \omega},\mu > \sigma = 
\cf(\mu)$ and $(\forall \alpha < \mu)(|\alpha|^\chi < \mu)$.
If ${\cF}_\varepsilon \subseteq {}^\theta \mu$ has cardinality
   $\lambda$ for $\varepsilon < \chi$, \then \, we can find ${\cF}
\subseteq {}^\theta \mu$ of cardinality $\lambda$ such that:
\mn
\begin{enumerate}
\item[{}] if for some $\varepsilon,\cF_\varepsilon$ is
$(\kappa_2,\kappa_1)$-free, \then \, ${\cF}$ 
is $(\kappa_2,\kappa_1)$-free.
\end{enumerate}
\mn
4) In part (3); if $\chi = \theta$ then we can assume just $(\forall
\alpha < \mu)(|\alpha|^{<\chi} \le \mu)$.

\noindent
5) In (1)-(3) we can use an ideal $J$ on $\theta$.
\end{claim}

\begin{remark}
See \ref{a10}, \ref{1f.10}.
\end{remark}

\begin{PROOF}{\ref{1f.41}}
1) By \ref{a10}(2).

\noindent
2) By \ref{1f.28}(1A) there is a $(\kappa^{+ \omega},\kappa)$-free
$\cF \subseteq {}^\omega \mu$ and by the compactness theorem for
singulars it follows that $\cF$ is $(\kappa^{+ \omega +1},\kappa)$-free,
(really an obvious case).

\noindent
3) Let $\langle \lambda_i:i < \sigma\rangle$ be increasing
with limit $\mu, \lambda_i = \lambda^\chi_i$ and let $\cd_i:{\cH}_{\le
\chi}(\lambda_i) \rightarrow \lambda_i$ be one-to-one and onto; and let
${\mat F}_\varepsilon = \{f^\varepsilon_\alpha:\alpha < \lambda\}$.
Lastly, $f_\alpha \in {}^\sigma \mu$ is defined by 
$f_\alpha(i) = \cd_i(\langle f^\varepsilon_\alpha \cap (\lambda_i \times
\lambda_i):\varepsilon < \chi\rangle)$. 
\end{PROOF}

\noindent
In particular recalling \ref{0p.6}(2)
\begin{claim}
\label{1f.13}
1) Assume ${\cF} \subseteq {}^\kappa \mu$ is
$(\theta,\kappa^{++},J^{\bd}_\kappa)$-free and $\kappa = 
\cf(\kappa) < \mu$.  \Then \, we can find ${\cG} \subseteq 
{}^{(\kappa^+ \times \kappa)}\mu$ of cardinality
$|{\cF}|$ such that ${\cG}$ is $(\theta,
J^{\text{\rm bd}}_{\kappa^+ \times \kappa})$-free and normal.

\noindent
2) If $\lambda = \cf(\lambda) > \mu > \kappa = 
\cf(\kappa)$ and there is a $\theta$-free $\cF \subseteq {}^\kappa
   \mu$ of cardinality $\ge \lambda$ and $S \subseteq
   S^\lambda_\kappa$ is stationary and for simplicity $\delta \in S
\Rightarrow \mu \cdot \delta = \delta$ \then \, there is a $\theta$-free
   strict $S$-ladder system $\langle C_\delta:\delta \in S \rangle$.

\noindent
2A) In part (2) also for every $\sigma = \text{\rm cf}(\sigma) \in
(\kappa,\lambda)$ and stationary $S \subseteq S^\lambda_\sigma$
\underline{there} is a $(\theta,J_{\sigma * \theta})$-free strict
$S$-ladder system $\langle C_\delta:\delta \in S\rangle$.
\end{claim}

\begin{PROOF}{\ref{1f.13}}
1) If $\mu = \kappa^+$ then the construction below gives $\cG
\subseteq {}^{\kappa^+ \times \kappa}(\kappa^+ + \mu)$ rather than
$\cG \subseteq {}^{\kappa^+ \times \kappa}(\mu)$, but this is enough 
so we shall ignore this point.
For $f \in {\mat F}$ let $g_f:\kappa^+ \times \kappa
\rightarrow \mu$ be defined by:
\mn
\begin{enumerate}
\item[$(*)_0$]  for $\zeta < \kappa^+,i<\kappa$ we let
$g_f(\zeta,i) = \kappa^+ \cdot f(i) + \kappa \cdot \zeta + i$.
\end{enumerate}
\mn
Let ${\cG} = \{g_f:f \in {\mat F}\}$, now
\mn
\begin{enumerate}
\item[$(*)_1$]  if $f_1 \ne f_2 \in \cF$ then $g_{f_1} \ne g_{f_2}$ and
moreover $\{(\zeta,i) \in \kappa^+ \times \kappa:g_{f_1}(\zeta,i) =
g_{f_2}(\zeta,i)\} \in J^{\text{\rm bd}}_{\kappa^+ \times \kappa}$.
\end{enumerate}
\mn
[Why?  By Definition \ref{1.3.14}(1) we 
know $i(*) :=  \sup\{i < \kappa:f_1(i) = f_2(i)\} < \kappa$ and hence
$\{(\zeta,i) \in \kappa^+ \times \kappa:g_{f_1}((\zeta,i)) =
g_{f_2}((\zeta,i))\} \subseteq 
\{(\zeta,i):\zeta < \kappa^+$ and $i < i(*)\} \in 
J^{\bd}_{\kappa^+ \times \kappa}$, so we are done.]
\mn
\begin{enumerate}
\item[$(*)_2$]  assume ${\mat G}' \subseteq {\cG}$ is of
cardinality $< \theta$ and we shall find $\langle u^1_g:g \in
\cG'\rangle$ as required.
\end{enumerate}
\mn
Why?  We can choose ${\cF}' \subseteq {\cF}$ of cardinality $<
\theta$ such that ${\cG}' = \{g_f:f \in {\cF}'\}$.  We can apply
the assumption ``${\cF}$ is $(\theta,\kappa^{++})$-free" and let
$\langle u_f:f \in {\mat F}'\rangle$ be as in Definition
\ref{1.3.14}(1); moreover let $\langle {\cF}_\varepsilon:\varepsilon
< \varepsilon(*)\rangle$ be as guaranteed in \ref{1.3.15}(3), so in
particular $|{\mat F}_\varepsilon| \le \kappa^+$.

For each $\varepsilon < \varepsilon(*)$ 
let $\langle f_{\varepsilon,\iota}:\iota <
|\cF_\varepsilon|\rangle$ list ${\cF}_\varepsilon$ with no repetitions
and let $g_{\varepsilon,\iota} = g_{f_{\varepsilon,\iota}}$.  First
assume $|{\cF}_\varepsilon| \le \kappa$, then for $\iota <
|\cF_\varepsilon|$ we let $u^0_{\varepsilon,\iota} =
\{i < \kappa$: the sequence $\langle f_{\varepsilon,\iota_1}(i):\iota_1 \le
\iota\rangle$ has some repetitions \underline{or}\footnote{actually
  ``$i \in u_{f_{\varepsilon,\iota}}$" suffice}
$i \in \cup\{u_{f_{\varepsilon,\iota_1}}:\iota_1 \le \iota\}\}$.  As
$J^{\text{\rm bd}}_\kappa$ is $\kappa$-complete, clearly
$u^0_{\varepsilon,\iota} \in J^{\bd}_\kappa$ and we let 
$u^1_{g_{\varepsilon,\iota}} := \kappa^+ \times u^0_{\varepsilon,\iota}$.

Second, assume $|{\cF}_\varepsilon| = \kappa^+$ and for each $\zeta
\in [\kappa,\kappa^+)$ let $\langle \xi(\zeta,j):j < \kappa\rangle$
list $\zeta$ without repetition and for $\zeta \in [\kappa,\kappa^+),
j < \kappa$ let

$u^0_{\varepsilon,\zeta,j} = \{i < \kappa$: the sequence $\langle
f_{\varepsilon,\xi(\zeta,j_1)}(i):j_1 \le j\rangle$ has some
repetitions or $i \in \{u_{f_{\varepsilon,\xi(\zeta,j_1)}}:j_1 \le
j\}\}$

and for $\iota < |\cF_\varepsilon|$ let

$u^1_{g_{\varepsilon,\iota}} = \{(\zeta,i):\zeta \in (\kappa + 
\iota,\kappa^+),i
< \kappa$ and $i \in u^0_{\varepsilon,\zeta,j}$ where $j$ is the
unique $j < \kappa$ such that $\iota = \xi(\zeta,j)\}$.

Now check that $\langle u^1_{g_{\varepsilon,\iota}}:\varepsilon <
\varepsilon(*)$ and $\iota < |\cF_\varepsilon|\rangle$ is as required,
i.e. witnessing the freeness of $\cF'$.

\noindent
2) Let $\langle f_\delta:\delta \in S\rangle$ be a
   sequence of pairwise distinct members of $\cF$ and for $\delta \in
   S$ let $\langle \alpha_{\delta,i}:i < \kappa\rangle$ be an 
increasing sequence of ordinals with limit $\delta$.  

Lastly, let $C_\delta = \{\mu \alpha_{\delta,i} + 
f_\delta(i):i < \kappa\}$ for $\delta \in S$ recalling $\delta \in S
\Rightarrow \delta = \mu \cdot \delta$.

\noindent
2A) The proof is similar.  
\end{PROOF}

\noindent
How is this connected to Abelian groups?
\begin{definition}
\label{1f.15}
1) We say that $G$ is an abelian group derived from $\cF \subseteq
{}^\omega \mu$ \when \, $G$ is generated by $\{x_\alpha:\alpha < \mu\} 
\cup \{y_{\eta,n}:\eta \in \cF$ and $n < \omega\}$ freely
except a set of equations $\Gamma = \cup\{\Gamma_\eta:\eta \in
{\cF}\}$ where each $\Gamma_\eta$ has the form
 $\{y_{\eta,n} = a_{\eta,n} \cdot 
y_{\eta,n+1} + x_{\eta(n),n}:n < \omega\}$ where
$a_{\eta,n} \in \bbZ \backslash \{-1,0,1\}$. 

\noindent
2) We say that $G$ is an abelian group derived from ${\cF}
   \subseteq {}^{\omega_1 \times \omega}\mu$ \when \, $G$ is generated
by $\{x_{\alpha,\varepsilon,n}:\alpha < \mu$ and $\varepsilon <
\omega_1,n < \omega\} \cup \{y_{\eta,\varepsilon,n}:\eta
\in {\cF},\varepsilon < \omega_1,n < \omega\} \cup \{z_{\eta,n}:\eta
   \in \cF$ and $n < \omega\}$ freely except a
   set of equations $\Gamma = \cup\{\Gamma_\eta:\eta \in {\cF}\}$
where each $\Gamma_\eta$ has the form 

$\{y_{\eta,\varepsilon,n} = a_{\eta,\varepsilon,n} y_{\eta,\varepsilon,n+1} +
b_{\eta,\varepsilon,n} z_{\eta,\rho_{\eta,\varepsilon}(n)} +
c_{\eta,\varepsilon,n} x_{\eta(\varepsilon,n),\varepsilon,n}:
\varepsilon < \omega_1,n < \omega\}$ 

\noindent
\underline{where}

$a_{\eta,\varepsilon,n} \in \bbZ \backslash
\{-1,0,1\},b_{\eta,\varepsilon,n} \in \bbZ \backslash \{0\},
c_{\eta,\varepsilon,n} \in \bbZ,\rho_{\eta,\varepsilon} \in
{}^\omega \omega$ is increasing and $\varepsilon_1 <
\varepsilon_2 < \omega_1 \Rightarrow \text{ Rang}
(\rho_{\eta,\varepsilon_1}) \cap 
\text{ Rang}(\rho_{\eta,\varepsilon_2})$ is finite. 
\end{definition}

\begin{remark}
\label{a24}
1) Here choosing $\rho_{\eta,\varepsilon} \in {}^\omega(\omega +
\varepsilon)$ is alright but not for \S4.

\noindent
2) So in \ref{1f.15} if $a_{\eta,n} = n+1$, considering $G$ as a
   metric space with $\bold d_G(x,y) = \inf\{2^{-n}:x-y \in (n!)G\}$ we
   have $y_{\eta,n} = \sum\limits_{m \ge n} (m!)/(n!) x_{\eta(m)}$ for
   $\eta \in \cF,n < \omega$.  In general for $n_1 < n_2$ we have

\[
y_{\eta,n_1} = (\sum\limits^{n_2-1}_{m=n_1} (\prod\limits^{m}_{\ell=n_1}
a_{\eta,\ell}) \, x_{\eta(m),m}) +
(\prod\limits^{n_2-1}_{m=n_1} a_{\eta,m}) y_{\eta,n_2}.
\]
\end{remark}

\noindent
Easily (see \cite{EM02} on the subject):
\begin{claim}
\label{1f.14}
If ${\cF} \subseteq {}^\omega \mu$ is $\theta$-free \underline{or}
${\cF} \subseteq {}^{\omega_1 \times \omega}\mu$ is
$(\theta,J^{\bd}_{\omega_1 \times \omega})$-free, 
\then \, any abelian group derived from it is $\theta$-free. 
\end{claim}

\noindent
Similarly to \ref{1.3.15}
\begin{claim}
\label{1f.16}
1) If $\cF \subseteq {}^{\Dom(J)}\mu$ is
   $(\theta,\sigma^+_2,J)$-free and $J$ is a
  $(\sigma_2,\sigma^+_1)$-regular\footnote{that is, there are
   $A_\alpha \in J$ for $\alpha < \sigma_2$ such that $u \subseteq
   \sigma_2 \wedge |u| \ge \sigma^+_1 \Rightarrow
   \cup\{A_\alpha:\alpha \in u\} = \Dom(J)$.}
 and $\sigma_1$-complete ideal \then \, ${\cF}$ is $(\theta,J)$-free.

\noindent
2) Assume $I,J$ is an ideal on $S,T$ respectively.  If ${\cF}
   \subseteq {}^S \mu$ is $(\theta,\sigma,I)$-free, $\pi$ is a function
 from $T$ onto $S$ and $\pi''(J) := \{\{\pi(i):i \in s\}:s 
\in J\} \supseteq I$ 
\then \, ${\cF} \circ \pi = \{f \circ \pi:f \in {\cF}\} \subseteq {}^T \mu$ 
is $(\theta,\sigma,J)$-free.
\end{claim}

\begin{definition}
\label{a26}
1) Let $(D \ell)_S$ mean that:
\mn
\begin{enumerate}
\item[$(a)$]  $\lambda = \sup(S)$ is a regular uncountable cardinal
\sn
\item[$(b)$]  $S$ is a stationary subset of $\lambda$
\sn
\item[$(c)$]  there is a witness $\bar{\mat P}$ by which we mean:
\sn
\begin{enumerate}
\item[$(\alpha)$]  $\bar{\cP} = \langle {\cP}_\alpha:\alpha \in S\rangle$
\sn
\item[$(\beta)$]  ${\cP}_\alpha \subseteq {\mat P}(\alpha)$
has cardinality $< \lambda$
\sn
\item[$(\gamma)$] for every subset ${\mat U}$ of $\lambda$, the
set $S_{\cU} := \{\delta \in S:
{\cU} \cap \delta \in {\mat P}_\delta\}$ is a stationary subset of $\lambda$.
\end{enumerate}
\end{enumerate}
\mn
2) Let $(D \ell)^*_S$ be defined similarly but in clause $(c)(\gamma)$ we
demand $S \backslash S_{\cU}$ is not stationary. 

\noindent
3) We write $(D \ell)_{D,S},(D \ell)^*_{D,S}$ when $D$ is a normal
filter on $\lambda$ and replace ``stationary" by ``$\in D^+$".
\end{definition}

\begin{definition}
\label{0p.32}  
1) For a regular uncountable
cardinal $\lambda$ let $\check I[\lambda] = \{S \subseteq \lambda$: some
pair $(E,\bar a)$ witnesses $S \in \check I(\lambda)$, see below$\}$. 

\noindent
2) We say that $(E,\bar u)$ is a witness for $S \in \check I[\lambda]$
\when \,:
\mn
\begin{enumerate}
\item[$(a)$]   $E$ is a club of the regular cardinal $\lambda$
\sn
\item[$(b)$]   $\bar u = \langle u_\alpha:\alpha < \lambda
\rangle,u_\alpha \subseteq \alpha$ and $\beta \in u_\alpha \Rightarrow
u_\beta = \beta \cap u_\alpha$
\sn
\item[$(c)$]   for every $\delta \in E \cap S,u_\delta$ is an
unbounded subset of $\delta$ of order-type $< \delta$ (and $\delta$ is
a limit ordinal).
\end{enumerate}
\end{definition}

\begin{claim}
\label{1.3.23}
1) If $\lambda = \lambda^{< \lambda}$
and $\kappa = \cf(\kappa) < \lambda$ and $\alpha < \lambda
\Rightarrow |\alpha|^{<\kappa>_{\tr}} <\lambda$ and $S \subseteq
S^\lambda_\kappa$ is a stationary subset of $\lambda$, \then \, $(D \ell)_S$.

\noindent
2) If $\mu$ is a strong limit cardinal and
$\lambda = \cf(\lambda) > \mu$, \then \, $\mu > 
\sup\{\kappa < \mu:\kappa =\cf(\kappa)$ and $(\exists \alpha
< \lambda)(|\alpha|^{<\kappa>_{\tr}} \ge \lambda)\}$.

\noindent
3) If $\lambda = \lambda^{<\lambda} > \beth_\omega$, \then \,
$\{\kappa:\kappa = \cf(\kappa)$ and $\beth_\omega(\kappa) <
\lambda$ and $\neg(D \ell)_{S^\lambda_\kappa}$ or just 
$\neg(D \ell)^*_S$ for some stationary $S \in 
\check I_\kappa[\lambda]\}$ is finite where $\check I_\kappa[\lambda]$
is from \ref{0p.32}.

\noindent
4) If $\lambda = \chi^+$ and $S \subseteq \lambda$ is 
stationary, \then \, $(D \ell)^*_S$ is equivalent 
to $\diamondsuit_S$.

\noindent
5) If $\lambda > \kappa$ are regular and $S \in \check
   I_\kappa[\lambda]$ is a stationary subset of $\lambda$ \then \,
there is a shallow, use \ref{0p.35}(5) strict $S$-club system.
\end{claim}

\begin{PROOF}{\ref{1.3.23}}
1), 2), 3): \, See \cite{Sh:829}.
\newline
4) A result of Kunen; for a proof of a somewhat more general result see
\cite{Sh:247}.

\noindent
5) See \cite{Sh:108} or \cite{Sh:88a}.
\end{PROOF}

\begin{discussion}
\label{1f.18}
1) Of course, $(D \ell)_S$ is a 
relative of the diamond, see \cite{Sh:107}.

\noindent
2) $(D\ell)^*_S$ is consistently not equivalent to $\diamondsuit^*_S$
when $\lambda$ is a limit (regular) cardinal.

\noindent
3) Trivially $(D\ell)^*_S \Rightarrow (D\ell)_S$.

For $\boxdot_3$ of \S0, (it was previously known only when $\chi$ is 
regular by using partial squares which holds by \cite[\S4]{Sh:351}).
\end{discussion}

\begin{fact}
\label{a35}
If $\lambda = 2^\chi = \chi^+ > \kappa = \cf(\kappa)$ and $\kappa \ne
\cf(\chi)$ then $\diamondsuit_{S^\lambda_\kappa}$ moreover
$\diamondsuit_S$ for every stationary $S \subseteq S^\lambda_\kappa$.
\end{fact}

\begin{PROOF}{\ref{a35}}
By \cite{Sh:922}.
\end{PROOF}

\noindent
Now by \cite[1.10]{Sh:775}, this is used in \ref{h.7}, \ref{h.9d}.
\begin{theorem}
\label{1.3.25}
We have {\rm BB}$(\lambda,\bar C,(\lambda,\theta),< \mu)$ recalling
\ref{0p.15}(1),(3),(4)  \when \,:
\mn
\begin{enumerate}
\item[$(a)$]  $\mu \in \bold C_\kappa,\lambda = \cf(2^\mu)$ and
  $\theta < \mu,\sigma = \cf(\sigma) < \mu$,
\sn
\item[$(b)$]  $S \subseteq S^\lambda_\sigma$ is stationary,
\sn
\item[$(c)$]  $\bar C = \langle C_\delta:\delta \in S\rangle,C_\delta
\subseteq \delta,|C_\delta| \le \mu$ recalling\footnote{actually
 $2^{|C_\delta|} \le 2^\mu$ is sufficient} \ref{0p.15}(4),
\sn
\item[$(d)$]  $\chi < 2^\mu \Rightarrow 
\chi^{<\sigma>_{\tr}} < 2^\mu$,
\sn
\item[$(e)$]  $\bar C$ is shallow, that is, 
$|\{C_\delta \cap \alpha:\alpha \in C_\delta\}| <
\lambda$ for $\alpha < \lambda$.
\end{enumerate}
\end{theorem}

\begin{remark}
\label{a38}
1) Of course, if $S \in \check I_\kappa[\lambda]$ is stationary \then
   \, there is $\bar C$ as in clauses (c) + (e) (and, of course, (b)).

\noindent
2) There are such stationary $S$ as $\kappa^+ < \mu < \lambda$ by
   \cite{Sh:420}. 
\end{remark}

\begin{definition}
\label{1.3.27}
We say a filter $D$ on a set $X$
is weakly $\lambda$-saturated \when \, there is no partition $\langle
X_\alpha:\alpha < \lambda\rangle$ of $X$ such that $\alpha < \lambda
\Rightarrow X_\alpha \in D^+ := \{Y \subseteq X:X \backslash Y \notin D\}$.
\end{definition}
\bigskip

\centerline {$* \qquad * \qquad *$}
\bigskip

A notable consequence of the analysis in this work is the BB (Black Box)  
Trichotomy Theorem \ref{h.7}.

\begin{remark}
\label{h.15}
Using $\bar C = \langle C_\delta:\delta \in S\rangle$ below
or using $\bar f = \langle f_\delta:\delta \in S\rangle$ where
$f_\delta$ is an increasing function from $\otp(C_\delta)$ onto $C_\delta$,
does not make a real difference.
\end{remark}

\begin{trichotomy}
\label{h.7}
If $\mu \in {\bold C}_\kappa$ and $\kappa > \sigma = \cf(\sigma)$,
\then \, at least one of the following holds:
\mn
\begin{enumerate}
\item[$(A)_{\mu,\kappa}$]  there is a 
$\mu^+$-free ${\cF} \subseteq {}^\kappa \mu$ of cardinality $2^\mu$
\sn
\item[$(B)\,\,\,$]  $(a) \quad \lambda := 2^\mu = \lambda^{<\lambda}$ 
(so $\lambda$ is regular) and $\chi < \lambda \Rightarrow 
\chi^\sigma < \lambda$
\sn
\item[${{}}$]  $(b)_{\lambda,\mu,\sigma} \quad$ if 
$S \subseteq S^\lambda_\sigma$ is
stationary, $\bar C = \langle C_\delta:\delta \in S\rangle$ is a weak
ladder system

\hskip25pt  (i.e., $C_\delta \subseteq \delta$ so, e.g., the 
choice $C_\delta = \delta$ for $\delta \in S$ is all right); 

\hskip25pt \then \, 
\sn
\item[${{}}$]  $(c)_{\lambda,\mu,\sigma} \quad$ letting
$J^{\nst}_S = \{A \subseteq \lambda:A \cap S$ is not 
stationary in $\lambda\}$ we have\footnote{What about freeness?  
We may get it by the choice of $\bar C$, also if 
$\bar C$ is a ladder system (particularly if strictly), 
we get a weak form, e.g. stability.}
\sn
\begin{enumerate}
\item[${{}}$]  $(i) \quad \BB(J^{\nst}_S,\bar C,\theta,\le \mu)$ 
for every $\theta < \mu$ provided that

\hskip30pt $\delta \in S \Rightarrow |C_\delta| < \mu$, see
\ref{0p.15}(4)
\sn
\item[${{}}$]  $(ii) \quad \BB(J^{\nst}_S,\bar C,(2^\mu,\theta),<
  \lambda)$ for any $\theta < \mu$ 
\end{enumerate}
\sn
\item[$(C)_{\mu,\kappa}$]  $(a) \quad \lambda_2 = 2^\mu$ is regular, $\chi <
\lambda_2 \Rightarrow \chi^\sigma < \lambda_2$ and

\hskip25pt $\lambda_1 = \min\{\partial:2^\partial 
> 2^\mu\}$ is (regular and) $< 2^\mu$
\sn
\item[${{}}$]  $(b) \quad$ like $(b)_{\lambda,\mu,\sigma}$ of clause
  (B) for $\lambda = \lambda_2$ but $|C_\delta| < \lambda_1$ for $\delta \in S$

\hskip25pt  (so $C_\delta = \delta$ is not all right)
\sn
\item[${{}}$]  $(c) \quad \BB(J^{\nst}_S,\mu^+,\theta,\kappa)$
for any $\theta < \mu$ and any stationary 
subset $S$ of $\lambda_1$
\sn
\item[${{}}$]  $(c)' \quad$ like $(b)_{\lambda,\mu,\sigma}$ of $(B)$
  but for $\lambda = \lambda_1,S$ a club or just $S$ not in the weak 

\hskip25pt diamond ideal (\cite{DvSh:65}).
\end{enumerate}
\end{trichotomy}
\bigskip

\begin{remark}
\label{h.7d}
1) If $\kappa = \aleph_0$ above, then there is no infinite cardinal 
$\sigma < \kappa$ as required, but the proof still gives something 
(e.g. for $\sigma = \aleph_1$).  In this case we cannot get ``for every 
stationary $S \subseteq S^\lambda_\sigma$", still by 
\cite[3.1]{Sh:829} one has ``for all but finitely many 
regular $\sigma < \mu$ for almost every stationary $S 
\subseteq S^\lambda_\sigma$"; see \ref{1.3.23}.

\noindent
2) Assume $\mu \in \bold C_\kappa,\lambda = 2^\mu = \chi^+$.  If
   $\chi$ is regular then (A) of \ref{h.7} holds because by \ref{2b.91}, 
there is $\bar C = \langle C_\delta:\delta \in S^\lambda_\kappa,\mu$
   divides $\delta\rangle,C_\delta \subseteq \delta =
\sup(C_\delta),\otp(C_\delta) = \kappa$ and $\bar C$ is $\mu^+$-free
   and shallow.  
If $\kappa \ne \cf(\chi)$ and $\lambda = \lambda^{< \lambda}$ then
for every stationary $S \subseteq S^\lambda_\kappa$ we have
   $\diamondsuit_S$, see \cite{Sh:922}.

\noindent
3) What happens if $\lambda := 2^\mu$ is weakly inaccessible?  Here it
   seems plausible to assume, for some $\mu_0$
\mn
\begin{enumerate}
\item[$(*)$]  $(a) \quad \mu \le \mu_0 < \lambda$
\sn
\item[${{}}$]  $(b) \quad \alpha < \lambda \Rightarrow \lambda >
\text{ cov}(|\alpha|,\mu^+_0,\mu,2)$
\sn
\item[${{}}$]  $(b)^+ \quad \alpha < \lambda \Rightarrow \lambda >
\text{ cov}(|\alpha|,\mu^+_0,\mu^+_0,2)$.
\end{enumerate}
\mn
Now $(b)^+$ implies (by \cite{DjSh:562})
\mn
\begin{enumerate}
\item[${{}}$]  $(c) \quad$ there is $\bar{\mat P}$ such that
\sn
\begin{enumerate}
\item[${{}}$]  $(\alpha) \quad \bar{\mat P} = \langle {\cP}_\alpha:
\alpha < \lambda\rangle$,
\sn
\item[${{}}$]  $(\beta) \quad |{\mat P}_\alpha| < \lambda$
\sn
\item[${{}}$]  $(\gamma) \quad{\mat P}_\alpha \subseteq \{u:|u| \le
\mu_0,u$ is a closed subset of $\alpha\}$,
\sn
\item[${{}}$]   $(\delta) \quad$ if
$\alpha \in u \in {\mat P}_\beta$, then $u \cap \alpha \in
  {\cP}_\alpha$,
\sn
\item[${{}}$]   $(\varepsilon) \quad$ if $\delta < \lambda,\cf(\delta)
  \le \mu_0$ then $\sup(u) = \delta$ for some $u \in \cP_\delta$.
\end{enumerate}
\end{enumerate}
\mn
This is enough for the argument above.

\noindent
4) Does clause (b) in $(*)$ above suffice?
\end{remark}

\begin{PROOF}{\ref{h.7}}
\underline{Proof of \ref{h.7}}:

Recall that for every $\chi \in (\mu,2^\mu)$ there is a $\mu^+$-free
${\mat F} \subseteq {}^\kappa \mu$ of cardinality $\chi$ (see
\ref{1.3.3}(c)).

If for some $\chi < 2^\mu$ we have $\chi^\sigma = 2^\mu$ then by
\ref{1f.21}, clause (A) holds (when $\theta$ there stands for $\sigma$
here), so we can assume that there is no such $\chi$.  If $2^\mu$ is a 
singular cardinal then by \ref{1f.28}(3),
clause (A) holds, so we can assume that $\lambda := 2^\mu$ is regular.
Now assume $\lambda = \lambda^{< \lambda}$ and 
we shall prove clause (B).  Obviously 
clause (B)(a) holds and clause (B)(b)(ii) holds by \ref{1.3.25} above and
clause (B)(b)(i) follows.  Note that any strict club system $\langle
C_\delta:\delta \in S\rangle$ is shallow as 
$|\{C_\delta \cap \alpha:\delta \in S$
satisfies $\alpha \in C_\delta\}| \le |\alpha|^{< \sigma} \le
|\alpha|^\sigma < \lambda$.

So assume $\lambda < \lambda^{<\lambda}$, hence necessarily there is
$\partial < \lambda$ such that $\lambda < 2^\partial$.

Assume $\lambda_1 = \min\{\chi:2^\chi > 2^\mu\} < \lambda_2 := 2^\mu$, 
then trivially clause (C)(a) holds and by
Conclusion \ref{d.11}(1) clauses $(C)(c),(c)'$ hold.  
Clause (b) of (C) holds by \cite{Sh:775},
i.e. \ref{1.3.25}, because we are assuming $(\forall \chi <
\lambda)(\chi^\sigma < \lambda)$ so clause (C) holds.
\end{PROOF}

\begin{remark}
How can the Black Box Trichotomy Theorem \ref{h.7} help?

If possibility (A) holds for $\kappa \in \{\aleph_0,\aleph_1\}$, 
we have, e.g., abelian groups as in
Definition \ref{1f.15}; so we have
$G_0 \subseteq_{\pr} G_1$ (that is, $G_0$ is a pure subgroup of the
abelian group $G_1$) such that $G_1$ is torsion-free, $G_0$
is free, $G_1$ quite free, $|G_0| = \mu$ and, e.g. if $a_{\eta,n} =
n+1$, then  $G_1/G_0$ is divisible, and a list of $|G_1| = 2^\mu$
partial endomorphisms of $G_1$ such that if $G_0 \subseteq_{\pr}
G \subseteq_{\pr} G_1$, any endomorphism of $G$ is included in one
of the endomorphisms in the list.  So by diagonalization we can  build 
an endo-rigid group.  On the other hand, possibilities (B),(C) 
help in another way: as in black boxes, see \cite{EM02},
\cite{GbTl12}, this is continued in \cite{Sh:F1200}.
\end{remark}

\noindent
Recall
\begin{definition}
\label{2b.113}
Assume $J$ is an ideal of $\kappa$ and $\bar f = \langle
f_\alpha:\alpha < \alpha(*)\rangle$ is a $<_J$-increasing sequence of
members of ${}^\kappa\Ord$.  

\noindent
Let $S^{\ged}_{\bar f}$, the good set of $S$, be 
$\{\delta < \lambda:\cf(\delta) > \kappa$ and
we can find sequence $\bar A = \langle A_\alpha:\alpha \in u\rangle$
witnessing $\delta$ is a good point of $\bar f\}$ which means:
\mn
\begin{enumerate}
\item[$\bullet$]  $u \subseteq \delta = \sup(u)$
\sn
\item[$\bullet$]  $A_\alpha \in J$ for $\alpha \in u$
\sn
\item[$\bullet$]  if $\alpha < \beta$ are from $u$ and $i \in \kappa
\backslash A_\alpha \backslash A_\beta$ then $f_\alpha(i) < f_\beta(i)$.
\end{enumerate}
\end{definition}

\begin{claim}
\label{2b.111}
$\bar C$ is $(\aleph_\kappa,J)$-free  and even
$(\theta^{+ \kappa},J)$-free \when \,:
\mn
\begin{enumerate}
\item[$(a)$]  $\mu > \cf(\mu) = \kappa$ and 
$\theta \in (\kappa,\mu)$ is regular
\sn
\item[$(b)$]  $\bar\lambda = \langle \lambda_i:i <
\kappa\rangle$ is a sequence of regular cardinals $< \mu$ with
$\lim_J(\bar\lambda) = \mu$
\sn
\item[$(c)$]  $J = J_{\theta * \kappa}$, see Definition \ref{0p.6}(3)
\sn
\item[$(d)$]  $\lambda = \tcf(\prod\limits_{i < \kappa} 
\lambda_i,<_{J^{\bd}_\kappa})$ is exemplified
by $\bar f = \langle f_\alpha:\alpha < \lambda\rangle$
\sn
\item[$(e)$]   $S \subseteq S^\lambda_\theta \cap S^{\ged}_{\bar f}$
 is stationary (on $S^{\ged}_{\bar f}$ see Definition \ref{2b.113} above), 
$\delta \in S \Rightarrow \mu |\delta$
\sn
\item[$(f)$]  $\bar C = \langle C_\delta:\delta \in S \rangle$ is a
strict $\lambda$-ladder system such that $\otp(C_\delta) = \theta$ and
$C_\delta \subseteq \delta = \sup(C_\delta)$
\sn
\item[$(g)$]  if $\delta \in S,\alpha < \kappa$ and 
$i < \kappa$, then the $(\kappa \alpha +i)$-th member of $C_\delta$ is equal to
$f_\delta(i)$ modulo $\mu$.
\end{enumerate}
\end{claim}

\begin{remark}
The proof is similar to in Magidor-Shelah \cite{MgSh:204} where
the assumptions are quite specific.
\end{remark}

Hence we get
\begin{conclusion}
\label{2b.112}
Assume that $\kappa = \cf(\mu) < \mu$ and $\lambda = \cf(\lambda)
 =^+ \pp_{J^{\bd}_\kappa}(\mu)$.

\Then \, there is a $(\kappa^{+ \kappa +1},J_{\kappa^+ \times
\kappa})$-free strict ladder system $\langle \eta_\delta:\delta \in
S\rangle$ for some stationary $S \subseteq S^\lambda_{\kappa^+}$.
\end{conclusion}

\begin{remark}  This statement is used in the proof of Theorem \ref{h.9d}.
\end{remark}

\begin{PROOF}{\ref{2b.112}}
We shall apply \ref{2b.111}.  As we are assuming
$\pp_{J^{\bd}_\kappa}(\mu) =^+ \lambda = \cf(\lambda)$ there is a
sequence $\bar\lambda = \langle \lambda_i:i < \kappa\rangle$ of
regular cardinals $< \mu$ such that $\mu = \lim_J(\bar\lambda)$ and
$\lambda = \tcf(\prod\limits_{i < \kappa}
\lambda_i,<_{J^{\bd}_\kappa})$ and let $\bar f = \langle
f_\alpha:\alpha < \lambda\rangle$ exemplify it; \wilog \,
$\bar\lambda$ is increasing.

Now $\lambda$ is regular $> \mu > \kappa^{++}$ hence by \cite{Sh:420}
there is a stationary $S \subseteq S^\lambda_{\kappa^+}$ which is from
$\check I_\kappa[\lambda]$ hence by \cite{Sh:506} \wilog \, $S
\subseteq S^{\ged}_{\bar f}$.

As $S \in \check I_{\kappa^+}[\lambda]$ there is a strict club system
$\bar C = \langle C_\delta:\delta \in S\rangle$.  Easily \wilog \,
$\bar C$ satisfies clause (g) of \ref{2b.111}.  Hence by
\ref{2b.111}, $\bar C$ is as required.
\end{PROOF}

\noindent
Recall the following (see \cite[ChII]{Sh:g}, more in \cite{MgSh:204}).
Proving \ref{2b.111} we in fact use
\begin{claim}
\label{2b.117}
If $\circledast$ below holds \then \, we can find a $\theta$-free,
$(\lambda,\kappa)$-ladder system $\bar C' = \langle C'_\delta:\delta \in
S\rangle$ such that $(\forall \alpha \in C'_\delta)(\exists! \beta \in
C_\delta)(\alpha + \mu = \beta + \mu)$.  Moreover there is $\langle
f_\delta:\delta \in S\rangle \in {}^S \cF$ without repetitions such that
$C'_\delta \subseteq \{\beta + i:\beta \in C_\delta,i < \mu$ and $(\exists
\alpha,j)(\mu|\alpha \wedge j < \mu \wedge \beta = \alpha + j \wedge
\beta + i = \alpha + \cd(\otp(C_\delta \cap \alpha),i,
f_\delta(\otp(C_\delta \cap \alpha))\}$, \when \,
\mn
\begin{enumerate}
\item[$\circledast$]  $(a) \quad S \subseteq \lambda$ is stationary
and $\delta \in S$ implies $\mu \backslash \delta$ or even $\mu \cdot
\delta = \mu$
\sn
\item[${{}}$]  $(b) \quad \bar C = \langle C_\delta:\delta \in S\rangle$ is
a $(\lambda,\kappa)$-ladder system
\sn
\item[${{}}$]  $(c) \quad \mu < \lambda$ and ${\mat F} \subseteq {}^\kappa
\mu$ has cardinality $\ge \lambda$ and is $\theta$-free
\sn
\item[${{}}$]  $(d) \quad \cd:\kappa \times \mu \times \mu
\rightarrow \mu$ is one-to-one.
\end{enumerate}
\end{claim}

\begin{PROOF}{\ref{2b.117}}
Straightforward. 
\end{PROOF}

\begin{remark}
\label{2b.16}
The problem in proving the conjecture TDU$_{\aleph_\omega}$ is
to have $(D \ell)_S$ assuming $\lambda =
\lambda^{< \lambda}$; this would have solved the problem in \S0.  As in
many cases here, this is very persuasive but we do not know how to prove
this  in full generality.
\end{remark}

\noindent
The following will be useful showing that if ($R$ a suitable ring), 
$\SP_{\lambda,\theta}(R)$, see Definition \ref{5e.7}, contains
enough ideals (say $J^{\bd}_\kappa,J^{\bd}_{\kappa^+ \times \kappa}, 
J^{\bd}_{\kappa^{++} \times \kappa^+})$ then $\TDU_{\kappa^{+ \omega}}(R)$;
$\bbZ$ ``just" miss this criterion; see also \ref{h16}
\begin{theorem}
\label{h.9d}
For $\mu \in \bold C_\kappa$ one of the following holds:
\mn
\begin{enumerate}
\item[$(A)$]  {\rm BB}$(2^\mu,\mu^+,< \mu,\kappa)$
\sn
\item[$(B)$]  {\rm BB}$(\lambda,\mu^+,< \mu,\kappa)$ where $\lambda =
\min\{\chi:2^\mu < 2^\chi\}$
\sn
\item[$(C)$]  $\lambda := 2^\mu$ satisfies $\lambda = \lambda^{<
\lambda}$ and {\rm BB}$(\lambda,\kappa^{+ \omega +1},< \mu,J_{\kappa^+
\times \kappa})$
\sn
\item[$(D)$]  $\lambda := 2^\mu$ satisfies $\lambda = \lambda^{<
\lambda}$ and {\rm BB}$(\lambda,\kappa^{+ \omega +1},< \mu,
J_{\kappa^{++} \times \kappa^+})$ and also
\sn
\item[${{}}$]  $\bullet_1 \quad$ there is 
$\chi \in (\mu,\lambda)$ such that $\cf(\chi) =
\kappa^+$ and $\chi^{<\kappa^+>_{\tr}} =^+ \lambda$
\sn
\item[${{}}$]  $\bullet_2 \quad \cF \subseteq 
{}^{(\kappa^+)}\chi,|\cF| = \lambda
\Rightarrow (\kappa^+,\kappa^{++}) \in \text{\rm issp}(\cF)$.
\end{enumerate}
\end{theorem}

\begin{PROOF}{\ref{h.9d}} 
First, if Theorem \ref{h.7} case (A) or case (C) holds then case (A)
or case (B) respectively here holds too, so we can 
assume case (B) of \ref{h.7} holds
and in particular $\lambda := 2^\mu$ satisfies $\lambda =
\lambda^{< \lambda}$ and $\alpha < \lambda \wedge \sigma < \kappa
\Rightarrow |\alpha|^\sigma < \lambda$.

Second, assume there is no $\chi \in
(\mu,\lambda)$ such that $\lambda =^+ \chi^{<\kappa^+>_{\tr}}$
then by \ref{1.3.23}(1) we have $(D \ell)_S$ for every stationary $S \subseteq
S^\lambda_{\kappa^+}$, and then by \ref{2b.112}, we can find
stationary $S \subseteq S^\lambda_{\kappa^+}$ and (see \ref{0p.35}(4))
a strict $(\lambda,\kappa^+)$-ladder system
$\langle \eta_\delta:\delta \in S\rangle$ which is $(\kappa^{+ \omega
+1},J_{\kappa^+  \times \kappa})$-free hence by \ref{1.3.25} we have
BB$(\lambda,\kappa^{+ \omega +1},< \mu,J_{\kappa^+ \times \kappa})$
so clause (C) of the theorem holds.  

Third, assume that there is $\chi_1 < \lambda$ such that $\lambda =^+
(\chi_1)^{<\kappa^+>_{\tr}}$ and there is $\cF \subseteq
{}^{(\kappa^+)}\mu$ of cardinality $\lambda$ which is
$\kappa^{++}$-free or just such that $(\kappa^+,\kappa^{++}) \notin
\issp(\cF)$ \then \, by clause (i) of Claim \ref{1f.10} 
clause (A) of the theorem holds.

Fourth, assume that for $\ell=1,2$ for some $\chi_\ell < \lambda$ we
have $(\chi_\ell)^{<\kappa^{+ \ell}>_{\tr}} =^+ 2^\lambda$ so \wilog \,
$\pp_{J^{\bd}_{\kappa^{+ \ell}}}(\chi_\ell) =^+ 2^\lambda$; so the first
assumption of ``third" hold and its second (by \S3) hence clause (C)
of the theorem holds.

So we can assume that none of the above apply, and we shall prove
clause (D), first $\bullet_1 - \bullet_2$.  By ``second" above \wilog \,
we can choose $\chi_1 \in (\mu,\lambda)$ such that
$(\chi_1)^{<\kappa^+>_{\tr}} =^+ \lambda$ and \wilog \, $\cf(\chi_1) =
\kappa^+,\pp_{J^{\bd}_{\kappa^+}}(\chi_1) =^+ \lambda$ (by \cite{Sh:589}), so
$\bullet_1$ holds. 

By ``third" \wilog \,there is no $\cF \subseteq {}^{(\kappa^+)}\mu$ of
cardinality $\lambda$ such that $(\kappa^+,\kappa^{++}) \notin
\issp(\cF)$, hence $\bullet_2$ holds.

Now by ``fourth" we can assume that there is no $\chi_2 
\in (\mu,\lambda)$ such that $\lambda =^+
\chi^{<\kappa^{++}>_{\tr}}_2$, hence by
\ref{1.3.23}(1) for every stationary 
$S \subseteq S^\lambda_{\kappa^{++}}$ we have $(D \ell)_S$.  
Again we apply \ref{2b.112} with $\chi_2$ here for $\mu$ there 
and we can find a stationary set $S \subseteq
S^\lambda_{\kappa^{++}}$ and a strict ladder system $\langle \eta_\delta:\delta
\in S\rangle$ which is $(\kappa^{+ \omega +1},
J_{\kappa^{++} \times \kappa^+})$-free, hence by \ref{1.3.25} we have
BB$(\lambda,\kappa^{+ \omega +1},< \mu,J_{\kappa^{++} \times \kappa^+})$, so
clause (D) of the theorem holds.  So we are done.
\end{PROOF}

\begin{claim}
\label{2b.107}
Assume $\chi < \chi^+ \le \lambda = \cf(\lambda)$ and
$\alpha < \lambda \Rightarrow \cf([\alpha]^{\le \chi},\subseteq) < \lambda$.

\noindent
1) If $2^\sigma < \lambda,\sigma = \cf(\sigma) \le \chi$ 
and $\lambda = \lambda^{< \lambda}$, \then \,
$(D \ell)^*_{S^\lambda_\sigma}$.

\noindent
2) We can find $\bar{\mat P} = \langle {\mat P}_\alpha:\alpha <
   \lambda\rangle$ such that:
\mn
\begin{enumerate}
\item[$(a)$]  ${\mat P}_\alpha \subseteq {\mat P}(\alpha)$,
\sn
\item[$(b)$]  $|{\mat P}_\alpha| < \lambda$,
\sn
\item[$(c)$]  if $u \in{\mat P}_\alpha$, then $|u| \le \chi$ and $u$
  is a closed subset of $\alpha$,
\sn
\item[$(d)$]  if $\alpha \in {\mat P}_\alpha$ and $\beta 
\in u$, then $u \cap \beta \in {\mat P}_\beta$,
\sn
\item[$(e)$]  if $\delta < \lambda,\aleph_0 < \cf(\delta) \le \chi$
  \then \, $\delta = \sup(u)$ for some $u \in {\mat P}_\delta$.
\end{enumerate}
\end{claim}

\begin{PROOF}{\ref{2b.107}}
1) By \cite{Sh:829}.

\noindent
2)  See D\v{z}amonja-Shelah \cite{DjSh:562}.
\end{PROOF}

\begin{observation}
\label{h.11}
1) Assume
\mn
\begin{enumerate}
\item[$(A)$]  $\lambda = \chi^+,\chi= \text{ cf}(\chi) \ge \mu$
\underline{or} 
\sn
\item[$(B)$]  $\lambda = \chi^+ > \mu^+$, cf$([\chi]^{\le \mu},
\subseteq) = \chi$, see \ref{0p.17}(2).
\end{enumerate}
\mn
\Then \, we can find $\langle \bar e_\varepsilon:\varepsilon <
\chi\rangle$ such that:
\mn
\begin{enumerate}
\item[$(a)$]  $\bar e_\varepsilon = \langle
e_{\varepsilon,\alpha}:\alpha < \lambda\rangle$
\sn
\item[$(b)$]  $e_{\varepsilon,\alpha} \subseteq \alpha$ is closed
\sn
\item[$(c)$]  $\sup\{\otp(e_{\varepsilon,\alpha}):\alpha < \lambda\} <
\mu$ for each $\varepsilon < \chi$
\sn
\item[$(d)$]  if $\alpha \in e_{\varepsilon,\beta}$ then 
$e_{\varepsilon,\alpha} = e_{\varepsilon,\beta} \cap \alpha$
\sn
\item[$(e)$]  if $\alpha < \lambda \wedge \cf(\alpha) < \mu$
then for some $\varepsilon < \chi$ the set $e_{\varepsilon,\alpha}$
contains a club of $\alpha$
\sn
\item[$(f)$]   for every 
$\alpha < \lambda$ and $u \in [\alpha]^{< \mu}$ for some
$\varepsilon < \chi$ we have $u \subseteq e_{\varepsilon,\alpha}$.
\end{enumerate}
\end{observation}

\begin{remark}  Used in \ref{2b.91}.
\end{remark}

\begin{PROOF}{\ref{h.11}}
First assume clause (A) holds.
By \cite[\S4]{Sh:351} or \cite[3.7]{Sh:309} there is a sequence $\langle \bar
e_\varepsilon:\varepsilon < \chi\rangle$ satisfying clauses
(a),(b),(d) and
\mn
\begin{enumerate}
\item[$(c)'$]   $e_{\varepsilon,\alpha}$ has cardinality $< \chi$
\sn
\item[$(e)$]   if $u \subseteq \alpha < \lambda$ has cardinality $<
\chi$ then $u \subseteq e_{\varepsilon,\alpha}$ for some $\varepsilon$
\sn
\item[$(f)'$]  $\langle e_{\varepsilon,\alpha}:\varepsilon <
\chi\rangle$ is $\subseteq$-increasing.
\end{enumerate}
\mn
Manipulating those $\bar e_\varepsilon$'s we get the desired
conclusion (e.g. ignoring clause (f) choose $\langle e_\delta:\delta
< \mu$ limit$\rangle$, $e_\delta$ a club of $\delta$ of order type
cf$(\delta)$ and for $\varepsilon < \chi \wedge \delta < \mu$ we
define $\bar e^\delta_\varepsilon = \langle
e^\delta_{\varepsilon,\alpha}:\alpha < \lambda\rangle$ by
$e^\delta_{\varepsilon,\alpha} := \{\gamma \in
e_{\varepsilon,\alpha}:\otp(\gamma \cap e_{\varepsilon,\alpha})
\in e_\delta\}$, now check).

Second, assume clause (B).  The proof is similar using \ref{2b.107},
i.e. Dzamonja-Shelah \cite{DjSh:562}.
\end{PROOF}

\begin{claim}
\label{h16}
We have $\BB(2^\mu,\kappa^{+ \omega+1},\theta,
J^{\bd}_{\kappa^+ \times \kappa})$ if
$\theta < \mu \in {\bold C}_\kappa$ and $(\forall \chi)(\chi < 2^\mu
\Rightarrow \chi^{<\kappa^+>_{\text{tr}}} < 2^\mu)$.
\end{claim}

\begin{PROOF}{\ref{h16}}
See in the proof of \ref{h.9d}, ``second...". That is, by \ref{2b.112}
there is a $(\kappa^{+ \kappa +1},J_{\kappa^+ \times \kappa})$-free
ladder system $\langle C_\delta:\delta \in S\rangle,S \subseteq
S^\lambda_{\kappa^+}$ stationary.

We claim that $\bar C$ exemplifies $\BB(\lambda,\kappa^{+ \omega +1},<
\lambda,J^{\bd}_{\kappa^+ \times \kappa})$.  Recalling the assumption
$\chi < 2^\mu \Rightarrow \chi^{<\kappa^+>_{\tr}} < 2^\mu$ by Claim
\ref{1.3.23} we have $(D\ell)_{S_1}$ 
for every stationary $S_1 \subseteq S$,  hence
  by \ref{1.3.25} we have clause (B) of Definition \ref{0p.15}.
\end{PROOF}

\noindent
Note (will be useful together with \ref{h.9d}, \ref{5e.8},
\ref{1f.53}).
\begin{observation}
\label{h18}
If (A) then (B) where:
\mn
\begin{enumerate}
\item[$(A)$]  $(a) \quad J_\ell$ is an ideal on $\kappa_\ell$ for
  $\ell=1,2$ and $\kappa_1 = \kappa_2 \wedge J_1 \subseteq J_2$ 

\hskip25pt or $J_1 \le_{\RK} J_2$ or just for some function $h$ from $\kappa_2$
  onto $\kappa_1$ we have 

\hskip25pt $(\forall A \in J_1)(\{\beta <
  \kappa_2:h(\beta) \in A\} \in J_1)$
\sn
\item[${{}}$]  $(b) \quad \bar c_\ell = \langle c^\ell_\alpha:\alpha
  \in S_\ell\rangle,\otp(c^1_\alpha) = \kappa_1$
\sn
\item[${{}}$]  $(c) \quad S_2 = \{\kappa_2 \cdot \delta:\delta \in
  S_1\}$ and for $\delta \in S_1$ we have $C^2_{\kappa_2 \cdot \delta} =$

\hskip25pt $\{\kappa_2 \cdot \beta + \otp(C^1_\delta \cap \alpha):\alpha \in
  C^1_\delta$ and $\beta = h(\alpha)\}$
\sn
\item[$(B)$]  $(a) \quad$ if $\bar c_1$ is $(\mu,J_1)$-free \then \,
  $\bar c_2$ is $(\mu,J_2)$-free
\sn
\item[${{}}$]  $(b) \quad$ if $BB(\lambda,\mu,\theta,J_1)$ and $\theta
 = \theta^{\kappa_2}$ then $BB(\lambda,\mu,\theta,J_2)$.
\end{enumerate}
\end{observation}

\begin{PROOF}{\ref{h18}}
Straightforward.
\end{PROOF}
\newpage

\section{Cases of weak G.C.H.} \label{Cases}

Note that if $\mu \in {\bold C}_\kappa$ and $\lambda < 2^\mu <
2^\lambda$,  then we can find a $\mu^+$-free ${\mat F} \subseteq 
{}^\kappa \mu$ of cardinality $\lambda$ (by the ``No hole Conclusion", 
\cite[Ch.II,2.3 pg.53]{Sh:g} or here \ref{h.11}(3)) 
so by the Section Main Claim \ref{d.6} we 
can deduce BB$(\lambda,\mu^+,(2^\mu,\theta),\kappa)$ for $\theta 
< \mu$ - see conclusion \ref{d.11}.

Observe below that if $\theta = 2,\bar C = \langle C_\gamma:\gamma <
\lambda\rangle,C_\gamma \subseteq \mu$ (and $2^\mu < 2^\lambda$),
then easily clause $(\beta)$ of the conclusion of the Section Main Claim 
\ref{d.6} below holds by counting - see \ref{d.6.3}(5).
The point is to prove it for more colors, this is a relative of
\cite[1.10]{Sh:775} but this section is self contained.  
Also Definition \ref{d.8} repeats Definition \cite[1.9]{Sh:775}.

This section is close to \cite[\S1]{Sh:775} hence we try to keep
similar notation.
\bigskip

\begin{definition}
\label{d.8}
1) Sep$(\mu',\mu,\chi,\theta,\Upsilon)$ means that for some $\bar f$:
\mn
\begin{enumerate}
\item[$(a)$]  $\bar f = \langle f_\varepsilon:\varepsilon < \mu' \rangle$
\sn
\item[$(b)$]  $f_\varepsilon$ is a function from
${}^\mu \chi$ to $\theta$
\sn
\item[$(c)$]  for every $\varrho \in {}^{\mu'} \theta$ the set $\{\nu \in
{}^\mu \chi$ : for every $\varepsilon < \mu'$ we have $f_\varepsilon(\nu)
\ne \varrho(\varepsilon)\}$ has cardinality $< \Upsilon$.
\end{enumerate}
\mn
2) We may omit $\chi$ if $\chi = \theta$.  We write
$\Sep(\mu,\theta,\Upsilon)$ for Sep$(\mu,\mu,\theta,\theta,\Upsilon)$ and
$\Sep(\mu,\theta)$ means that for some $\Upsilon = \cf(\Upsilon) \le
2^\mu$ we have $\Sep(\mu,\mu,\theta,\theta,\Upsilon)$ and $\Sep(<
\mu,\theta)$ if for some $\Upsilon = \cf(\Upsilon) \le 2^\mu$
and some $\sigma < \mu$ we have $\Sep(\sigma,\mu,\theta,\theta,\Upsilon)$.
Let $\Sep^+(\mu,\theta)$ mean $\Sep(\mu,\mu,\theta,\theta,\mu)$.
\end{definition}

\begin{smain claim}
\label{d.6}
Assume
\mn
\begin{enumerate}
\item[$(a)$]   $2^\mu < 2^\lambda$
\sn
\item[$(b)$]   $D$ is a $\mu^+$-complete filter on $\lambda$
extending the co-bounded filter
\sn
\item[$(c)$]   $\bar C = \langle C_\gamma:\gamma < \lambda \rangle,
C_\gamma \subseteq \mu$,
\sn
\item[$(d)$]   $2 \le \theta \le \mu$ and $\Upsilon \le \mu$ (or
just $D$ is $\Upsilon^+$-complete, $\Upsilon \le 2^\mu$)
\sn
\item[$(e)$]   $\Sep(\mu,\theta,\Upsilon)$
\sn
\item[$(f)$]   $\lambda = \min\{\partial:2^\partial > 2^\mu\}$ or at least
\sn
\item[$(f)^-$]   we have $h_\xi \in {}^\lambda(2^\mu)$ for $\xi <
(2^\mu)^+$ such that $\zeta \ne \xi \Rightarrow h_\zeta \ne_D h_\xi$.
\end{enumerate}
\mn
\Then \,
\mn
\begin{enumerate}
\item[$(\alpha)$]  if $\chi$ satisfies $\gamma < \lambda \Rightarrow
\chi^{|C_\gamma|} \le \theta$, \then \, we can find $\bar f = \langle
f_\gamma:\gamma < \lambda\rangle$ satisfying $f_\gamma \in
{}^{(C_\gamma)}\chi$ such that (see \ref{d.6.3}(1)):

for every $f:\mu \rightarrow \chi$, for some $\gamma < \lambda,f_\gamma
\subseteq f$ (and even for $D^+$-many 

$\gamma < \lambda$)
\sn
\item[$(\beta)$]  if ${\bold F}_\gamma:{}^{(C_\gamma)}(2^\mu)
\rightarrow \theta$ for $\gamma < \lambda$, \then \, we can find
$\bar c = \langle c_\gamma:\gamma < \lambda\rangle \in
{}^\lambda\theta$ such that:
\sn
\begin{enumerate}
\item[$(*)$]   for any mapping $f:\mu \rightarrow
2^\mu$, for some $\gamma < \lambda,{\bold F}_\gamma(f \restriction
C_\gamma) = c_\gamma$ (even for $D^+$-many $\gamma < \lambda$)
\end{enumerate}
\sn
\item[$(\gamma)$]  if $\bar \chi = \langle
\chi_\varepsilon:\varepsilon < \mu\rangle$ satisfies $\gamma < \lambda
\Rightarrow \prod\limits_{\varepsilon \in C_\gamma} \chi_\varepsilon
\le \theta$, \then \, we can find $\bar f = \langle f_\gamma:\gamma <
\lambda\rangle$ satisfying $f_\gamma \in \prod\limits_{\varepsilon \in
  C_\gamma} \chi_\varepsilon$ such that for every $f \in
\prod\limits_{\varepsilon < \mu} \chi_\varepsilon$, for some $\gamma <
\lambda,f_\gamma = f \restriction C_\gamma$ (and even for $D^+$-many 
$\gamma$'s).
\end{enumerate}
\end{smain claim}

\begin{remark}
\label{d.6.3}
1) Of course ``for $D^+$ many $t \in I$ we have $xx$" means that
$D$ is a filter on $I$ and $\{t \in I:t$ satisfies $xx\} \in D^+$, see below.

\noindent
2) For $D$ a filter on $I$ let Dom$(D) = I$ and let $D^+ = \{A \subseteq I:I
 \backslash A \notin D\}$.

\noindent
3) Similarly for $J$ an ideal on $I$.

\noindent
4) Note that in \ref{d.6}, clause (f) implies clause (a) and even
clause (f)$^-$ does.  Note that clause (f) implies $\lambda$ is regular
(but not (f)$^-$) and clause (b) implies cf$(\lambda) > \mu$.

\noindent
5) Concerning clause $(\beta)$ in \ref{d.6}, when $\theta=2$,
this is easy:
let $D$ be the filter of co-bounded subsets of $\lambda$, and let $\langle
f_\alpha:\alpha < 2^\mu\rangle$ list ${}^\mu(2^\mu)$, each appearing
$\lambda$ times.  Now $\cF := \{\langle 1 - {\bold F}_\gamma(f_\alpha
\restriction C_\gamma):\gamma < \lambda\rangle:\alpha < 2^\mu\}$ is a
subset of ${}^\lambda 2$ of cardinality $2^\mu < 2^\lambda
=|{}^\lambda 2|$.   So every sequence $\bar c \in {}^\lambda 2
\backslash \cF$ is as required.  Concerning this proof we 
can use any filter $D$ on $\lambda$
such that $|2^\lambda/D| > 2^\mu$,

\noindent
6)  In the Section Main Claim \ref{d.6} we can replace $\mu$ by any set 
of cardinality $\mu$.  E.g., ${}^{\omega >} \mu$.  Hence replacing
$\bar C$ by $\bar C' = \langle C'_\alpha:\alpha <
\lambda\rangle,C'_\alpha = {}^{\omega >}(C_\alpha)$ in clause 
$(\beta)$ of \ref{d.6}
we can assume $\Dom({\bold F}_\gamma) = \{f:f$ a function from
${}^{\omega >}(C_\alpha)$ to $2^\mu\}$.

\noindent
7) We may wonder if clause (e) of the assumption of the 
Section Main Claim \ref{d.6} is
reasonable; the following Claim \ref{d.7} gives some sufficient
conditions for clause (e) of \ref{d.6} to hold.

\noindent
8) In \ref{d.6} we implicitly assert that $(f) \Rightarrow (f)^-$; for
completeness we recall the justification (as there $(2^\mu)^+ \le 2^\lambda$).
\end{remark}

\begin{obs}
\label{d.6.5}
We have $(f) \Rightarrow (f)^-$ in \ref{d.6}, i.e. if $\lambda = 
\min\{\partial:2^\partial > 2^\mu\}$ \then \, there are $h_\xi:\lambda
\rightarrow 2^\mu$ for $\xi < 2^\lambda$ such that $\xi < \zeta <
2^\lambda \Rightarrow h_\xi \ne h_\zeta \mod J^{\bd}_\lambda$.
\end{obs}

\begin{PROOF}{\ref{d.6.5}}
As $\alpha < \lambda \Rightarrow |{}^\alpha 2| =
2^{|\alpha|} \le 2^\mu$ and $\mu \le \lambda \le 2^\mu$ clearly ${}^{\lambda
>}2 = \cup\{{}^\alpha 2:\alpha < \lambda\}$ has cardinality $2^\mu$,
so there is a one-to-one function $\bold g$ from ${}^{\lambda >}2$
onto $2^\mu$.

Let $\langle \eta_\xi:\xi < 2^\lambda\rangle$ list ${}^\lambda 2$ and
let $h_\xi:\lambda \rightarrow 2^\mu$ be defined by $h_\xi(\alpha) =
\bold g(\eta_\xi \restriction \alpha)$ for $\alpha < \lambda$.

Clearly $\langle h_\xi:\xi < 2^\lambda\rangle$ is as required.
\end{PROOF}

\noindent
In order to give a sufficient condition for clause (e) of \ref{d.6} we
recall
\begin{definition}
\label{d9}
1) For $J$ an ideal on $\sigma$ and cardinal $\mu$ let 
${\bold U}_J(\mu) = \min\{|{\mat P}|:{\mat P} \subseteq 
[\mu]^{\le \sigma}$ and for every $f \in {}^\sigma \mu$, for some $u \in
{\mat P}$, we have $\{\varepsilon < \sigma:f(\varepsilon) \in u\}
\ne \emptyset \mod J\}$.

\noindent
2) If $J = J^{\bd}_\sigma$ and $\sigma$ is a regular cardinal, we may 
write $\bold U_\sigma(\mu)$.
\end{definition}

\begin{claim}
\label{d.7}
Clause {\rm (e)} of \ref{d.6} holds, i.e., $\Sep(\mu,\theta,\Upsilon)$
holds, \when \, $\theta \ge \aleph_0$ and\footnote{On the case
  ``$\theta$ finite", see \ref{d20}.} at least one of the following holds:
\mn
\begin{enumerate}
\item[$(a)$]  $\mu = \mu^\theta$ and $\Upsilon = \theta$
\sn
\item[$(b)$]   ${\bold U}_\theta(\mu) = \mu$ and $2^\theta < \mu$
 and $\Upsilon = (2^\theta)^+$
\sn
\item[$(c)$]   ${\bold U}_J(\mu) = \mu$ where for some $\sigma$ we
have $J = [\sigma]^{< \theta},\theta \le \sigma,\sigma^\theta \le \mu$
and $\theta^{< \sigma} < \mu$ and $\Upsilon = (\theta^{< \sigma})^+$
\sn
\item[$(d)$]   $\mu$ is strong limit of cofinality $\ne
\theta,\theta < \mu$ and $\Upsilon = (2^\theta)^+$
\sn
\item[$(e)$]  $\mu \ge \beth_\omega(\theta)$ and $\Upsilon = \mu$.
\end{enumerate}
\end{claim}

\begin{PROOF}{\ref{d.7}}   
By the proof of \cite[1.11]{Sh:775}, (not the
statement!); however, for completeness, below we 
shall give a complete proof (after the proofs of \ref{d.6}, \ref{d.11} and
\ref{d.11p}). 
We shall use mainly \ref{d.7} clause (d).

\noindent
Proof of the Section Main Claim \ref{d.6}:

It is enough to prove clause $(\beta)$,
as it implies the others.  Why?  Clearly clause $(\alpha)$ is a
special case of clause $(\gamma)$ and for clause $(\gamma)$ note that
\wilog \, $(\forall \varepsilon)(\chi_\varepsilon \le \theta)$ hence
$(\forall \varepsilon)(\chi_\varepsilon \le 2^\mu)$ so we can choose
$\bold F_\gamma$ as any function from ${}^{(C_\gamma)}(2^\mu)$ 
onto $\theta$ such that:
\mn
\begin{enumerate}
\item[$\bullet$]  $\bold F_\gamma \rest \prod\limits_{\varepsilon \in
  C_\gamma} \chi_\varepsilon$ is a one to one function.
\end{enumerate}
\mn
Now by clause $(\beta)$ we can find $\langle c_\gamma:\gamma <
\lambda\rangle$ such that $(*)$ there holds and for $\gamma < \lambda$
let $f_\gamma$ be the unique $f \in \prod\limits_{\varepsilon \in
  C_\gamma} \chi_\varepsilon$ such that $\bold F_\gamma(f) = c_\gamma$
and $f_\gamma$ constantly zero if there is no such $f$.

Now check; so indeed is sufficient to prove clause $(\beta)$.

Let $\langle {\bold F}_\gamma:\gamma < \lambda\rangle$ be as in clause
$(\beta)$ and we shall prove that there is
$\langle c_\gamma:\gamma < \lambda\rangle$ as promised therein.

By assumption (e) we have Sep$(\mu,\theta,\Upsilon)$
which means (see Definition \ref{d.8}(2)) that
we have $\Sep(\mu,\mu,\theta,\theta,\Upsilon)$.

Let $\bar f = \langle f_\varepsilon:\varepsilon < \mu\rangle$
exemplify Sep$(\mu,\mu,\theta,\theta,\Upsilon)$, see Definition \ref{d.8}(1)
and
\mn
\begin{enumerate}
\item[$(*)_0$]   for $\varrho \in {}^\mu \theta$ let
Sol$_\varrho := \{\nu \in {}^\mu \theta:\text{ for every } \varepsilon
< \mu$ we have $\varrho(\varepsilon) \ne f_\varepsilon(\nu)\}$
\end{enumerate}
\mn
where Sol stands for solutions, so by clause (c) of the Definition
\ref{d.8}(1) of Sep it follows that:
\mn
\begin{enumerate}
\item[$(*)_1$]  $\varrho \in {}^\mu \theta
\Rightarrow |\Sol_\varrho| < \Upsilon$.
\end{enumerate}
Let cd be a one-to-one function from ${}^\mu (2^\mu)$ onto $2^\mu$
such that (this is possible as $\cf(2^\mu) > \mu)$:

\[
\alpha = \cd(\langle \alpha_\varepsilon:\varepsilon < \mu \rangle) 
\Rightarrow \alpha \ge \sup\{\alpha_\varepsilon:\varepsilon < \mu\}
\]

\mn
Let $\cd_\varepsilon:2^\mu \rightarrow 2^\mu$ for $\varepsilon < \mu$
be such that $\alpha < 2^\mu \Rightarrow \alpha = \cd(\langle
\cd_\varepsilon(\alpha):\varepsilon < \mu\rangle)$.

Let ${\bold H}$ be a one-to-one function from $2^\mu$ onto ${}^\mu
\theta$, such ${\bold H}$ exists as $2 \le \theta \le \mu$ 
by clause (d) of the assumption.  For
$\varrho \in {}^\mu \theta$ let $\Sol'_\varrho := \{\alpha <
2^\mu:{\bold H}(\alpha) \in \Sol_\varrho\}$, so
\mn
\begin{enumerate}
\item[$(*)_2$] $\varrho \in {}^\mu \theta \Rightarrow
|\text{Sol}'_\varrho| < \Upsilon$.
\end{enumerate}
\mn
Clearly in the assumption, if clause $(f)$ holds, then clause $(f)^-$
holds (see \ref{d.6.5}), so we can assume that $\langle
h_\xi:\xi < (2^\mu)^+\rangle$ are as in clause $(f)^-$ so in
particular $h_\xi \in {}^\lambda(2^\mu)$.

Fix $\xi < (2^\mu)^+$ for a while.

For $\gamma < \lambda$ let
\mn
\begin{enumerate}
\item[$(*)_3$]   $\varrho^*_{\xi,\gamma} := {\bold H}
(h_\xi(\gamma)) \in {}^\mu \theta$.
\end{enumerate}
\mn
Let $\varepsilon < \mu$.  Recall that $\varrho^*_{\xi,\gamma} \in
{}^\mu \theta$ for $\gamma < \lambda$ and $f_\varepsilon$
 is a function from ${}^\mu \theta$ to $\theta$ so 
$f_\varepsilon(\varrho^*_{\xi,\gamma}) < \theta$.  Hence
we can consider the sequence $\bar c^\xi_\varepsilon = \langle
f_\varepsilon(\varrho^*_{\xi,\gamma}):\gamma < \lambda\rangle \in
{}^\lambda \theta$ as a candidate for being as required (on $\langle
c_\gamma:\gamma < \lambda\rangle$) in the desired
conclusion $(*)$ from clause $(\beta)$ of the Section Main Claim \ref{d.6}.
If one of them is as required, we are done. So assume towards a
contradiction that for each $\varepsilon< \mu$ (recall we are fixing $\xi <
(2^\mu)^+$) there is a sequence
$\eta^\xi_\varepsilon \in {}^\mu (2^\mu)$ that exemplifies
the failure of $\bar c^\xi_\varepsilon$ to satisfy $(*)$,
hence there is a set $E^\xi_\varepsilon \in D$, so necessarily a
subset of $\lambda$, such that
\mn
\begin{enumerate}
\item[$(*)_4$]   $\gamma \in E^\xi_\varepsilon
\Rightarrow {\bold F}_\gamma(\eta^\xi_\varepsilon \restriction C_\gamma) \ne
f_\varepsilon(\varrho^*_{\xi,\gamma})$.
\end{enumerate}
\mn
Define $\eta^*_\xi \in {}^\mu(2^\mu)$ by
\mn
\begin{enumerate}
\item[$\boxtimes_1$]   $\eta^*_\xi(\alpha) = \cd
(\langle \eta^\xi_\varepsilon(\alpha):\varepsilon < \mu \rangle)$ for
$\alpha < \mu$; so
$\eta^*_\xi \in {}^\mu(2^\mu)$ for our $\xi < (2^\mu)^+$.
\end{enumerate}
\mn
By clause (b) in the assumption of our Section Main Claim \ref{d.6}, the filter
$D$ is $\mu^+$-complete hence
\mn
\begin{enumerate}
\item[$(*)_5$]   $E^*_\xi := \cap \{E^\xi_\varepsilon:
\varepsilon < \mu\}$ belongs to $D$.
\end{enumerate}
\mn
Now we vary $\xi < (2^\mu)^+$.  For each such $\xi$ we have chosen
$\eta^*_\xi \in {}^\mu(2^\mu)$, and clearly the number of such $\eta^*_\xi$'s
is $\le |{}^\mu(2^\mu)| = (2^\mu)^\mu = 2^\mu$ hence for some $\eta^*$
and unbounded ${\mat U} \subseteq (2^\mu)^+$ we have $\xi \in {\mat U}
\Rightarrow \eta^*_\xi = \eta^*$.

For $\varepsilon < \mu$ we define $\eta'_\varepsilon \in
{}^\mu(2^\mu)$ by $\eta'_\varepsilon(\alpha) = 
\cd_\varepsilon(\eta^*(\alpha))$ for $\alpha < \mu$.  

So by the choice of $\eta^*_\xi$ in $\boxtimes_1$ above:
\mn
\begin{enumerate}
\item[$\boxtimes_2$]  if $\xi \in {\mat U}$, then
$\varepsilon < \mu \Rightarrow \eta^\xi_\varepsilon = \eta'_\varepsilon$.
\end{enumerate}
\mn
So by $(*)_4 + (*)_5$
\mn
\begin{enumerate}
\item[$\boxtimes_3$]  if $\gamma \in E^*_\xi$ where $\xi \in {\mat U}$
 \then \, $\varepsilon < \mu \Rightarrow {\bold F}_\gamma(\eta'_\varepsilon
\restriction C_\gamma) \ne f_\varepsilon(\varrho^*_{\xi,\gamma})$.
\end{enumerate}
\mn
So noting $\langle {\bold F}_\gamma(\eta'_\varepsilon \restriction
C_\gamma):\varepsilon < \mu\rangle \in {}^\mu \theta$, clearly
by $(*)_0$ and $\boxtimes_3$ we have:
\mn
\begin{enumerate}
\item[$\boxtimes_4$]  if $\gamma \in E^*_\xi$ where $\xi \in {\mat U}$, then
$\varrho^*_{\xi,\gamma} \in \Sol_{\langle {\bold F}_\gamma
(\eta'_\varepsilon \restriction C_\gamma):\varepsilon < \mu\rangle}$.
\end{enumerate}
\mn
As $\xi$ was any member of ${\mat U}$, by the choice of
$\varrho^*_{\xi,\gamma}$, i.e.
$(*)_3$ which says that $\varrho^*_{\xi,\gamma} = {\bold
H}(h_\xi(\gamma))$ and the definition of $\Sol'$ (just before
$(*)_2$), we have:
\mn
\begin{enumerate}
\item[$\boxtimes_5$]  if $\xi \in {\mat U}$, then $\gamma \in E^*_\xi
\Rightarrow h_\xi(\gamma) \in \Sol'_{\langle {\bold F}_\gamma 
(\eta'_\varepsilon \restriction C_\gamma):\varepsilon < \mu\rangle}$.
\end{enumerate}
\mn
Let $\bar \xi = \langle \xi_i:i < \Upsilon\rangle$ be a sequence of
pairwise distinct members of ${\mat U}$, which 
is possible as ${\mat U}$ is an unbounded subset
of $(2^\mu)^+$ and $\Upsilon \le 2^\mu$ (see clause (d) of the
assumption).  As $D$ is $\mu^+$-complete and $\Upsilon \le \mu$ or
just $D$ is $\Upsilon^+$-complete,
also $E^* := \cap\{E^*_{\xi_i}:i < \Upsilon\}$
belongs to $D$.  By the above,

\[
\gamma \in E^* \wedge i < \Upsilon \Rightarrow h_{\xi_i}(\gamma) \in \text{
Sol}'_{\langle {\bold F}_\gamma(\eta'_\varepsilon \restriction
C_\gamma):\varepsilon < \mu\rangle}.
\]

\mn
But by $(*)_2$ we have $|\text{Sol}'_{\langle {\bold F}_\gamma
(\eta'_\varepsilon \restriction C_\gamma):\varepsilon < \mu\rangle}|
< \Upsilon$, hence by $\boxtimes_5$ for each $\gamma \in E^*$ we can
choose $i_\gamma < j_\gamma < \Upsilon$ such
that $h_{\xi_{i_\gamma}}(\gamma) = h_{\xi_{j_\gamma}}(\gamma)$.

As $\Upsilon \le \mu$ and $D$ is $\mu^+$-complete or just $D$ is
 $\Upsilon^+$-complete recalling $E^* \in D$
clearly for some $i < j < \Upsilon$ the set $\{\gamma
\in E^*:i_\gamma = i \wedge j_\gamma = j\}$ is $\ne \emptyset$ mod $D$.
As $i < j$, by the choice of $\bar\xi$ (after $\boxtimes_5$) 
we have $\xi_i \ne \xi_j$ and by
the previous sentences $\{\gamma \in E^*:h_{\xi_i}(\gamma) =
h_{\xi_j}(\gamma)\} \ne \emptyset$ mod $D$.
But this contradicts the choice of $\langle h_\zeta:\zeta <
(2^\mu)^+\rangle$, i.e., clause $(f)^-$ of the assumption which is
 enough by \ref{d.6.5}.
\end{PROOF}

\begin{conclusion}
\label{d.11}
1) $\BB(\lambda,\mu^+,\theta,\kappa)$ and if $\lambda$ is regular even 
$\BB(J^{\nst}_\lambda,\mu^+,\theta,\kappa)$ - see Definition \ref{0p.14} -
holds \when \, $\theta < \mu \in {\bold C}_\kappa$
and $\mu < \lambda < 2^\mu < 2^\lambda$.

\noindent
2) $\BB(\lambda,\mu^+,(2^\mu,\theta),\kappa)$ - see Definition
\ref{0p.15} - holds \underline{when} $\theta,\mu,\lambda$ are as
above.
\end{conclusion}

\begin{PROOF}{\ref{d.11}}
1) Let $\Upsilon = (2^{\theta + \kappa^+})^+$, so $\Upsilon <
\mu$.  By case (d) of \ref{d.7}, we have
$\Sep(\mu,\theta,\Upsilon)$.  Let $\langle C_\gamma:\gamma \in [\mu,\lambda)
\rangle$ be a $\mu^+$-free family of subsets of $\mu$ each of
order type $\kappa$ (exist by \ref{1.3.3}(c)) 
and let $\langle S_i:i < \lambda\rangle$ be a
partition of $[\mu,\lambda)$ into $\lambda$ (pairwise disjoint) sets each of
cardinality $\lambda$, stationary if $\lambda$ is regular 
and let $\langle \xi_{i,\alpha}:\alpha <
\lambda\rangle$ list $S_i$ in increasing order.  Clearly $\langle
C_\gamma:\gamma \in [\mu,\lambda)\rangle$ is a weak
  $(\lambda,\kappa)$-ladder system and is $\mu^+$-free so is as
  required in clause (A) of \ref{0p.14}.  Hence it 
suffices to find for each $i < \lambda$ a
function $\bold c_i$ with domain $S_i$, such that $\bold c_i(\gamma)
\in {}^{(C_\gamma)}\theta$ as in Definition \ref{0p.14}.

Clearly $\lambda \ge \lambda_0 := 
\min\{\partial:2^\partial > 2^\mu\}$, so if equality
holds, by \ref{d.6.5} there are $h_\xi \in {}^\lambda(2^\mu)$ for $\xi <
2^\lambda$ such that $\zeta \ne \varepsilon \Rightarrow h_\zeta
\ne_{J^{\bd}_\lambda} h_\varepsilon$.  So we can apply the Section
Main Claim \ref{d.6}$(\alpha)$ with $D$ taken to be the club filter and with
$\langle C_{\xi_{i,\alpha}}:\alpha \in [\mu,\lambda)\rangle$ here standing for
$\bar C$ there; we get $\bold c'_i$ with domain $\lambda$.  Let $\bold
c_i$ have domain $S_i,\bold c_i(\xi_{i,\alpha}) = \bold c'_i(\alpha)$
so $\bold c_i$ is as required.  If otherwise, 
i.e., $\lambda > \lambda_0$, the result
``$\BB(\lambda,\mu^+,\theta,\kappa)$" follows by
monotonicity of BB in $\lambda$.

To get ``if $\lambda$ is regular then
$\BB(J^{\nst}_\lambda,\mu^+,\theta,\kappa)$", let $g:\lambda
\rightarrow [\mu,\lambda_0)$ be such that $g^{-1}\{\alpha\}$ is a
stationary subset of $\lambda$ for $\alpha \in [\mu,\lambda_0)$ let
$\langle S'_i:i < \lambda\rangle$ be a partition of
$[\mu,\lambda_0)$ into stationary sets and use $S''_i = \{\beta <
 \lambda:g(\beta) \in S'_i\rangle,C''_\beta = C_{g(\beta)}$ and $D
      = \{A \subseteq \lambda$: for club $E$ of $\lambda_0,(\forall
 \beta < \lambda)(g(\beta) \in E \Rightarrow \beta \in A)\}$.

\noindent
2) The proof is similar. 
\end{PROOF}

\begin{conclusion}
\label{d.11p}
Suppose we add clause (g) and replace clause (b) by (b)$^+$ 
in the Section Main Claim \ref{d.6} where
\mn
\begin{enumerate}
\item[$(g)$]  $\lambda = \cf(\lambda)$ and ${\gd}_\lambda > 2^\mu$,
 recalling ${\gd}_\lambda = \cf({}^\lambda
 \lambda,<_{J^{\bd}_\lambda})$
\sn
\item[$(b)^+$]  $\lambda$ is regular and $D$ is the club filter on $\lambda$.
\end{enumerate}
\mn
\Then \, we can strengthen clause $(\beta)$ of the conclusion to:
\mn
\begin{enumerate}
\item[$(\beta)^+$]  if ${\bold F}_\gamma:{}^{(C_\gamma)}(2^\mu)
\rightarrow \theta$ for $\gamma < \lambda$ and ${\bold F}':{}^\mu(2^\mu)
\rightarrow {}^\lambda \lambda$, \then \, we can find 
$\bar c = \langle c_\gamma:\gamma \in S_* \rangle \in {}^\lambda
\theta$ with $S_* \in D^+$ such that:
\mn
\begin{enumerate}
\item[$(*)$]  for any $f:\mu \rightarrow 2^\mu$ for some $\gamma <
\lambda$ (and even for $D^+$-many $\gamma \in S_*$) we have

\[
{\bold F}_\gamma(f \restriction C_\gamma) = c_\gamma 
\text{ and } ({\bold F}'(f))(\gamma) < \min(S_* \backslash (\gamma +1))
\]

\end{enumerate}
\end{enumerate}
\end{conclusion}

\begin{PROOF}{\ref{d.11p}}
Note that clause (b)$^+$ here implies clause (b) from \ref{d.6}, so
the conclusion of  \ref{d.6} holds.
We do not have to repeat the proof of the Section Main Claim \ref{d.6}; just
to quote it as ${\mat F} = 
\{{\bold F}'(f):f$ a function from $\mu$ to $2^\mu\}$ is
a subset of ${}^\lambda \lambda$ of cardinality $\le 2^\mu$.  

Let $\cF' := \{\sup\{f_i:i < \mu\}:f_i \in \cF$ for $i < \mu\}$, so clearly:
\mn
\begin{enumerate}
\item[$(*)$]  $(a) \quad \cF' \subseteq {}^\lambda \lambda$
\sn
\item[${{}}$]  $(b) \quad |\cF'| \le 2^\mu$
\sn
\item[${{}}$]  $(c) \quad (\cF',\le)$ is $\mu^+$-directed.
\end{enumerate}
\mn
[Why Clause (c)?  Because if $f_i \in \cF'$ for 
$i < \mu$ then $\sup\{f_i:i < \mu\} \in \cF'$.]

Now we apply a result from Cummings-Shelah \cite[\S8]{CuSh:541}
possible as $\lambda > \mu,\mu$ strong limit, saying that $\cf({}^\lambda
\lambda,<_{J^{\bd}_\lambda}) = \cf({}^\lambda \lambda,
<_{J^{\nst}_\lambda})$, that is
${\gd}_\lambda = \cf({}^\lambda \lambda,<_{J^{\nst}_\lambda})$.  Hence
there is $f_* \in {}^\lambda \lambda$ such that the set $\{\alpha <
\lambda:f(\alpha) < f_*(\alpha)\}$ is a stationary subset of $\lambda$
for every $f \in {\mat F}'$.  For $f \in \cF$ let $S_f = \{\delta <
\lambda:\delta$ a limit ordinal and $f_*(\alpha) \le f(\delta)\}$
hence
\mn
\begin{enumerate}
\item[$(*)$]  $(a) \quad$ if $f_1 \le f_2$ are from $\cF'$ then
  $S_{f_1} \subseteq S_{f_2}$
\sn
\item[${{}}$]  $(b) \quad S_f \notin D$ for $f \in \cF'$.
\end{enumerate}
\mn
Now apply \ref{d.6} for the filter $D_* := \{S \subseteq \lambda:
S \cup S_f \in D$, i.e. contains a club of $\lambda$ for some $f \in \cF\}$.  
\end{PROOF}

\noindent
We still owe a proof of Claim \ref{d.7} giving sufficient conditions for
Sep$(\mu,\mu,\theta,\theta,\Upsilon)$.
\begin{PROOF}{\ref{d.7}}
\underline{Proof of \ref{d.7}}   

The cases 1-4 below cover all the clauses (a)-(e) of Claim
\ref{d.7} recalling
\mn
\begin{enumerate}
\item[$(*)_1$]   $\Sep(\mu,\theta,\Upsilon) =
  \Sep(\mu,\mu,\theta,\theta,\Upsilon)$
\end{enumerate}
\mn
and using freely the obvious
\mn
\begin{enumerate}
\item[$(*)_2$]   monotonicity: if $\Sep(\mu'_1,
\mu_1,\chi_1,\theta_1,\Upsilon_1)$ and $\mu'_1 \le
  \mu'_2,\mu_1 = \mu_2,\chi_1 \le \chi_2,\theta_1 =
  \theta_1,\Upsilon_1 \le \Upsilon_2$ \then \,
  $\Sep(\mu'_2,\mu_2,\chi_2,\theta_2,\Upsilon_2)$.
\end{enumerate}
\mn
Clause (a) is fully covered by case 1 using $\chi = \theta$, clause
(b) follows from clause (c) for the case $\sigma = \theta$ (and
monotonicity in $\Upsilon$), clause (c) by case 2 for $\chi = \theta$, 
clause (d) by case 3 letting $\sigma = \theta$ 
and clause (e) by case 4.
\mn
\newline
\underline{Case 1}:  $\mu = \mu^\theta,\Upsilon = \theta,\chi \in
[\theta,\mu]$ and we shall prove $\Sep(\mu,\mu,\chi,\theta,\theta)$.  Let

\[\begin{array}{ll}
{\mat F} = \biggl\{ f:&f \text{ is a function from } {}^\mu \chi
\text{ into } \theta \text{ and} \\
  &\text{for some } u \in [\mu]^\theta \text{ and a sequence } \bar
\rho = \langle \rho_i:i < \theta \rangle \\
  &\text{ with no repetition}, \rho_i \in {}^u \chi,
 \text{ we have} \\
  &(\forall \nu \in {}^\mu \chi)[\rho_i \subseteq \nu \Rightarrow
f(\nu) = i] \text{ and} \\
  &(\forall \nu \in {}^\mu \chi)[( \bigwedge_{i < \theta}
(\rho_i \nsubseteq \nu)) \Rightarrow f(\nu) = 0]\biggr\}.
\end{array}\]

\mn
We write $f = f^*_{u,\bar \rho}$, if $u,\bar \rho$ witness that
$f \in {\mat F}$ as above.  Notice that the size of the set of such pairs
$(u,\bar\rho)$ is $\mu^\theta$, and each such pair determines a unique $f$.

\noindent
Recalling $\mu = \mu^\theta$, clearly $|{\mat F}| = \mu$.  
Let ${\mat F} = \{f_\varepsilon:\varepsilon < \mu\}$ 
and we let $\bar f = \langle
f_\varepsilon:\varepsilon < \mu \rangle$.  Clearly clauses (a),(b) of
Definition \ref{d.8} (with $\mu,\mu,\chi,\theta,\theta$ here
standing for $\mu',\mu,\chi,\theta,\Upsilon$ there)
hold; let us check clause (c).  So suppose
$\varrho \in {}^\mu \theta$ and let $R = R_\varrho :=
\{\nu \in {}^\mu \chi$: for
every $\varepsilon < \mu$ we have $f_\varepsilon(\nu) \ne
\varrho(\varepsilon)\}$.   We have to prove that $|R| < \theta$
(as we have chosen $\Upsilon = \theta$).

\noindent
Towards a contradiction, assume that
$R \subseteq {}^\mu \chi$ has cardinality $\ge \theta$ and
choose $R' \subseteq R$ of cardinality $\theta$.  Hence we can find $u \in
[\mu]^\theta$ such that $\langle \nu \restriction u:\nu \in R'
\rangle$ is without repetitions.

Let $\{\nu_i:i < \theta\}$ list $R'$ without repetitions
and let $\rho_i := \nu_i \restriction u$ for $i < \theta$.
Now let $\bar \rho = \langle \rho_i:i < \theta \rangle$, so
$f^*_{u,\bar \rho}$ is well-defined and belongs
to ${\mat F}$.  Hence for some $\zeta < \mu$ we
have $f^*_{u,\bar \rho} = f_\zeta$.  Now for each $i < \theta,\nu_i \in
R' \subseteq R$, hence by the definition of $R,(\forall
\varepsilon < \mu)(f_\varepsilon(\nu_i) \ne \varrho(\varepsilon))$ and, in
particular, for $\varepsilon = \zeta$, we get $f_\zeta(\nu_i) \ne
\varrho(\zeta)$.  But by the choice of $\zeta,f_\zeta(\nu_i) =
f^*_{u,\bar \rho}(\nu_i)$ and by the definition of $f^*_{u,\bar \rho}$,
recalling $\nu_i \restriction u = \rho_i$, we
have $f^*_{u,\bar \rho}(\nu_i) = i$, so $i = f_\zeta(\nu_i)
\ne \varrho(\zeta)$.  This holds for every 
$i < \theta$ whereas $\varrho \in {}^\mu \theta$, a contradiction.
\bn
\newline
\underline{Case 2}: $\theta \le \chi < \mu,
\chi^{< \sigma} < \mu,\chi^\theta \le \mu,\sigma^\theta \le \mu,\theta \le
\sigma,J = [\sigma]^{< \theta}$ so it is an ideal on 
$\sigma,{\bold U}_J(\mu) = \mu,\Upsilon = (\chi^{< \sigma})^+$ 
recalling Definition \ref{d9}.  We shall prove
$\Sp(\mu,\mu,\chi,\theta,\Upsilon)$ which is more than required.

Let $\{u_\gamma:\gamma < \mu\} \subseteq [\mu]^{\le\sigma}$ exemplify
${\bold U}_J(\mu) = \mu$. Define ${\mat F}$ as in case 1 replacing
``$u \in [\mu]^\theta$" by ``$u \in {\mat P} := 
\bigcup\{[u_\gamma]^\theta:\chi < \mu\}$". 
As $\sigma^\theta \le \mu$ easily $|\cP| = \mu$ and as $\chi^\theta
\le \mu$ clearly $|{\mat F}| = \mu$.  Let $\langle
f_\varepsilon:\varepsilon < \mu\rangle$ list $\cF$, clearly clauses
(a),(b) of Definition \ref{d.8} hold and we shall prove clause (c).

Assume that $\varrho \in {}^\mu \theta$ and $R = R_\varrho
\subseteq {}^\mu \theta$ is defined as in case 1, and towards a
contradiction assume that $|R| \ge \Upsilon = (\chi^{<\sigma})^+$. 
We can find $\nu^*,\langle (\alpha_\zeta,\nu_\zeta):\zeta < \sigma
\rangle$ such that:
\mn
\begin{enumerate}
\item[$\boxplus$]  $(a) \quad \nu^*,\nu_\zeta \in R_\varrho$
\sn
\item[${{}}$]  $(b) \quad \alpha_\zeta < \mu$
\sn
\item[${{}}$]  $(c) \quad \nu_\zeta \restriction 
\{\alpha_\xi:\xi < \zeta\} = \nu^*
\restriction \{\alpha_\xi:\xi < \zeta\}$
\sn
\item[${{}}$]  $(d) \quad \nu_\zeta(\alpha_\zeta) \ne
\nu^*(\alpha_\zeta)$.
\end{enumerate}
\mn
[Why?  Obvious, as in the proof of the 
Erd\"os-Rado theorem; let $\langle \eta_i:i
< \Upsilon\rangle$ be a sequence with no repetitions of members of
$R$.  For each $j < \Upsilon$, we try to choose by induction on
$\zeta < \sigma$ ordinals $i(j,\zeta),\alpha_{j,\zeta}$ such that:
\mn
\begin{enumerate}
\item[$(a)$]  $i(j,\zeta) < j$ is increasing with $\zeta$
\sn
\item[$(b)$]  $\alpha_{j,\zeta} = \text{ min}\{\alpha:\eta_j(\alpha)
\ne \eta_{i(j,\zeta)}(\alpha)\}$
\sn
\item[$(c)$]  $i(j,\zeta) = \min\{i:i(j,\varepsilon) < i < j$
and $\eta_i(\alpha_{j,\varepsilon}) = \eta_j(\alpha_{j,\varepsilon})$
for $\varepsilon < \zeta\}$.
\end{enumerate}
\mn
If we succeed for some $j$ we are done.  Otherwise for each $j <
\Upsilon$ there is $\xi(j) < \sigma$ such that
$(i(j,\zeta),\alpha_{j,\zeta})$ is well defined iff $\zeta < \xi(j)$.

Let ${\mat T} = \{\langle(i(j,\zeta),\alpha_{j,\zeta}):\zeta < \xi\rangle:j
< \Upsilon$ and $\xi \le \xi(j)\}$ which is, under $\triangleleft$, a tree
with $\le \sigma$ levels, is normal, has a root and each node has at
most $\chi$ immediate successors, hence $|\mat T| \le \sum\limits_{i <
\sigma} |{}^i \chi| = \Sigma\{\chi^{|i|}:i < \sigma)\} = \chi^{< \sigma}$. 
But $j \mapsto \langle(i(j,\zeta),\alpha_{j,\zeta}):\zeta < \xi(j)\rangle$
is a one-to-one function from $\Upsilon$ into $\mat T$, a contradiction.]

Clearly $\langle \alpha_\zeta:\zeta < \sigma \rangle$ has no repetitions. 

So by the choice of $\{u_\gamma:\gamma < \mu\}$ as 
exemplifying ${\bold U}_J(\mu) = \mu$, i.e., the 
definition of ${\bold U}_J(\mu)$ and the choice of $J$,
for some $i < \mu$ the set $u_\gamma \cap \{\alpha_\zeta:\zeta <
\sigma\}$ has cardinality $\ge \theta$; choose a subset $u$ of this
intersection of cardinality $\theta$, hence $u \in {\mat P}$.
So $\{\nu \restriction u:\nu \in R\}$ has cardinality $\ge \theta$; 
\wolog \, $u = \{\alpha_{\zeta_i}:i < \theta\}$ where $\zeta_i$,
 increasing with $i$, and let $\rho^*_i =
\nu_{\zeta_i} \restriction u$ for $i < \theta$ and we can continue as
in Case 1.  
\bn
\newline
\underline{Case 3}:  $\mu > \theta \ne \text{ cf}(\mu)$ and $\sigma =
\theta$ (or $\theta \le \sigma \in \Reg \cap \mu \backslash \{\cf(\mu)\}$)
and $\mu$ is a strong limit cardinal, $\Upsilon = (2^{< \sigma})^+$
and we shall prove $\Sep(\mu,\theta,\Upsilon)$.

Letting $\chi = \theta$, this follows by case 2, 
the main point is ``$\bold U_J(\mu) = \mu$ 
where $J = [\sigma]^{< \theta}$, recalling Definition \ref{d9}.

Let $\cP = \cup\{u:u$ is a bounded subset of $\mu$ of cardinality $\le
\sigma\}$.  So $\cP \subseteq [\mu]^{\le\sigma}$ and as $\mu$ is a
strong limit cardinal clearly $\cP$ has cardinality $\le \mu$ and if
$f$ is a function from $\sigma$ to $\mu$, as $\sigma = \cf(\sigma) \ne
\cf(\mu)$ necessarily for some $\alpha < \mu$ the set $u_* :=
\{\varepsilon < \sigma:f(\varepsilon) < \alpha\}$ is of cardinality
$\sigma$ hence it belongs to $\cP$ (and has subsets of cardinality
exactly $\theta$ which necessarily belong to $\mu$).
\bn
\newline
\underline{Case 4}:  $\mu \ge \beth_\omega(\theta)$ and $\Upsilon =
\mu$ and we shall prove $\Sep(\mu,\theta,\Upsilon)$.

Let $\chi = \theta$ so we should prove $\Sep(\mu,\mu,\chi,\theta,\Upsilon)$.
By \cite{Sh:460} or see \cite{Sh:829} we can find a regular 
$\sigma < \beth_\omega(\theta)$
which is greater than $\theta$ and is such that ${\bold U}_\sigma(\mu) = \mu$
(i.e., the ideal is $J^{\bd}_\sigma$); hence $J := [\sigma]^{< \theta}
\subseteq J^{\bd}_\sigma$ hence trivially $\bold U_J(\mu) = \mu$; so 
case 2 applies and by monotonicity we are done.
\end{PROOF}
\bigskip
\centerline{$* \qquad * \qquad *$}
\bigskip

\begin{discussion}
\label{d19}
We may try to strengthen the results on
$\Sep(\mu,\theta,\kappa)$ assuming $\mu^\sigma = \mu$, a case which is
unnatural for \cite{Sh:775} but may be helpful.
\end{discussion}

\begin{claim}
\label{d20}
1) $\Sep(\mu,\theta,\Upsilon)$ when $\mu \ge \aleph_0 > \theta$ and
   $\Upsilon \ge \theta$.

\noindent
2) If $\BB(I,\bar C,(\lambda,\theta_1),< \kappa)$ and $[\alpha < \kappa
   \Rightarrow \theta^{|\alpha|}_2 \le \theta_2]$ \then \, $\BB(I,\bar
   C,\theta_2,< \kappa)$.
\end{claim}

\begin{PROOF}{\ref{d20}}
1) By the proof of \ref{d.6}, clause (a) and monotonicity of $\Sep$ in
   $\Upsilon$.

\noindent
2) As in the beginning of the proof of \ref{d.6}, i.e. proving it
 suffices to prove clause $(\beta)$ implies clause $(\gamma)$ of the
 conclusion. 
\end{PROOF}
\newpage

\section{Getting large 
$\mu^+$-free subsets of ${}^\kappa \mu$} \label{Getting} 

Recall that $\mu = \bold C_\kappa \Rightarrow \text{ pp}(\mu) =^+
 2^\mu$ and easily (see \ref{2b.98}(2))
\mn
\begin{enumerate}
\item[$\boxplus$]  if ${\mat F} \subseteq {}^\kappa \mu$ is
$\mu_1$-free and $\lambda = |{\mat F}| = 2^\mu$, then
BB$(\lambda,\mu_1,\lambda,\kappa)$, (and hence TDU$_{\mu_1}$ holds
when $\kappa \in \{\aleph_0,\aleph_1\}$). 
\end{enumerate}
\mn
This is a motivation of the investigation here, i.e., trying to
get more cases of $\mu^+$-free subsets for ${}^\kappa \mu$ of
cardinality $\pp(\mu)$.  In \ref{1f.7} the case of our interest 
is $\mu = \beth_\omega,\mu < \chi < \lambda = \beth_{\omega +1}(= 2^\mu)$,
cf$(\chi) = \theta \in (\aleph_\omega,\mu)$.
\bigskip

\begin{claim}
\label{1f.7}
There is a set ${\mat F} \subseteq {}^\kappa \mu$ of
cardinality $\lambda$ satisfying $\boxtimes$ if $\circledast$
holds \underline{where}
\mn
\begin{enumerate}
\item[$\boxtimes$]  $(\alpha) \quad$ the set ${\mat F}$ is
$(\theta,J_1)$-free, see Definition \ref{1.3.14},
\sn
\item[${{}}$]  $(\beta) \quad {\mat F}$ is
$(\mu^+,(2^\theta)^+,J_1)$-free - see Definition \ref{1.3.14},
\sn
\item[$\circledast$]   $(a) \quad \mu < \chi < \lambda$,
\sn
\item[${{}}$]  $(b) \quad \kappa = \text{\rm cf}(\mu) < \mu$,
\sn
\item[${{}}$]  $(c) \quad \theta$ is regular (naturally but not
necessarily $\theta = \text{\rm cf}(\chi))$,
\sn
\item[${{}}$]  $(d) \quad \kappa < \theta < \mu$ \underline{or} just $\kappa
\ne \theta$ are both $< \mu$,
\sn
\item[${{}}$]  $(e) \quad \alpha < \mu \Rightarrow |\alpha|^\theta < \mu$,
\sn
\item[${{}}$]  $(f) \quad J=J_1$ is a $\kappa$-complete ideal on
  $\kappa$, including $J^{\bd}_\kappa$, of course
\sn
\item[${{}}$]  $(g) \quad \chi^{<\theta>_{\text{\rm tr}}} \ge^+
\lambda$ as witnessed by ${\mat T}$; i.e., the tree ${\mat T}$ has 
$\theta$ levels, 

$\hskip25pt \le \chi$ nodes and $\ge \lambda$ distinct
$\theta$-branches,
\sn
\item[${{}}$]  $(h) \quad$ {\rm pp}$_{J_1}(\mu) > \chi$
\end{enumerate}
\end{claim}

\begin{claim}
\label{1f.8}
In Claim \ref{1f.7} we can replace
$\circledast$ by $\circledast'$ and $\boxtimes(\beta)$ by $\boxtimes'(\beta)'$
below, i.e. if $\circledast'$ holds then 
there is $\cF \subseteq {}^\kappa \mu$ of cardinality
$\lambda$ such that $\boxtimes'$ holds where:
\mn
\begin{enumerate}
\item[$\boxtimes'$]  $(\alpha) \quad$ the set $\cF$ is $(\theta_1,J_1)$-free,
\sn
\item[${{}}$]  $(\beta)' \quad {\mat F}$ is $(\mu^+,\sigma,J_1)$-free,
\sn
\item[$\circledast'$]  $(a) \quad \mu < \chi < \lambda$,
\sn
\item[${{}}$]  $(b) \quad \kappa = \text{\rm cf}(\mu) < \mu$,
\sn
\item[${{}}$]  $(c) \quad J_2$ is an ideal on $\theta$,
\sn
\item[${{}}$]  $(d) \quad J= J_1$ is an ideal on $\kappa$,
\sn
\item[${{}}$]  $(e) \quad \alpha < \mu \Rightarrow |\alpha|^\theta
< \mu$ (hence $\theta < \mu$),
\sn
\item[${{}}$]  $(f) \quad \theta_1$ satisfies $(\alpha)$ or $(\beta)$ where
\sn
\begin{enumerate}
\item[${{}}$]  $(\alpha) \quad \theta_1 \le \theta$ and
$J_2$ is $\theta_1$-complete,
\sn
\item[${{}}$]  $(\beta) \quad J_1$ is $\cf(\theta_1)^+$-complete 
and $J_2 = J^{\bd}_{\theta_1}$ and $\theta_1 < \kappa$ of course,
\end{enumerate}
\sn
\item[${{}}$]  $(g) \quad$ there are $\eta_\alpha \in {}^\theta \chi$ for
$\alpha < \lambda$ such that $\alpha < \beta < \lambda
\Rightarrow \{\varepsilon < \theta$:

\hskip25pt $\eta_\alpha(\varepsilon) = \eta_\beta(\varepsilon)\} \in J_2$,
\sn
\item[${{}}$]   $(h) \quad$ there is a $(\mu^+,J_1)$-free ${\mat F}
\subseteq {}^\kappa \mu$ of cardinality $\ge \chi$,
\sn
\item[${{}}$]  $(i) \quad (\alpha) \quad
{\mat P}(\theta)/J_2$ satisfies the $\sigma$-c.c. \underline{or} just
\sn
\item[${{}}$]  $\quad \quad (\beta) \quad$ for some $\kappa^+$-complete
ideal $J'_2 \supseteq J_2$ of $\theta$,

\hskip35pt $\sigma \ge \sup\{|{\mat A}|^+:{\mat A} \subseteq
{\mat P}(\theta) \backslash J'_2$ and $A \ne B \in {\mat A}
\Rightarrow A \cap B \in J_2\}$.
\end{enumerate}
\end{claim}

\begin{remark}
\label{1f.9}
1) Recall Definition \ref{1.3.14} where we defined 
notions of freeness for sets and for sequences.

\noindent
2) The proof of \ref{1f.7} is written so it can be adapted to become
a proof of \ref{1f.8}.
\end{remark}
\bigskip

\begin{PROOF}{\ref{1f.7}}
Proof of Claim \ref{1f.7}:   
As $\cf(\mu) = \kappa < \mu$ by clause (b) of
$\circledast$; and $\alpha < \mu \Rightarrow |\alpha|^\theta < \mu$ by
clause (e), we can let $\langle \mu_i:i < \kappa\rangle$ be increasing 
with limit $\mu$ such that $(\mu_i)^\theta = \mu_i > 2^\theta$.  Let 
$\mu^-_i = \bigcup_{j<i} \mu_j$; 
\wilog \, $\mu^-_i < \mu_i < \mu$; if $\sigma \le \kappa$ and 
$(\forall \alpha < \mu)(|\alpha|^\sigma < \mu)$, we can 
add $(\mu_i)^{\kappa + \theta} = \mu_i$.

There is $\bar \rho = \langle \rho_\gamma:\gamma < \chi\rangle$ such
that:
\mn
\begin{enumerate}
\item[$(*)_1$]  $(a) \quad \rho_\gamma \in  \prod_{i <
\kappa} \mu_i$ with no repetition; moreover $\rho_\gamma(i) \in
[\mu^-_i,\mu_i)$
\sn
\item[${{}}$]  $(b) \quad$ the set $\{\rho_\alpha:\alpha <
\chi\}$ is $(\mu^+,J_1)$-free (in fact we can add that

\hskip25pt  even the sequence $\langle \rho_\alpha:\alpha < \chi\rangle$ is
$\mu^+$-free, recalling 

\hskip25pt Definition \ref{1.3.14}(1),(2) but this is immaterial here). 
\end{enumerate}
\mn
[Why?  For any regular $\chi_1 \in (\mu,\chi]$ by clause (h) of the
assumption $\circledast$ and the no-hole claim, 
there is an increasing sequence $\langle
\lambda_i:i < \kappa\rangle$ of regular cardinals $< \mu$ with limit
$\mu$ such that $\chi_1 < \tcf(\prod\limits_{i < \kappa}
\lambda_i,<_{J_1})$.

As we can replace $\langle \mu_i:i < \kappa\rangle$ by any subsequence
of length $\kappa$, for some non-decreasing sequence, \wilog \, $\mu^-_i
 < \lambda_i < \mu_i$.

By the no-hole-claim (really \cite[Ch.II,1.5A]{Sh:g}) there are
$\rho_\gamma \in \prod\limits_{i < \kappa} [\mu^-_i,\lambda_i)
\subseteq \prod\limits_{i < \kappa} [\mu^-_i,\mu_i)$ for $\gamma <
\chi_1$ such that $\langle \rho_\gamma:\gamma < \chi_1\rangle$ is
$(\mu^+,J_1)$-free.  If $\chi$ is regular, we can use $\chi_1 :=
\chi$.  We are left with a case $\chi$ is singular; however, by the
strenghening of the no-hole claim in \cite[Ch.II,1.5A,pg.51]{Sh:g}
there is a sequence $\langle \rho_\gamma:\gamma < \chi\rangle$ as
above.  So $(*)_1$ holds indeed.

Let $J_2 = J^{\bd}_\theta$, (for \ref{1f.8} the ideal $J_2$ is
given in clause (c)); and let
${\mat T}$ be a tree as in clause (g) of the assumption
$\circledast$.  Without loss of generality
\mn
\begin{enumerate}
\item[$(*)_2$]  $(a) \quad {\mat T} \subseteq {}^{\theta >}\chi$
and $<_{\mat T}$ is $\triangleleft$, i.e. being an initial segment
\sn
\item[${{}}$]  $(b) \quad$ if $\eta_1,\eta_2 \in {\mat T} \wedge
  \varepsilon_1 < \theta \wedge \varepsilon_2 < \theta \wedge 
\eta_1(\varepsilon_1) = \eta_2(\varepsilon_2)$ then $\varepsilon_1 
= \varepsilon_2$

\hskip25pt  and we can add, but not used, $\eta_1 \restriction
  \varepsilon_1 =$$\eta_2 \restriction \varepsilon_2$.
\end{enumerate}
\mn
Recall $\lim_\theta({\mat T}) = \{\eta \in {}^\theta
\chi:(\forall \varepsilon < \theta)(\eta \restriction \varepsilon \in
{\mat T})\}$, so it has $\ge \lambda$ members.

Let $\langle \eta_\alpha:\alpha < \lambda\rangle$ be a sequence of
pairwise distinct members of lim$_\theta({\mat T})$.  Let cd$_*:\cup
\{{}^\theta(\mu_i):i < \kappa\} \rightarrow \mu$ be one-to-one onto
$\mu$ such that $\rho \in {}^\theta(\mu_i) \Leftrightarrow
\cd_*(\rho) < \mu_i$.  Let
$\langle \cd_\varepsilon:\varepsilon < \theta\rangle$ be the
sequence of functions with domain $\mu$
 such that $\zeta = \cd_*(\rho) \Rightarrow \rho = \langle
\cd_\varepsilon(\zeta):\varepsilon < \theta \rangle$.  Let
$\cd'_\varepsilon(\zeta) = \cd_\varepsilon(\cd_0(\zeta))$.

Lastly, for $\alpha < \lambda$ (the second and third demands 
are for later claims using this proof)
\mn
\begin{enumerate}
\item[$\boxtimes_1$]  $\nu_\alpha \in {}^\kappa \mu$ is defined as follows:
\sn
\begin{enumerate}
\item[$\bullet$]  for $i < \kappa$, let $\nu_\alpha(i) 
\in [\mu^-_i,\mu_i)$ be such that cd$'_\varepsilon(\nu_\alpha(i)) =
\rho_{\eta_\alpha(\varepsilon)}(i)$ for every $\varepsilon < \theta$
\sn
\item[$\bullet$]   if $(\forall \alpha < \mu)(|\alpha|^\kappa < \mu)$, 
\then \, we can make $\nu_\alpha(i)$ also code $\nu_\alpha \restriction i$,
e.g. $\cd_1(\nu_\alpha(i))$ codes $\nu_\alpha \rest i$
\sn
\item[$\bullet$]  if $\varrho_\alpha \in \prod\limits_{i <\kappa}
  \mu_i$ for $\alpha < \lambda$ are given \then \, we can add that 
$\nu_\alpha(i)$ codes $\varrho_\alpha(i)$,
too, e.g. $\varrho_\alpha(i) = \cd_0(\cd_2(\nu_\alpha(i)))$.
\end{enumerate}
\end{enumerate}
\mn
[Why?  E.g. why the demand $\nu_\alpha(i) \ge \mu^-_i$ is O.K.?  Because
  $\cd_*$ is a one-to-one function and the freedom in choosing
  $\cd_3(\nu_\alpha(i))$.] 

We shall prove that the set ${\mat F} = \{\nu_\alpha:\alpha <
\lambda\}$ is as required and let
$\bar \nu = \langle \nu_\alpha:\alpha < \lambda\rangle$.

Now
\begin{enumerate}
\item[$\boxtimes_2$]  $\bar \nu$ is without repetition, i.e., $\alpha
< \beta < \lambda \Rightarrow \nu_\alpha \ne \nu_\beta$: and so the set
${\mat F}$ has cardinality $\lambda$.
\end{enumerate}
\mn
[Why?  If $\nu_\alpha = \nu_\beta$, then for every $\varepsilon <
\theta$ and $i < \kappa$, we have $\rho_{\eta_\alpha(\varepsilon)}(i)
= \cd'_\varepsilon(\nu_\alpha(i)) = \cd'_\varepsilon
(\nu_\beta(i)) = \rho_{\eta_\beta(\varepsilon)}(i)$.
Fixing $\varepsilon < \theta$, as this holds for every 
$i < \kappa$, we conclude that $\rho_{\eta_\alpha(\varepsilon)} = 
\rho_{\eta_\beta(\varepsilon)}$.  But $\langle \rho_\gamma:\gamma <
\chi\rangle$ is without repetitions, hence it follows that
$\eta_\alpha(\varepsilon) = \eta_\beta(\varepsilon)$.  
As this holds for every $\varepsilon <
\theta$, we conclude that $\eta_\alpha = \eta_\beta$ but $\langle
\eta_\alpha:\alpha < \lambda\rangle$ is without repetitions hence
$\alpha = \beta$, so we are done.]

Now the main point is proving clauses $(\alpha)$ and $(\beta)$ of
$\boxtimes$.
\bn
\newline
\underline{Step 1}:  To prove clause $(\alpha)$ of $\boxtimes$, i.e.,
``${\mat F}$ is $(\theta,J_1)$-free".

Assume $w \subseteq \lambda$ and $|w| < \theta$.  Recalling $(*)_1(b)$
and $\theta < \mu$, clearly the set 
$\{\rho_{\eta_\alpha(\varepsilon)}:\alpha \in w,\varepsilon < 
\theta\}$ being of cardinality $\le \theta < \mu^+$ is free, 
hence there is a sequence $\langle 
s_{\eta_\alpha(\varepsilon)}:\alpha \in w,\varepsilon < \theta\rangle$
of members of $J_1$ such that: if $(\alpha_\ell,\varepsilon_\ell) \in w
\times \theta$, for $\ell=1,2$, and $\eta_{\alpha_1}(\varepsilon_1)
\ne \eta_{\alpha_2}(\varepsilon_2)$ and $i \in \kappa
\backslash s_{\eta_{\alpha_1}(\varepsilon_1)}  \backslash
s_{\eta_{\alpha_2}(\varepsilon_2)}$ (recalling \ref{0p.31}(0)), then
$\rho_{\eta_{\alpha_1}(\varepsilon_1)}(i) \ne
\rho_{\eta_{\alpha_2}(\varepsilon_2)}(i)$.

Now as $\langle \eta_\alpha:\alpha \in w\rangle$ is a sequence of $<
\theta$ distinct $\theta$-branches of ${\mat T}$ and
$\eta_{\alpha_1}(\varepsilon_1) = \eta_{\alpha_2}(\varepsilon_2)
\Rightarrow \varepsilon_1 = \varepsilon_2$ and 
$\eta_{\alpha_1}(\varepsilon) = \eta_{\alpha_2}(\varepsilon)
\Rightarrow \eta_{\alpha_1} \restriction \varepsilon = \eta_{\alpha_2}
\restriction \varepsilon$ by $(*)_2$, i.e., by the choice of ${\mat T}$.
Hence by the regularity of $\theta$ we can find 
$\varepsilon_* < \theta$ such that $\langle
\eta_\alpha(\varepsilon_*)):\alpha \in w\rangle$ has no
repetitions, and define $s'_\alpha = s_{\eta_\alpha(\varepsilon_*)}
\subseteq \kappa$ for $\alpha \in w$; now $\langle s'_\alpha:\alpha
\in w\rangle$ is as required.
\mn
[Why?  First $s'_\alpha \in J_1$ by the choice of $s'_\alpha$.  
Second, assume $\alpha \ne \beta$ are from $w$ and 
$i \in \kappa \backslash s'_\alpha \backslash
s'_\beta$ and we should prove $\nu_\alpha(i) \ne \nu_\beta(i)$.  Now
$\eta_\alpha(\varepsilon_*) \ne \eta_\beta(\varepsilon_*)$ by the
choice of $\varepsilon_*$ and $s'_\alpha = 
s_{\eta_\gamma(\varepsilon_*)},s'_\beta =
s_{\eta_\beta(\varepsilon_*)}$ hence $i \in \kappa \backslash
s_{\eta_\alpha(\varepsilon_*} \backslash
s_{\beta_\beta(\varepsilon_*)}$ so by the choice of $\langle
s_{\eta_\gamma(\varepsilon)}:\gamma \in w,\varepsilon < \theta\rangle$
we have $\rho_{\eta_\alpha(\varepsilon_*)}(i) \ne
\rho_{\eta_\beta(\varepsilon_*)}(i)$ hence
$\cd'_{\varepsilon_*}(\nu_\alpha(i)) = 
\rho_{\eta_\alpha(\varepsilon_*)}(i) \ne
\rho_{\eta_\beta(\varepsilon_*)}(i) = 
\cd'_{\varepsilon_*}(\nu_\beta(i))$ which implies that $\nu_\alpha(i)
\ne \nu_\beta(i)$. 

Note also that $\mat{F}$ is normal by $\boxplus_1$ as the intervals
$[\mu^-_i,\mu_i)$ for $i < \kappa$ are pairwise disjoint.
\bn
\newline
\underline{Step 2}:  To prove clause $(\beta)$ of $\boxtimes$.

Let ${\mat F}' \subseteq \{\nu_\alpha:\alpha < \lambda\}$ have
cardinality $\le \mu$.  Choose $w$ such that
${\mat F}' = \{\nu_\alpha:\alpha \in w\}$,
so that $w \in [\lambda]^{\le \mu}$ and let $u := \cup\{\Rang(\eta_\alpha):
\alpha \in w\}$.  Clearly $u \in [\chi]^{\le \mu}$.
By the choice of $\langle \rho_\gamma:\gamma < \chi\rangle$ we can
find a sequence $\langle s_\gamma:\gamma \in u\rangle$ such that
$s_\gamma \in J_1$ and $i \in \kappa \backslash (s_{\gamma_1} \cup
s_{\gamma_2}) \wedge \gamma_1 \ne \gamma_2 \wedge
\{\gamma_1,\gamma_2\} \subseteq u \Rightarrow \rho_{\gamma_1}(i) \ne
\rho_{\gamma_2}(i)$.

For $\alpha \in w$ let $t_\alpha := \{i < \kappa$ : 
the set of $\varepsilon < \theta$ such that $i \notin
s_{\eta_\alpha(\varepsilon)}$ belongs to $J_2 =J^{\bd}_\theta\}$.

We shall now show that $\bar t := 
\langle t_\alpha:\alpha \in w\rangle$ is as required in Definition
\ref{1.3.14}(1),(2); that is, we have to prove that $t_\alpha \in J_1$
and that for any $\xi < \mu$ and $i_* < \kappa$ 
 the set of $\alpha \in w$ such that $i_* \notin t_\alpha \wedge
\nu_\alpha(i_*) = \xi$ is small, i.e. of cardinality $\le 2^\theta$;
these demands are proved below in $(*)_4$ and $(*)_3$ respectively.
So let $\xi < \mu$ and $i_* < \kappa$ and let $v = v_{\xi,i_*}
= \{\alpha \in w:i_* \notin t_\alpha$ and $\nu_\alpha(i_*) = \xi\}$.

First we shall prove below that
\mn
\begin{enumerate}
\item[$(*)_3$]  $|v| \le 2^\theta$.
\end{enumerate}
\mn
This will do one half of proving ``$\bar t$ is as required in Definition
\ref{1.3.14}(1),(2)."

\noindent
Why does $(*)_3$ hold?    Now if
$\alpha \in v$, then $i_* \in \kappa \backslash t_\alpha$, hence (by the
definition of $t_\alpha$) we have ${\mat U}_{\alpha,i_*} :=
\{\varepsilon < \theta:i_* \notin s_{\eta_\alpha(\varepsilon)}\} \in
J^+_2$.  So if $\alpha \ne \beta$ are from $v$ and $\varepsilon \in
{\mat U}_{\alpha,i_*} \cap {\mat U}_{\beta,i_*}$ and
$\eta_\alpha(\varepsilon) \ne \eta_\beta(\varepsilon)$, then we have $i_*
\notin s_{\eta_\alpha(\varepsilon)}$ (as $\varepsilon \in {\mat
U}_{\alpha,i_*})$ and $i_* \notin s_{\eta_\beta(\varepsilon)}$ (as
$\varepsilon \in {\mat U}_{\beta,i_*})$, and hence by the choice of
$\langle s_\gamma:\gamma \in u\rangle$, we have
$\rho_{\eta_\alpha(\varepsilon)}(i_*) \ne
\rho_{\eta_\beta(\varepsilon)}(i_*)$, so
\mn
\begin{enumerate}
\item[$(*)_4$]  $\cd'_\varepsilon(\nu_\alpha(i_*)) =
\rho_{\eta_\alpha(\varepsilon)}(i_*) \ne
\rho_{\eta_\beta(\varepsilon)}(i_*) = \cd'_\varepsilon(\nu_\beta(i_*))$.
\end{enumerate}
\mn
Recall that $\nu_\alpha(i_*) = \xi = \nu_\beta(i_*)$ because 
$\varepsilon \in {\mat U}_{\alpha,i_*} \cap {\mat
U}_{\beta,i_*}$, but this contradicts $(*)_4$. It follows that $\alpha
\in v \wedge \beta \in v \wedge \alpha \ne \beta \wedge
\varepsilon \in {\mat U}_{\alpha,i_*} \cap {\mat U}_{\beta,i_*}
\Rightarrow \eta_\alpha(\varepsilon) = \eta_\beta(\varepsilon)$;
but $\alpha \ne \beta \Rightarrow \{\varepsilon <
\theta:\eta_\alpha(\varepsilon) = \eta_\beta(\varepsilon)\} \in
J_2$, hence this implies $\alpha \in v \wedge \beta \in v \wedge
\alpha \ne \beta \Rightarrow {\mat U}_{\alpha,i_*} \cap {\mat
U}_{\beta,i_*} \in J_2$. As we have noted earlier that $\alpha \in
v \Rightarrow {\mat U}_{\alpha,i_*} \in J^+_2$, it follows that
${\mat P}(\theta)/J_2$ fails the $|v|$-c.c.  But for the present
proof, ${\mat P}(\theta)$ has cardinality $2^\theta$, hence ${\mat
P}(\theta)/J_2$ satisfies the $(2^\theta)^+$-c.c., and so $|v|
\le 2^\theta$, as required in $(*)_3$.  For proving ``$\bar t$
is as required in Definition \ref{1.3.14}", we need also the second
half:
\mn
\begin{enumerate}
\item[$(*)_5$]  $t_\alpha \in J_1$ for $\alpha \in w$.
\end{enumerate}
\mn
Why does $(*)_5$ hold?  Firstly, assume $\kappa < \theta$; 
towards a contradiction assume that $t_\alpha \in
J_1^+$.  By the choice of $t_\alpha$, for each 
$i \in t_\alpha$, the set $\{\varepsilon < \theta: 
i \notin s_{\eta_\alpha(\varepsilon)}\}$ belongs to $J_2$, but $J_2$, 
being euqal to $J^{\text{bd}}_\theta$ (and recalling $\theta$ is regular),
is $\kappa^+$-complete and $|t_\alpha| \le \kappa$, hence the set

\[
r_{\eta_\alpha} := \bigcup\limits_{i \in t_\alpha} \{\varepsilon < \theta:i
\notin s_{\eta_\alpha(\varepsilon)}\} 
\]

\mn
lies in $J_2$ hence we can choose $\varepsilon_\alpha < \theta$ such that
$\varepsilon = \varepsilon_\alpha \Rightarrow
\bigwedge\limits_{i \in t_\alpha} i \in s_{\eta_\alpha(\varepsilon)}$, so
$t_\alpha \subseteq s_{\eta_\alpha(\varepsilon_\alpha)}$, but
$s_{\eta_\alpha(\varepsilon_\alpha)} \in J_1$, 
and hence $t_\alpha \in J_1$ as required.

Secondly, assume $\kappa > \theta$; towards a contradiction,
assume $t_\alpha \in J^+_1$. Again $i \in t_\alpha \Rightarrow
\{\varepsilon < \theta:i \notin s_{\eta_\alpha(\varepsilon)}\} \in
J_2$, but $J_2 = J^{\text{bd}}_\theta$, hence we can find
$\bar\varepsilon_\alpha = \langle \varepsilon_{\alpha,i}:i \in
t_\alpha \rangle \in {}^{(t_\alpha)}\theta$ such that
$\varepsilon_{\alpha,i}
= \sup\{\varepsilon < \theta:i \notin s_{\eta_\alpha(\varepsilon)}\} <
\theta$.  However, $J_1$ is $\kappa$-complete (see clause (f) of
$\circledast$) hence $J_1$ is $\theta^+$-complete, so for some
$\varepsilon^*_\alpha < \theta$, we have $t'_\alpha := \{i \in
t_\alpha:\varepsilon_{\alpha,i} < \varepsilon^*_\alpha\} \in J^+_1$.
So $i \in t'_\alpha \Rightarrow \varepsilon_{\alpha,i} <
\varepsilon^*_\alpha\Rightarrow \sup\{\varepsilon < \theta:i \notin
s_{\eta_\alpha(\varepsilon)}\} < \varepsilon^*_\alpha \Rightarrow i
\in s_{\eta_\alpha(\varepsilon^*_\alpha)}$ so $t'_\alpha \subseteq
s_{\eta_\alpha(\varepsilon^*_\alpha)}$.  But
$s_{\eta_\alpha(\varepsilon^*_\alpha)} \in J_1$, while $t'_\alpha \notin J_1$,
a contradiction.  
\end{PROOF}

\begin{PROOF}{\ref{1f.8}}
Proof of \ref{1f.8}:

We note the points of the proof of \ref{1f.7}
which have to be changed.  The choice of $\bar\rho = \langle
\rho_\gamma:\gamma < \chi\rangle$, i.e. $(*)_1$ is now done by using
$\circledast'(h)$.  Before $(*)_2$, instead of defining $J_2$ recall that
it is given (see $\circledast'(f))$
and if $J'_2$ is not given (see $\circledast'(i)(\beta)$) let $J'_2 = J_2$.
After $(*)_2$, instead of choosing $\langle \eta_\alpha:
\alpha < \lambda\rangle$ it is given in
$\circledast'(g)$ and the tree $\cT$ disappears, so we ``lose" the
statement ``$\eta_1 \rest \varepsilon_1 = \eta_2 \rest \varepsilon_2$"
in the end of $(*)_2(h)$, the ``$\eta_1(\varepsilon_1) =
\eta_2(\varepsilon_2)$" is easy to get.

Now step 1 says that ``${\cF}$ is $(\theta_1,J_1)$-free".  Thus we 
have to choose $\varepsilon_*$ as there.  Of course, 
now $|w| < \theta_1$ as we are proving ``$\cF$ is $(\theta_1,J_1)$-free".

First, if clause $(\alpha)$ of $\circledast'(f)$ holds, as 
$\cU^1_{\alpha,\beta} := \{\varepsilon <
\theta:\eta_\alpha(\varepsilon) = \eta_\beta(\varepsilon)\} \in
J_2$ for $\alpha \ne \beta$ from $w$, but $J_2$ is $\theta_1$-complete, so
$\{\varepsilon < \theta:\eta_\alpha(\varepsilon) = 
\eta_\beta(\varepsilon)$ for some
$\alpha \ne \beta$ from $w\}$ belongs to $J_2$, hence there is
$\varepsilon_* < \theta$ not in $\cup\{\cU^1_{\alpha,\beta}:\alpha \ne
\beta$ are from $\omega\}$. 

Second, if clause $(\beta)$ of $\circledast'(f)$ clearly $\theta_1 <
\kappa$, so as $J_1$ is $\kappa$-complete it suffices to prove $\alpha
< \beta < \lambda \Rightarrow s_{\alpha,\beta} = \{i <
\kappa:\nu_\alpha(i) = \nu_\beta(i)\} \in J_1$ but for $\alpha \ne
\beta$ we have $\eta_\alpha
\ne \eta_\beta$ hence for some $\varepsilon < \theta$ we have
$\eta_\alpha(\varepsilon) \ne \eta_\beta(\varepsilon)$ hence
$s_{\alpha,\beta} \subseteq \{i <
\kappa:\rho_{\eta_\alpha(\varepsilon)}(i) =
\rho_{\eta_\beta(\varepsilon)}(i)\} \in J_1$ so we are done.

Turning to step 2, now to define $t_\alpha$ we use ``belongs to
$J'_2$"; then $(*)_3$ should say $|v| < \sigma$ and in the proof
instead of ``$\cP(\theta)/J_2$ satisfies the $(2^\theta)^+$-c.c." we use
clause $\circledast'(i)(\alpha)$ if it holds and
$\circledast'(i)(\beta)$ otherwise, as still $\alpha \ne \beta \Rightarrow
\cU_{\alpha,i_*} \cap \cU_{\beta,i_*} \in J_2$.

Lastly, to prove $(*)_5$ we use clause $\circledast'(f)$.   
\end{PROOF}

\begin{claim}
\label{1f.10}
In \ref{1f.7}, recalling 
$J=J_1$ is a ($\kappa$-complete) ideal on $\kappa$, and
letting $J_2 = J^{\bd}_\theta$ assuming 
$(\forall \alpha < \mu)(|\alpha|^\kappa < \mu)$ 
we can add to the conclusion that ${\mat F}$ is $(\Upsilon,J)$-free \when \, 
(a) or (b) or (c) hold where:
\mn
\begin{enumerate}
\item[\underline{Case $(a)$}]  $\Upsilon = \theta^{+\omega +1}$ 
and we can choose $\eta_\alpha \in {}^\theta \chi$ 
for $\alpha < \lambda$ with no
repetitions such that 
$\theta^+ \notin \issp_J(\{\eta_\alpha:\alpha < \lambda\})$.
\sn
\item[\underline{Case $(b)$}]   $\theta^{+ \omega} < \Upsilon 
\le \mu$ and we can choose $\eta_\alpha \in {}^\theta \chi$ 
for $\alpha < \lambda$ with no
repetitions such that $\theta < \partial = \cf(\partial) \wedge
(< \partial,\partial) \in \issp_J(\{\eta_\alpha:
\alpha < \lambda\}) \Rightarrow \partial \ge \Upsilon$.
\sn
\item[\underline{Case $(c)$}]   there are pairwise distinct  
$\eta_\alpha \in {}^\theta \chi$ for $\alpha < \lambda$ and pairwise distinct
$\varrho_\gamma \in {}^\kappa \mu$ for $\gamma < \chi$ such that for
every regular $\partial \in (\theta + \kappa^+,\Upsilon)$ we have
$\partial \notin \text{\rm issp}(\{\eta_\alpha:
\alpha < \lambda\})$ and $\partial \notin 
\text{\rm issp}(\{\varrho_\gamma:\gamma < \chi\})$. 
\end{enumerate}
\end{claim}

\begin{PROOF}{\ref{1f.10}}
The proof splits to cases.
\medskip

\noindent
\underline{Case (a)}:

We reduce it to case (b) proved below.  It follows by \ref{1.3.15} but
we elaborate.  So assume toward contradiction that case (b) fails, so
there is $\partial$ such that $\theta < \partial = \cf(\partial)$ and
$(< \partial,\partial) \in \text{\rm issp}_J(\{\eta_\alpha:\alpha <
\lambda\})$ but $\partial < \Upsilon$.  We can choose a minimal such
that $\partial$ and let $\partial_1 < \partial$ be such that
$(\partial_1,\partial) \in \text{\rm issp}_J(\{\eta_\alpha:\alpha <
\lambda\})$.  So by Definition \ref{1.3.14}(6) with $(\chi,\theta)$
here standing for $(\mu,\kappa)$ there, there is a set $u \subseteq
\chi$ of cardinality $\le \partial_1$ such that $\partial \le |\mat
U_u|$ where $\mat U_u := \{\alpha:\alpha < \lambda$ and $\{i <
\theta:\eta_\alpha(i) \in u\} \in J^+_2\}$.

\Wilog \, $\partial_1 \ge \theta$; clearly $\partial = \partial^+_1$
by the minimality of $\partial$ as $\partial^+_1$ is regular $>
\theta$.  Also if $\partial_1 = \theta$ we get the desired
contradiction (i.e. clause (a) fails); so we can assume $\partial_1 >
\theta$.

Let $\langle \alpha_\varepsilon:\varepsilon < \partial_1\rangle$ list
the elements of $u$ and let $\cU_{u,\zeta} :=
\{\alpha_\varepsilon:\varepsilon < \zeta\}$ for $\zeta \le
\partial_1$.  As $\theta < \partial_1 < \partial < \partial < \Upsilon
= \theta^{+\theta +1}$ we have $\text{cf}(\partial_1) \ne \theta$ so
recalling $J_2 = J^{\bd}_\theta$, for every $\alpha \in \mat U_u$ for
some $\varepsilon(\alpha) < \partial_1$ we have $\{i <
\theta:\eta_\alpha(i) \in \mat U_{u,\varepsilon(\alpha)}\} \in
J^+_2$.  As $|\mat U_u| \ge \partial = \partial^+_1 > \partial_1$
necessarily for some $\varepsilon(*) < \partial_1$ the set $\{\alpha
\in \mat U_u:\{i < \theta:\eta_\alpha(i) \in \mat
U_{u,\varepsilon(*)}\}$ has cardinality $\ge \partial \ge
\partial_1$.  So $\mat U_{u,\varepsilon(*)}$ witness that also
$\partial_1$ satisfies the demand on $\partial$, contradicting the
minimality of $\partial$, so we are done.

So case (b) holds and this is proved below.
\medskip

\noindent
\underline{Case (b)}:

We shall prove that case (c) holds, so toward contradiction assume it
fails.  Recalling Definition \ref{1.3.14}(7), note that choosing any
$\varrho_\gamma \in {}^\kappa\mu$ for $\gamma < \chi$ clause (c) holds.
\medskip

\noindent
\underline{Case (c)}:

We repeat the proof of \ref{1f.7} but we use $\langle \eta_\alpha:\alpha
< \chi\rangle,\langle \varrho_\alpha:\alpha < \lambda\rangle$ 
from the assumption (c).  In the proof of \ref{1f.7} we use the
$\varrho_\alpha$'s in $\boxtimes_1$, that is, we demand that 
$\nu_\alpha(i)$ also codes $\varrho_\alpha(i)$.

Consider the statement
\mn
\begin{enumerate}
\item[$\boxplus$]  For regular $\partial \in (\theta + \kappa^+,\mu)$,
  the set $S$ is not a stationary subset of $\partial$ \when \,:
\begin{enumerate}
\item[$\odot_{\partial,S}$]  $\partial = \text{ cf}(\partial) \in
(\theta + \kappa^+,\mu),\alpha_\varepsilon < \lambda$ for $\varepsilon <
\partial$ with no repetitions and
\newline
$S = \{\zeta < \partial$ \,: 
for some $\xi \in
[\zeta,\partial)$, the set $\{i < \kappa:\nu_{\alpha_\xi}(i) \in
\{\nu_{\alpha_\varepsilon}(i):\varepsilon < \zeta\}\}$ belongs to
$J^+_1\}$.
\end{enumerate}
\end{enumerate}
\mn
\underline{It suffices to prove $\boxplus$}:
\medskip

\noindent
Why?  We prove that $\{\nu_\alpha:\alpha < \lambda\}$ is
$\partial^+$-free by induction on $\partial < \Upsilon$ so let $w
\subseteq \lambda,|w| \le \partial$.  
If $\partial \le \kappa$ just note that $\alpha
\ne \beta \in w \Rightarrow \{i < \kappa:\nu_\alpha(i) = \nu_\beta(i)\}
\in J_1$, if $\partial < \theta$ recall $\boxtimes(\alpha)$ of Claim
\ref{1f.7}.  If $\partial \ge \kappa^+ + \theta$ is singular use
compactness for singulars.  So assume $\partial = \text{ cf}(\partial)
\ge \kappa^+ + \theta$ so by the induction hypothesis \wolog \, $|w|
= \partial$ and let $\langle \alpha_\varepsilon:\varepsilon <
\partial\rangle$ list $w$ and define $S$ as in $\odot_S$ above from $\langle
\alpha_\varepsilon:\varepsilon < \partial\rangle$.  As we are
assuming $\boxplus$, necessarily $S$ is not a stationary subset of
$\partial$ so let $E$ be a club of $\partial$ disjoint to $S$.  Let
$\langle \varepsilon(\iota):\iota < \partial\rangle$ list $E \cup \{0\}$ in
increasing order.  For each $\iota < \theta$ we apply
the induction hypothesis to $w_\iota := \{\alpha_\varepsilon:\varepsilon \in
[\varepsilon(\iota),\varepsilon(\iota +1))\}$ and get the sequence $\langle
s_{\iota,\varepsilon} \in J_1:\varepsilon \in w_\iota\rangle$.

Lastly, for $\varepsilon < \partial$ let $\iota$ be such that
$\varepsilon \in [\varepsilon(\iota),\varepsilon(\iota +1))$ and
$s_\varepsilon = s_{\iota,\varepsilon} \cup \{i <
\kappa:\nu_\varepsilon(i)$ belong to
$\{\nu_{\alpha_\zeta}(i):\zeta < \varepsilon(\iota)\}\}$.
\medskip

\noindent
\underline{Why does $\boxplus$ hold}?

Towards a contradiction, suppose that $\langle
\alpha_\varepsilon:\varepsilon < \partial\rangle,S$ are as in
$\boxdot_{\partial,S}$ and $S$ is a stationary
subset of $\partial = \text{ cf}(\gamma) \in (\theta +
\kappa^+,\Upsilon)$.  Then \wolog \,:
\begin{enumerate}
\item[$(*)_5$]   $(a) \quad$ for some stationary $S_0 \subseteq S$, for
every limit $\zeta \in S_0,\zeta$ can itself serve

\hskip25pt  as the witness
$\xi$ (in fact we can have $S \backslash S_0$ not stationary)
\sn
\item[${{}}$]  $(b) \quad$ for some club $E$ of $\partial$, if
$\varepsilon < \xi$ and $E \cap (\varepsilon,\xi] \ne \emptyset$
then $\{i < \kappa:\nu_{\alpha_\xi}(i) \in$

\hskip25pt $\{\nu_{\alpha_\zeta}(i):\zeta < \varepsilon\}\} \in J_1$.
\end{enumerate}
\mn
[Why?  For clause (a) by renaming.  For clause (b), it suffices to
show that $(\forall \varepsilon < \partial)(f(\varepsilon) <
\partial)$ where for $\varepsilon < \partial,f(\varepsilon)$ is the
minimal ordinal $\gamma \le \partial$ such that if $\xi < \partial$ and
$\{i < \kappa:\nu_{\alpha_\xi}(i) \in \{\nu_{\alpha_\zeta}(i):\zeta <
\varepsilon\}\} \in J^+_1$ then $\xi < \partial$.  

Now, if $\varepsilon < \partial$ and $f(\varepsilon) = \partial$ then
by the third $\bullet$ of $\boxplus_2$ in the proof of \ref{1f.7} it
follows that $u = \{\varrho_{\alpha_\zeta}(i):i < \kappa$ and $\zeta <
\varepsilon\}$ and $\langle \alpha_\zeta:\zeta < \partial\rangle$
witness $\partial \in \text{ issp}_{J_1}(\{\varrho_\alpha:\alpha <
\lambda\})$ so also clause (b) of $(*)_5$ holds indeed.]

Clearly $\delta \in S \Rightarrow \text{ cf}(\delta) \le \kappa$, and because
$(\forall \alpha < \mu)(|\alpha|^\kappa < \mu)$, by the second
$\bullet$ in $\boxtimes_1$ in the proof of \ref{1f.7} 
we know that for $i < \kappa$ the value 
$\nu_\alpha(i)$ determine $\nu_\alpha \rest i$, hence easily \wilog \,
\mn
\begin{enumerate}
\item[$(*)_6$]   if $\delta \in S$ then $\delta \in E$ and 
$\text{cf}(\delta) =\kappa$. 
\end{enumerate}
\mn
Let

\[
S_1 := \{\zeta \in S_0:\{\varepsilon <
\theta:\eta_{\alpha_\zeta}(\varepsilon) \in
\{\eta_{\alpha_j}(\varepsilon):j < \zeta\}\} \text{ belongs to } J^+_2\}.
\]
\bn
\underline{Case A}:  $S_1$ is a stationary subset of $\partial$.

Firstly, assume $\kappa < \theta$. As, see above,
$\zeta \in S_0 \Rightarrow \text{ cf}(\zeta) \le \kappa$ and $\theta
> \kappa \Rightarrow J_2$ is $\kappa^+$-complete, clearly
for each $\zeta \in S_1$, for some $j_\zeta <
\zeta$, the set $\{\varepsilon < \theta:\eta_{\alpha_\zeta}(\varepsilon) \in
\{\eta_{\alpha_j}(\varepsilon):j < j_\zeta\}\}$ belongs to $J^+_2$.  By
Fodor's lemma, for some $j(*)$, the set $S_2 = \{\zeta \in S_1:j_\zeta \le
j(*)\}$ is a stationary subset of $\partial$.  Now
$\{\eta_{\alpha_\zeta}:\zeta \in S_2\}$ witnesses $(< \partial,\partial)
\in \ussp_{J_2}(\lim_\theta({\mat T}))$; 
but this contradicts a demand in case (c) of the assumption of \ref{1f.10}.

Secondly, if $\theta < \kappa$ but recalling $(*)_6$ (see above)
$\zeta \in S_0 \Rightarrow \text{ cf}(\zeta) = \kappa$ and now 
the proof is similar.
\bn
\newline
\underline{Case B}:  $\kappa < \theta$ and $S_1$ is not stationary.

So necessarily $S_0 \backslash S_1$ is a stationary subset of
$\partial$.  By the definition of $S_1$ (and $(*)_5$)
we can find $\bar s^* = \langle s^*_\zeta:\zeta \in (S_0 \backslash
S_1)\rangle$ such that:
\mn
\begin{enumerate}
\item[$(*)_7$]  $(a) \quad s^*_\zeta \in J_2$
\sn
\item[${{}}$]  $(b) \quad$ if $\zeta_1 \ne \zeta_2$ are from $(S_0
\backslash S_1)$ and $\varepsilon \in \theta \backslash s^*_{\zeta_1}
\backslash s^*_{\zeta_2}$, then

\hskip25pt $\eta_{\alpha_{\zeta_1}}(\varepsilon)
\ne \eta_{\alpha_{\zeta_2}}(\varepsilon)$.
\end{enumerate}
\mn
Let $\varepsilon(\zeta) = \text{ min}(\theta \backslash s^*_\zeta)$
for $\zeta \in (S_0 \backslash S_1)$.

So for some stationary $S_2 \subseteq (S_0 \backslash S_1)$,
we have $\zeta \in S_2 \Rightarrow \varepsilon(\zeta) =
\varepsilon(*)$ and so
\mn
\begin{enumerate}
\item[$(*)_8$]   $\langle \eta_{\alpha_\zeta}(\varepsilon(*)):\zeta
\in S_2\rangle$ is without repetitions.
\end{enumerate}
\mn
Now $(*)_2(b)$ in the proof of \ref{1f.7} says that 
$\langle \rho_\gamma:\gamma < \chi\rangle$ is
$(\mu^+,J^+_1)$-free; apply this to the subset
$\{\varrho_{\eta_{\alpha_\zeta}(\varepsilon(*))}:\zeta \in S_2\}$
which has cardinality $\partial < \mu^+$ hence (recall $(*)_8$)
\mn
\begin{enumerate}
\item[$(*)_9$]  some $\langle
s[\eta_{\alpha_\zeta}(\varepsilon(*))]:\zeta \in S_2\rangle$
witnesses that $\langle \rho_{\eta_{\alpha_\zeta}}(\varepsilon(*)):
\zeta \in S_2 \rangle$ is
free, i.e. $s_{\eta_{\alpha_\zeta}(\varepsilon(*))} \in J_1$ for
$\zeta \in S_1$ and $\zeta \ne \xi \in S_2 \wedge i \in \kappa
\backslash s[\eta_{\alpha_\zeta}(\varepsilon(*))] \backslash
s[\eta_{\alpha_\xi}(\varepsilon(*))] \Rightarrow
\varrho_{\eta_{\alpha_\zeta}(\varepsilon(*))}(i) \ne
\varrho_{\eta_{\alpha_\xi}(\varepsilon(*))}(i)$. 
\end{enumerate}
\mn
As $\kappa < \partial$, for some $i(*) < \kappa$,
\mn
\begin{enumerate}
\item[$(*)_{10}$]   the set $S_3 := \{\zeta \in S_2:i(*) 
\notin s[\eta_{\alpha_\zeta}(\varepsilon(*))]\}$ is
a stationary subset of $\partial$.
\end{enumerate}
\mn
Hence
\mn
\begin{enumerate}
\item[$(*)_{11}$]  $\langle \nu_{\alpha_\varepsilon}(i(*)):\varepsilon \in
S_2\rangle$ is a sequence without repetitions.
\end{enumerate}
\mn
By $(*)_6$ we know that $\nu_\alpha(i) = \nu_\beta(i) \Rightarrow 
\nu_\alpha \restriction i = \nu_\beta \restriction i$ for 
$\alpha,\beta < \lambda,i < \kappa$; but by the choice of $S$ we have 
$\zeta \in S_3 \Rightarrow \nu_{\alpha_\varepsilon}(i(*)) \in
\{\nu_{\alpha_\zeta}(i(*)):\zeta < \varepsilon\}$.
However, this contradicts $(*)_{10} + (*)_{11}$.
\end{PROOF}

\begin{claim}
\label{c10}
1) In \ref{1f.7}, $\cF$ satisfies: for $\kappa + \theta < \partial =
\cf(\partial) < \lambda$, we have $\cF$ is
$(\partial^+,\partial,J_1)$-free iff $(< \partial,\partial) \in 
\issp_{J_1}(\cF)$ and there are pairwise distinct $f_\varepsilon \in
\cF$ for $\varepsilon < \partial$ with no repetitions such that for
stationarily many $\delta \in S^\lambda_{\le\kappa},\{i < \kappa:f_\delta(i)
\in \{f_\alpha(i):\alpha < \delta\} \in J^+_1$.

\noindent
2) If in \ref{1f.7} also $\alpha < \mu \Rightarrow |\alpha|^{< \kappa}
< \mu$ then we can replace $S^\lambda_{\le\kappa}$ by $S^\lambda_\kappa$. 
\end{claim}

\begin{PROOF}{\ref{c10}}  
By the proof of the previous claim.
\end{PROOF}
\bigskip

\centerline {$* \qquad * \qquad *$}
\bigskip

In \ref{1f.21}, the case we are most interested in is 
$\mu = \beth_{\omega_1},\kappa = \aleph_1,\theta = \aleph_0$.
\begin{claim}
\label{1f.21}
There is ${\mat F} \subseteq {}^\kappa \mu$ of
cardinality $\lambda$ which is $(\mu^+,J)$-free \when \,:
\mn
\begin{enumerate}
\item[$\circledast$]  $(a) \quad \theta = \,\cf(\theta) < \kappa = 
\cf(\mu) < \mu$
\sn
\item[${{}}$]  $(b) \quad \lambda = \mu^\kappa$
\sn
\item[${{}}$]  $(c) \quad \mu < \chi < \chi^\theta = \lambda$
\sn
\item[${{}}$]  $(d) \quad \alpha < \mu \Rightarrow |\alpha|^\theta < \mu$
\sn
\item[${{}}$]  $(e) \quad J$ is a $\theta^+$-complete ideal on $\kappa$
\sn
\item[${{}}$]  $(f) \quad$ {\rm pp}$_J(\mu) =^+ \lambda$.
\end{enumerate}
\end{claim}

\begin{remark} 
This claim is used in the proof of the theorem \ref{h.7}.
\end{remark}

\begin{PROOF}{\ref{1f.21}}
Let $\langle \mu_i:i < \kappa\rangle$ be increasing with limit $\mu$
such that $(\mu_i)^\theta = \mu_i$ and let $\cd_*:{}^\theta \mu 
\rightarrow \mu$ and $\cd_\varepsilon$ (for $\varepsilon < \theta$) be 
as in the proof of \ref{1f.7}, noting that by clause $(a)$ of the
assumption of the claim ${}^\theta \mu = \cup
\{{}^\theta(\mu_i):i < \kappa\} = \mu$ and let $\mu^-_i = \cup\{\mu_j:j <i\}$.

As $\chi < \pp_J(\mu)$, by \ref{1.3.3}(c), i.e. \cite[Ch.II]{Sh:g} there
is a sequence $\bar \rho = \langle \rho_\gamma:\gamma < \chi\rangle$
of members of ${}^\kappa \mu$ which is $(\mu^+,J)$-free.
Let $\bar \eta = \langle \eta_\alpha:\alpha <
\lambda \rangle$ with $\eta_\alpha \in {}^\theta \chi$ be pairwise distinct.

Without loss of generality, $\rho_\gamma \in \prod\limits_{i < \kappa}
[\mu^-_i,\mu_i)$; we define $\nu_\alpha \in  \prod\limits_{i < \kappa}
\mu_i \subseteq {}^\kappa \mu$ for $\alpha < \lambda$ by
$\nu_\alpha(i) = \cd_*(\langle
\rho_{\eta_\alpha(\varepsilon)}(i):\varepsilon < \theta\rangle)$ for
$i < \kappa$.  We shall prove that $\langle \nu_\alpha:\alpha <
\lambda\rangle$ is as required, i.e. $\langle \nu_\alpha:\alpha <
\lambda\rangle$ is $(\mu^+,J)$-free; this suffices as it implies
$\alpha < \beta < \lambda \Rightarrow \nu_\alpha \ne \nu_\beta$ hence
$\{\nu_\alpha:\alpha < \lambda\} \subseteq {}^\kappa \mu$ has
cardinality $\lambda = \mu^\kappa$ (and is $(\mu^+,J)$-free).

For $w \in [\lambda]^{\le \mu}$, we let $u = \cup\{\Rang(\eta_\alpha):
\alpha \in w\}$, so $u$ is a subset of $\chi$ of cardinality $\le \mu$.

As $\bar\rho = \langle \rho_\alpha:\alpha < \chi\rangle$
is $(\mu^+,J)$-free, there is
$\bar s = \langle s_\gamma:\gamma \in u\rangle$ such that:
\mn
\begin{enumerate}
\item[$\circledast$]   $(\alpha) \quad s_\gamma \in J$ for every $\gamma
\in u$
\sn
\item[${{}}$]  $(\beta) \quad$ if $\gamma_1 \ne \gamma_2 \in u$ and $i
\in \kappa \backslash (s_{\gamma_1} \cup s_{\gamma_2})$, then
$\rho_{\gamma_1}(i) \ne \rho_{\gamma_2}(i)$.
\end{enumerate}
\mn
Now for each $\alpha \in w$, the set $t_\alpha :=
\cup\{s_{\eta_\alpha(\varepsilon)}:\varepsilon < \theta\}$ is the union
of $\le \theta$ members of $J$, but $J$ is a $\theta^+$-complete ideal by
assumption (e), hence $t_\alpha \in J$.

Suppose $\alpha_1 \ne \alpha_2$ are from $w$ and $i \in \kappa
\backslash (t_{\alpha_1} \cup t_{\alpha_2})$.  Can we have $\nu_{\alpha_1}(i)
= \nu_{\alpha_2}(i)$?  If so, then for every $\varepsilon < \theta$, we
have $i \in \kappa \backslash s_{\eta_{\alpha_1}(\varepsilon)} \backslash
s_{\eta_{\alpha_2}(\varepsilon)}$ and $\rho_{\eta_{\alpha_1}(\varepsilon)}(i)
= \rho_{\eta_{\alpha_2(\varepsilon)}}(i)$, hence necessarily
$\eta_{\alpha_1}(\varepsilon) = \eta_{\alpha_2}(\varepsilon)$.  As this
holds for every $\varepsilon < \theta$, we get 
$\eta_{\alpha_1} = \eta_{\alpha_2}$.  This implies $\alpha_1 = \alpha_2$.

So $i \in \kappa \backslash (t_{\alpha_1} \cup t_{\alpha_2})
\wedge \nu_{\alpha_1}(i) = \nu_{\alpha_2}(i) \Rightarrow \alpha_1 =
\alpha_2$.  Thus $\langle \nu_\alpha:\alpha \in w \rangle$ is free, so we
are done.
\end{PROOF}

\begin{conclusion}
\label{1f.22}
If clauses (a)-(f) of \ref{1f.21} hold and $\lambda = \mu^\kappa =
2^\mu$, \underline{then} BB$(\lambda,\mu^+,\lambda,J)$.
\end{conclusion}

\begin{PROOF}{\ref{1f.22}}
  By claim \ref{1f.21} there is ${\mat F} \subseteq
{}^\kappa \mu$ of cardinality $\lambda$ which is $(\mu^+,J)$-free.
By assumption $|{\mat F}| = \mu^\kappa = 2^\mu$ hence by
\ref{2b.98} we get BB$(2^\mu,\mu^+,\chi,J)$ so we are done.
\end{PROOF}

A relative of \ref{1f.21} is
\begin{claim}
\label{1f.23}
There is a $(\mu^+,J_1)$-free ${\mat F}
\subseteq {}^\kappa \mu$ of cardinality $\lambda$ \when \,
\mn
\begin{enumerate}
\item[$\circledast$]  $(a) \quad \sigma < \theta < \kappa =
\cf(\mu) < \mu < \lambda$
\sn
\item[${{}}$]  $(b) \quad (\alpha) \quad J_2$ is a $\sigma^+$-complete 
ideal on $\theta$ and
\sn
\item[${{}}$]  \hskip20pt $(\beta) \quad$ there
are $\lambda$ pairwise $J_2$-distinct members of ${}^\theta \chi$
\sn
\item[${{}}$]  $(c) \quad 2^\kappa < \mu < \chi < \lambda$ and
$2^\kappa < \text{\rm cf}(\lambda)$
\sn
\item[${{}}$]  $(d) \quad  \alpha < \mu \Rightarrow \text{\rm cov}
(|\alpha|,\theta^+,\theta^+,\sigma^+) \le \mu$
\sn
\item[${{}}$]  $(e) \quad J_1$ is a $\theta^+$-complete ideal on
$\kappa$ 
\sn
\item[${{}}$]  $(f) \quad \chi < \text{\rm pp}_{J_1}(\mu)$.
\end{enumerate}
\end{claim}

\begin{PROOF}{\ref{1f.23}}
By clauses (f) and (c) there is an increasing sequence $\langle \lambda_j:j <
\kappa\rangle$ of regular cardinals $\in(2^\kappa,\mu)$
with limit $\mu$ such that $\chi^+ = \tcf(\prod\limits_{i < \kappa} 
\lambda_i,<_{J_1})$ and we let 
$\lambda^-_i = \Sigma\{\lambda_j:j<i\}$ for $i < \kappa$.

By clause (f) and \ref{1.3.3}(c), \wilog \, 
there is a $(\mu^+,J_1)$-free sequence $\langle
\rho_\gamma:\gamma < \chi\rangle$ of members of 
$\prod\limits_{j < \kappa} \lambda_j$.  Let 
${\mat P}_i \subseteq [\lambda_i]^\theta$ be
a set of cardinality $\le \mu$ such that:
\mn
\begin{enumerate}
\item[$(*)_{\cP_i}$]   for every $u \in [\lambda_i]^\theta$, we can 
find $\zeta_u \le \sigma$ and $u_\zeta \in
{\mat P}_i$ for $\zeta < \zeta_u$ such that $u \subseteq
\cup\{u_\zeta:\zeta < \zeta_u\}$.
\end{enumerate}
\mn
Note that ${\mat P}_i$ exists by clause (d) of the assumption.  Let ${\mat P} 
= \cup\{{\mat P}_i:i < \kappa\}$, so that $|{\mat P}| \le \mu,
{\mat P} \subseteq [\mu]^\theta$.

By clause $(b)(\beta)$, let 
$\bar\eta = \langle \eta_\alpha:\alpha < \lambda\rangle$
with $\eta_\alpha \in {}^\theta \chi$ be such that $\alpha < \beta <
\lambda$ implies $\eta_\alpha \ne_{J_2} \eta_\beta$,
i.e. $\{\varepsilon < \theta:\eta_\alpha(\varepsilon) =
\eta_\beta(\varepsilon)\} \in J_2$.

Lastly, for each
$\alpha < \lambda$, for each $i < \kappa$, we know that
$\{\rho_{\eta_\alpha(\varepsilon)}(i):\varepsilon < \theta\} \in
[\lambda_i]^{\le \theta}$, hence we can find a sequence $\langle
u^i_{\alpha,\zeta}:\zeta < \sigma\rangle$ of members of ${\mat P}_i$
such that $\{\rho_{\eta_\alpha(\varepsilon)}(i):\varepsilon < \theta\}
\subseteq \cup\{u^i_{\alpha,\zeta}:\zeta < \sigma\}$.

For each $\alpha < \lambda$ and $i < \kappa$, as $J_2$ is a 
$\sigma^+$-complete ideal on $\theta$, 
for some $\zeta_{\alpha,i} < \sigma$, the set ${\mat W}_{\alpha,i} 
:= \{\varepsilon < \theta:\rho_{\eta_\alpha(\varepsilon)}(i) \in
u^i_{\alpha,\zeta_{\alpha,i}}\}$ belongs to $J^+_2$.  
Let ${\bold x}_\alpha :=
\{(i,\zeta_{\alpha,i},s_{\eta_\alpha(\varepsilon)}(i) \cap
u^i_{\alpha,s_{\alpha,i}}):i < \kappa$ and $\varepsilon \in {\mat
W}_{\alpha,i} \subseteq \theta\}$.

The number of possible ${\bold x}_\alpha$ is at most
$\le 2^\kappa$, but $2^\kappa < \cf(\lambda)$ by clause (c) of
the assumption. As we can replace $\langle \eta_\alpha:\alpha <
\lambda\rangle$ by $\langle \eta_\alpha:\alpha \in v\rangle$ for any
$v \in [\lambda]^\lambda$, \wilog \, for some ${\bold x} =
\{(i,\zeta_{i,\varepsilon},\gamma_{i,\varepsilon}):i <\kappa$ 
and $\varepsilon \in {\mat W}_i\}$, we have:
\mn
\begin{enumerate}
\item[$(*)_0$]  ${\bold x}_\alpha = {\bold x}$ for every $\alpha < \lambda$.
\end{enumerate}
\mn
For $\alpha < \lambda$ let $\nu_\alpha \in {}^\kappa {\mat P}$ be
defined by:
\mn
\begin{enumerate}
\item[$\odot_1$]  $\nu_\alpha(i) = u^i_{\alpha,\zeta_{\alpha,i}}$.
\end{enumerate}
\mn
Clearly it suffices to show that:
\mn
\begin{enumerate}
\item[$\odot_2$]  $\bar\nu = \langle \nu_\alpha:\alpha <
\lambda\rangle$ exemplifies the conclusion.
\end{enumerate}
\mn
This follows by $(*)_1,(*)_2,(*)_3$ below:
\mn
\begin{enumerate}
\item[$(*)_1$]   $\nu_\alpha \in {}^\kappa {\mat P}$ and $|{\mat
P}| \le \mu$.
\end{enumerate}
\mn
[Why?  Obviously.]
\mn
\begin{enumerate}
\item[$(*)_2$]   $\nu_\alpha \ne \nu_\beta$ for $\alpha < \beta <
\lambda$.
\end{enumerate}
\mn
[Why?  By the proof of $(*)_3$ using $w = \{\alpha,\beta\}$.]
\mn
\begin{enumerate}
\item[$(*)_3$]   $\{\nu_\alpha:\alpha < \lambda\}$ is
$(\mu^+,J_1)$-free.
\end{enumerate}
\mn
[Why?  Let $w \in [\lambda]^{\le \mu}$; we shall prove that
$\{\nu_\alpha:\alpha \in w\}$ is $J_1$-free.  Now
$u := \cup\{\text{Rang}(\eta_\alpha):\alpha \in w\} \in
[\chi]^{\le \mu}$, recalling $\varepsilon < \theta \Rightarrow 
\eta_\alpha(\varepsilon) < \chi$.  By the assumption on
$\{\rho_\gamma:\gamma < \chi\}$, we can find a sequence $\bar s$ such
that:
\mn
\begin{enumerate}
\item[$(\alpha)$]  $\bar s = \langle s_\gamma:\gamma \in u\rangle \in
{}^u(J_1)$
\sn
\item[$(\beta)$]  if $\gamma_1 \ne \gamma_2$ and $\gamma_1 \in
u,\gamma_2 \in u$ and $i \in \kappa \backslash s_{\gamma_1}
\backslash s_{\gamma_2}$, then $\rho_{\gamma_1}(i) \ne
\rho_{\gamma_2}(i)$.
\end{enumerate}
\mn
For each $\alpha \in w$, let $t_\alpha := 
\cup\{s_{\eta_\alpha(\varepsilon)}:\varepsilon < \theta\}$. 
Now $t_\alpha$ is the union of $\le \theta$ members of 
$J_1$ which is a $\theta^+$-complete ideal (by (e)), so
$t_\alpha \in J_1$.  It suffices to prove that $\langle
t_\alpha:\alpha \in w\rangle$ witnesses $\{\nu_\alpha:\alpha \in w\}$ 
is $J_1$-free, so, by the previous sentence, it suffices to
prove:
\mn
\begin{enumerate}
\item[$(*)'_3$]   if $\alpha_1 \ne \alpha_2$ are from $w$ and $i
\in \kappa \backslash t_{\alpha_1} \backslash t_{\alpha_2}$, then 
$\nu_{\alpha_1}(i) \ne \nu_{\alpha_2}(i)$.
\end{enumerate}
\mn
Toward a contradiction assume that $\nu_{\alpha_1}(i) = \nu_{\alpha_2}(i)$.
Recalling the choice of $\nu_\alpha$, i.e. $\odot_1$, this means that
$u^i_{\alpha_1,\zeta_{\alpha_1,i}} = u^i_{\alpha_2,\zeta_{\alpha_2,i}}$.

As $\bold x_{\alpha_1} = \bold x = \bold x_{\alpha_2}$, see condition $(*)_0$,
clearly $\cW_{\alpha_1,i} = \cW_{\alpha_2,i}$ but we are assuming 
$u^i_{\alpha_1,\zeta_{\alpha_1,i}} = u^i_{\alpha_2,\zeta_{\alpha_2,i}}$ so
by the definition of $\bold x_{\alpha_1},\bold x_{\alpha_2}$ we have
$\varepsilon \in \cW_{\alpha_1} = \cW_{\alpha_2} \Rightarrow
\rho_{\eta_{\alpha_1}(\varepsilon)}(i) = \rho_{\eta_{\alpha_2}(\varepsilon)}(i)
 \Rightarrow \eta_{\alpha_1}(\varepsilon) = \eta_{\alpha_2}(\varepsilon)$ so
$\{\varepsilon < \theta:\eta_{\alpha_1}(\varepsilon) =
\eta_{\alpha_2}(\varepsilon)\} \supseteq \cW_{\alpha_1}$ but
$\cW_{\alpha_1,i} \in J^+_2$ by the choice of
$\zeta_{\alpha_1,i}$.  So we get
$\neg(\eta_{\alpha_1} \ne_{J_2} \eta_{\alpha_2})$,
contradicting the choice of $\langle \eta_\alpha:\alpha <
\lambda\rangle$.]

So $(*)'_3$ holds, and hence $(*)_3$ holds. 
Therefore $\odot_2$ holds, so we are done.
\end{PROOF}

\begin{obs}
\label{1f.28}
1) Assume $\lambda > \mu > \kappa = \text{\rm cf}(\mu)$ 
and $\alpha < \mu \Rightarrow
|\alpha|^\sigma < \mu$, and $\theta = \sup\{\theta_i:i < \sigma\}$ 
and for each $i < \sigma$, there is a $\theta_i$-free 
${\mat F} \subseteq {}^\kappa \mu$ of 
cardinality $\lambda$. \underline{Then} there is a
$\theta$-free ${\mat F} \subseteq {}^\kappa \mu$ of cardinality
$\lambda$.

\noindent
1A) If $\kappa = \sigma$ \then \, $\alpha < \mu \Rightarrow
|\alpha|^{< \sigma} < \mu$ suffices.

\noindent
2) If ${\mat F} \subseteq {}^\kappa \mu$ is $\theta$-free, \underline{then}
there is a normal $\theta$-free ${\mat F}' \subseteq {}^\kappa \mu$ of
cardinality $|{\mat F}|$ - see Definition \ref{1.3.14}(5).

\noindent
3) If $J$ is an ideal on $\kappa,
\delta < \lambda$ and $\langle \lambda_i:i < \delta\rangle$ is
   increasing with limit $\lambda$ and there are $(\theta,J)$-free $\cF_i
   \subseteq {}^\kappa \mu$ of cardinality $\lambda_i$ for $i < \delta$
   \then \, there is a $(\theta,\cF)$-free $\cF \subseteq {}^\kappa \mu$
   of cardinality $\lambda$. 

\noindent
3A) In part (3), if $f \in \cF_i \wedge \varepsilon < \kappa,f(\varepsilon) \in
\cU_\varepsilon \subseteq \mu$ and $\cU_\varepsilon$ is infinite for
$\varepsilon < \kappa$ \then \, \wilog \, $f \in \cF \wedge
\varepsilon < \kappa \Rightarrow f(\varepsilon) \in \cU_\varepsilon$.

\noindent
4) We can in parts (3), (3A) add ``$(\cF_i,<_J)$ of order type
   $\lambda_i$" and change the conclusion to ``$\cF \subseteq
   {}^\kappa(\mu \times \mu),(\cF,\prec_J)$ of order type $\lambda$
   (and still is $(\theta,J)$-free)".

\noindent
5) Similarly to part (4) but $\cF \subseteq {}^\kappa \mu$
if $2^\kappa < \mu,\cf(\delta)$ recalling Definition \ref{1.3.14}(4).
\end{obs}

\begin{PROOF}{\ref{1f.28}}
1) By coding (separating the proof according to whether
$\sigma < \kappa$ or $\sigma \ge \kappa$).

In more detail, \wilog, \, $i < \sigma \Rightarrow \theta_i
 < \theta$;  let ${\mat F}_i \subseteq {}^\kappa \mu$ be $\theta_i$-free
of cardinality $\lambda$, let $\langle \eta^i_\alpha:\alpha <
\lambda\rangle$ list ${\mat F}_i$ with no repetitions, and let
cd:$\bigcup\limits_{\alpha < \mu} {}^\sigma \alpha \rightarrow 
\mu$ be a one-to-one mapping.
\bn
\newline
\underline{Case 1}:  $\sigma < \kappa$.

For $\alpha < \lambda$ and $\varepsilon < \kappa$ the sequence
$\langle \eta^i_\alpha(\varepsilon):i < \sigma\rangle$ belongs to
${}^\sigma \mu$ hence by the present case to $\cup\{{}^\sigma
\beta:\beta < \alpha\}$.

Let $\eta_\alpha := \langle \cd(\langle
\eta^i_\alpha(\varepsilon):i < \sigma\rangle):\varepsilon <
\kappa\rangle$, so $\eta_\alpha \in {}^\kappa \mu$, and clearly
$\langle \eta_\alpha:\alpha < \lambda\rangle$ is as required.
\bn
\newline
\underline{Case 2}:  $\sigma \ge \kappa$.

Let $\langle \mu_\varepsilon:\varepsilon < \kappa\rangle$ be 
increasing with limit $\mu$.  For $\varepsilon < \kappa$ let $\bold
h_\varepsilon:\sigma \times \varepsilon \rightarrow \sigma$ be
one-to-one and onto.

We define $\eta_\alpha \in {}^\kappa \mu$ as follows:
\mn
\begin{enumerate}
\item[$\bullet$]  for $\varepsilon < \kappa$ we let
  $\eta_\alpha(\varepsilon) = \cd(\langle \gamma_{\alpha,i}:i <
  \sigma\rangle)$ where
\sn
\item[$\bullet$]  if $j < \sigma$ and $\zeta < \varepsilon$ and $i =
  \bold h_\varepsilon(j,\zeta)$, then $\gamma_{\alpha,i} =
  \min\{\eta^j_\alpha(\zeta),\mu_\varepsilon\}$.
\end{enumerate}
\mn
Now first for $\varepsilon < \kappa,\eta_\alpha(\varepsilon)$ is well
defined $(< \mu)$ as $\langle \gamma_{\alpha,i}:i < \sigma\rangle \in
{}^\sigma(\mu_\varepsilon) \subseteq \dom(\cd)$; so indeed
$\eta_\alpha \in {}^\kappa \mu$.  Second, $\{\eta_\alpha:\alpha <
\lambda\}$ is $\theta$-free because if $w \subseteq \lambda,|w| <
\theta$ then for some $i < \sigma$ we have $|w| < \theta_i$, hence we
can find a sequence $\langle \zeta_\alpha:\alpha \in w\rangle$ of
ordinals $< \kappa$ such that:
\mn
\begin{enumerate}
\item[$\bullet$]   $\alpha \in w \wedge \beta \in w \wedge \varepsilon 
< \kappa \wedge \varepsilon \ge \zeta_\alpha \wedge \varepsilon \ge
\zeta_\beta \Rightarrow \eta^i_\alpha(\varepsilon) \ne
\eta^i_\beta(\varepsilon)$.
\end{enumerate}
\mn
Let $\xi_\alpha = \min\{\xi:\xi \ge \zeta_\alpha$ and
$\eta^i_\alpha(\zeta_\alpha) < \mu_\xi\}$.  Then, easily
\mn
\begin{enumerate}
\item[$\bullet$]  $\alpha \in w \wedge \beta \in w \wedge \varepsilon 
< \kappa \wedge \varepsilon \ge \xi_\alpha \wedge \varepsilon \ge
\xi_\beta \Rightarrow \eta_\alpha(\varepsilon) \ne
\alpha_\beta(\varepsilon)$.
\end{enumerate}
\mn
So we are done.

\noindent
1A) The proof is similar using $\eta_\alpha =
\langle\text{cd}(\eta^i_\alpha(\varepsilon):i \le
\varepsilon):\varepsilon < \kappa\rangle$ for an appropriate function
cd.  This is all right because $\alpha < \mu \Rightarrow |\alpha|^{< \sigma} < \mu$;
 actually $\alpha < \mu \Rightarrow |\alpha|^{< \sigma} \le \mu$ suffices.

\noindent
2) Easy.

\noindent
3) Let $i(*) = \min\{i:\delta \le \lambda_i\}$ and let
$\lambda^-_i = \cup\{\lambda_j:j < i\}$ for $i < \delta$, further let
$\langle f^i_\alpha:\alpha < \lambda_i\rangle$ list $\cF_i$ with no
repetitions, for $\varepsilon < \kappa$ 
let $\cd_\varepsilon:\mu \times \mu \rightarrow \mu$ be
one to one and for $\alpha < \lambda$ let
$f_\alpha \in {}^\kappa \mu$ be defined by: if $\alpha \in
[\lambda^-_i,\lambda_i)$ and $\varepsilon < \kappa$ then
$f'_\alpha(\varepsilon) = \cd_\varepsilon(f^i_\alpha(\varepsilon),
f^{i(*)}_i(\varepsilon))$.  One can now check that this works.

\noindent
3A) Similarly but add: $\cd_\varepsilon$ maps $\cU_\varepsilon \times
\cU_\varepsilon$ into $\cU_\varepsilon$.

\noindent
4) As we weaken the conclusion to ``there is a $<_J$-increasing
sequence of length $\lambda$ in ${}^\kappa(\mu \times \mu)$", the proof
of part (3) suffices if we add
\mn
\begin{enumerate}
\item[$\oplus$]   $\cd_\varepsilon(\alpha_1,\alpha_2) < 
\cd_\varepsilon(\alpha'_1,\alpha'_2)$ \Iff \, $(\alpha_2 <
\alpha'_2) \vee (\alpha_2 = \alpha'_2 \wedge \alpha_1 < \alpha'_1)$
\end{enumerate}
\mn
5) \Wilog \, $\lambda < \mu$ and $\delta = \cf(\delta)$.

\Wilog \, each $\lambda_i$ is regular and (even $> \mu$ and also
$\lambda_0 > \delta$).  For each $i < \delta$ let $\bar f^i = \langle
f^i_\alpha:\alpha < \lambda_i \rangle$ be a $<_J$-increasing sequence of
members of ${}^\kappa \mu$, in the role of $\cF_i$.  Let $\langle
\mu_\varepsilon:\varepsilon < \kappa\rangle$ be an increasing sequence
of regular cardinals $> \kappa$ with limit $\mu$ and for $i <
\delta,\alpha < \lambda_i$ let $g^i_\alpha:\kappa \rightarrow \kappa$
be defined by: for $\varepsilon < \kappa$ we let 
$g^i_\alpha(\varepsilon) = \min\{\zeta < \mu:f^i_\alpha(\varepsilon) <
\mu_\zeta\}$.  Hence $\{g_\alpha:\alpha < \lambda_i\} \subseteq {}^\kappa
\kappa$ has cardinality $\le 2^\kappa$ which is $< \mu < \lambda_i =
\cf(\lambda_i)$, so for some $g_i \in {}^\kappa\kappa$ the set
$\{\alpha < \lambda_i:g^i_\alpha = g_i\}$ is unbounded in $\lambda_i$.
Hence \wilog \, $i < \sigma \wedge \alpha < \lambda_2 \Rightarrow
g^i_\alpha = g_i$.

Also we can replace $\langle (\lambda_i,\bar f^i):i < \delta\rangle$
by its restriction to any $u \subseteq \delta$ which is unbounded in
$\delta$.  Hence \wilog \, $\langle g_i:i < \delta\rangle$ is constant
or with no repetitions.  The latter is impossible as $\cf(\delta) >
2^\kappa$.  Now we can just use the proof of part (3) using $\oplus$
from above. 
\end{PROOF}

\begin{obs}
\label{1f.31}
There is a $\sup\{\theta_i:i < i(*)\}$-free 
${\mat F} \subseteq {}^\kappa \mu$ of cardinality
$2^\mu$ \when \,:
\mn
\begin{enumerate}
\item[$(a)$]   $\mu \in {\bold C}_\kappa$
\sn
\item[$(b)$] for  each $i < i(*)$ at least one of the following holds:
\sn
\begin{enumerate}
\item[$(\alpha)$]  for some $\chi,\theta_i < \mu < \chi <
\lambda$ and $\chi^{<\theta_i>_{\text{\rm tr}}} = \lambda$ (and the
supremum is attained)
\sn
\item[$(\beta)$]  $\theta_i = \mu^+$ and for some $\chi$ and
$\sigma = \cf(\sigma) < \kappa$ we have 
$\mu < \chi < \lambda$ and $\chi^\sigma = \lambda$
\sn
\item[$(\gamma)$]   for some $\chi,\theta_i < \mu < \chi <
  \lambda,\kappa \ne \cf(\chi) < \mu$ and $\pp_{J^{\bd}_\kappa}(\chi)
  =^+ \lambda$.
\end{enumerate}
\end{enumerate}
\end{obs}

\begin{PROOF}{\ref{1f.31}}
Clearly $i < i(*) \Rightarrow \theta_i \le \mu^+$. 
Without loss of generality, $i(*) < \mu$.

\noindent
[Why?  Clearly we can replace $\langle \theta_i:i < i(*)\rangle$ by
$\langle \theta_i:i \in u \rangle$ \when \, $u \subseteq i(*)$ and
$\sup\{\theta_i:i < i(*)\} = \sup\{\theta_i:i \in u\}$, so \wolog \,
$\langle \theta_i:i < i(*)\rangle$ has no
repetitions, and so $i(*) \le \mu +1$, and if $i(*) \ge \mu$, 
we can find $u$ as above of cardinality $< \mu$.]

If for every $i < i(*)$ clause $(\alpha)$ or clause 
$(\gamma)$ of (b) of the assumption
holds \then \, by \ref{1f.7} or \ref{2b.111} there is a
$\theta_i$-free $\cF_i \subseteq {}^\kappa \mu$ of cardinality 
$\lambda$ for each $i < i(*)$  and by \ref{1f.28}(1) the conclusion
holds.  It holds by \ref{1f.21} if $(\beta)$ of (b) applies 
for some $i<i(*)$.
\end{PROOF}

\begin{claim}
\label{2b.91}
If $\mu \in {\bold C}_\kappa$ and $\lambda = 2^\mu 
= \chi^+$ and $\chi$ is regular or just $\cf([\chi]^{\le
  \mu},\subseteq) = \chi$ \then \,:
\mn
\begin{enumerate}
\item[$(a)$]   there is a $\mu^+$-free ${\mat F} \subseteq {}^\kappa \mu$ of 
cardinality $2^\mu = \mu^\kappa$
\end{enumerate}
\mn
 hence
\mn
\begin{enumerate}
\item[$(b)$]  {\rm BB}$(\lambda,\mu^+,\theta,\kappa)$ for every $\theta
< \mu$.
\end{enumerate}
\end{claim}

\begin{remark}
This is actually as in \cite[Ch.II,6.5(3),pg.100]{Sh:g} and the no-hole claim.
\end{remark}

\begin{PROOF}{\ref{2b.91}}  
By Definition \ref{1.3.1} there is an ideal $J$ on $\kappa$ and a
sequence $\langle \lambda_i:i < \kappa\rangle$ of regular cardinals $<
\mu$ such that $\lambda = \text{ tcf}(\prod\limits_{i < \kappa}
\lambda_i,<_J)$.  So there is a $<_J$-increasing cofinal sequence
$\langle f_\alpha:\alpha < \lambda\rangle$ of members of
$\prod\limits_{i < \kappa} \lambda_i$.  Let $\bar e'_\varepsilon =
\langle e_{\varepsilon,\alpha}:\alpha < \lambda\rangle$ for
$\varepsilon < \chi$ be as in \ref{h.11}, that is, if $\chi$ is
regular then we apply clause (A) of \ref{h.11} and if $\cf([\chi]^{\le
  \mu},\subseteq) = \chi$, then we apply clause (B) of \ref{h.11}.

Now by induction on $\alpha < \lambda$ we choose $\bar g_\alpha =
\langle g_{\varepsilon,\alpha}:\varepsilon < \chi\rangle$ and
$f^*_\alpha$ such that
\mn
\begin{enumerate}
\item[$\boxplus_2$]  $(a) \quad g_{\varepsilon,\alpha} \in
\prod\limits_{i < \kappa} \lambda_i$
\sn
\item[${{}}$]  $(b) \quad f^*_\alpha \in \prod\limits_{i < \kappa}
\lambda_i$
\sn
\item[${{}}$]  $(c) \quad g_{\varepsilon,\alpha} <_J f^*_\alpha$
\sn
\item[${{}}$]  $(d) \quad f^*_\gamma <_J g_{\varepsilon,\alpha}$ if
$\gamma < \alpha$
\sn
\item[${{}}$]  $(e) \quad g_{\varepsilon,\alpha}(i) >
\sup\{f^*_\beta(i),g_{\varepsilon,\beta}(i):
\beta \in e_{\varepsilon,\alpha}\}$ when $\lambda_i >
|e_{\varepsilon,\alpha}|$.
\end{enumerate}
\mn
As $(\prod\limits_{i < \kappa} \lambda_i,<_J)$ is $\lambda$-directed
we can carry out this definition.  In more detail, at stage $\alpha$,
first we can choose $f'_\alpha \in \prod\limits_{i < \kappa}
\lambda_i$ such that $\beta < \alpha \Rightarrow f_\beta <_J
f'_\alpha$ because $\lambda > |\{f_\beta:\beta < \alpha\}$.  Second, for
$\varepsilon < \chi$ we choose $g_{\varepsilon,\alpha} \in
\prod\limits_{i < \kappa} \lambda_i$ such that $\lambda_i >
|e_{\varepsilon,\alpha}| \Rightarrow g_{\varepsilon,\alpha}(i) =
\sup(\{f^*_\beta(i),g_{\varepsilon,\beta}(i);\beta \in
e_{\varepsilon,\alpha}\} \cup \{f'_\alpha(i) +1\})$.  Third, choose
$f_\alpha \in \prod\limits_{i <\kappa} \lambda_i$ such that
$\varepsilon < \chi \Rightarrow g_{\varepsilon,\alpha} <_J f_\alpha$
again possible as we have $< \lambda$ demands.

Now we can prove that for any $u \subseteq \lambda$ of cardinality
 $\le \mu$ the sequence $\langle f^*_\alpha:\alpha \in u\rangle$
is $J$-free (see \ref{1.3.14}(4)) by induction on $\otp(u)$, 
as in the proof of the no-hole claim, actually \cite[Ch.II,1.5A]{Sh:g}.
\end{PROOF}

\begin{remark}
1) Note that \ref{a35} is quoted in $\boxdot_3$ of \S0 in order to
show $\odot_{3.1}$, but we could also use \ref{2b.91}.

\noindent
2) How much partial square on $\lambda$ suffices in \ref{2b.91}?  One 
for cofinality $\ge \kappa$ where the ideal $J$ is $J^{\bd}_\kappa$ or just
$\kappa$-complete (which is all right).

\noindent
3) We may consider a parallel of \ref{2b.91} when $\chi$ is not as there.
So assume $\mu \in \bold C_\kappa,\lambda = 2^\mu = \chi^+$ and
   $\chi$ is singular and $\cf([\chi]^{\le \mu},\subseteq) = \lambda$.
\mn
\begin{enumerate}
\item[$(A)$]  Is there $\cf(\chi)$-free $\cF \subseteq {}^\kappa \mu$
of cardinality $\lambda$?
\end{enumerate}
\mn
4) If for some $\mu_1,\mu < \mu_1 < \chi$ and
$\cov(\chi,\mu^+_1,\mu^+_1,2) = \chi$, \then \, there is a
$\cf(\chi)$-free $\cF \subseteq {}^\kappa \mu$ of cardinality $\kappa$.

[Why?  We apply \ref{h.11}(B) with $\lambda,\mu_0$ here standing for
  $\lambda,\chi$  there getting $\langle
  e^1_{\varepsilon,\alpha}:\alpha < \lambda,\varepsilon <
  \chi\rangle$, so $\otp(e^1_{\varepsilon,\alpha}) < \mu^+_0 <
  \lambda$.  Let $\langle e^2_i:i < \mu^+_0\rangle$ be such that
  $e^2_i$ is a closed unbounded subset of $i$ of order type $\cf(i)$
  for each $i < \mu^+_0$.  Now let $\bar e = \langle
  e_{i,\varepsilon,\alpha}:\alpha < \lambda,\varepsilon < \chi,i <
  \mu^+_0\rangle$ be defined by $e_{i,\varepsilon,\alpha} = \{\beta
  \in e_{\varepsilon,\alpha}:\otp(e_{\varepsilon,\alpha}) \in
  e^2_i\}$.  So $\bar e$ is as required except that we use
  $(i,\varepsilon) \in \chi \times \mu^+_0$ instead of $\varepsilon <
  \chi$ but as $\chi \times \mu^+_0$ has cardinality $\chi$ this is
  all right.]
\end{remark}

\noindent
Now a variant of \ref{1f.7} is:
\begin{claim}
\label{1f.49}
If $\circledast$ holds, \then \, there is $\cF$ such that 
$\boxtimes$ holds where:
\mn
\begin{enumerate}
\item[$\boxtimes$]   $(\alpha) \quad {\mat F} \subseteq {}^\kappa \mu$
\sn
\item[${{}}$]  $(\beta) \quad |{\mat F}| = \lambda$
\sn
\item[${{}}$]  $(\gamma) \quad {\mat F}$ is $(\theta,J_1)$-free
\sn
\item[$\circledast$]  $(a) \quad \mu < \chi < \lambda$
\sn
\item[${{}}$]  $(b) \quad \kappa = \cf(\mu)$
\sn
\item[${{}}$]  $(c) \quad \theta$ is regular
\sn
\item[${{}}$]  $(d) \quad \sigma < \kappa < \theta < \mu$
\sn
\item[${{}}$]  $(e) \quad J_1$ is a $\sigma^+$-complete ideal on $\kappa$
\sn
\item[${{}}$]  $(f) \quad$ if $\alpha < \mu$, then
{\rm cov}$(|\alpha|,\theta^+,\theta^+,\sigma^+) \le \mu$

\hskip25pt \underline{or} just
\sn
\item[${{}}$]  $(f)^- \quad$ if $\alpha < \mu$, then ${\bold
U}_{J_2}(|\alpha|) \le \mu$, see Definition \ref{d9}
\sn
\item[${{}}$]  $(g) \quad$ there is a set of $\lambda$ pairwise
$J_2$-distinct members of ${}^\theta \chi$
\sn
\item[${{}}$]  $(h) \quad \pp_{J_1}(\mu)^{\sigma^+}$
\sn
\item[${{}}$]  $(i) \quad J_1$ is $\theta^+$-complete
\sn
\item[${{}}$]  $(j) \quad 2^\theta < \mu$
\end{enumerate}
\end{claim}

\begin{PROOF}{\ref{1f.49}}
Combine the proofs of \ref{1f.7} and \ref{1f.23}. 
\end{PROOF}

\begin{claim}
\label{1f.51}
In \ref{1f.49}:

\noindent
1) If in $\circledast,\partial \ge \theta$ clause $(g)$ is 
exemplified by ${\mat F}_2 \subseteq {}^\theta \chi$ which is 
$(\partial,J_2)$-free, $\partial < \mu$, \then \, ${\mat F}$ 
is $(\partial,\theta^+,J_2)$-free.

\noindent
2) If ${\mat F}' \subseteq {\mat F}$ has cardinality $> \theta$, \then
\, $\cup\{\Rang(\nu):\nu \in {\mat F}\}$ has cardinality $\ge \theta$.

\noindent
3) Clauses (f) + (e) from $\circledast$ implies clause (f)$^-$; in
   fact clause (e), ``$J_2$ is $\sigma^+$-complete", is needed only
   for this.

\noindent
4) We can in $\circledast$ weaken (also in part (1)) clause (h) to
\mn
\begin{enumerate}
\item[$(h)'$]  there is a $\partial$-free $\cF \subseteq {}^\kappa
\mu$ of cardinality $\lambda$.
\end{enumerate}
\end{claim}

\begin{PROOF}{\ref{1f.51}}
We leave the proof to the reader.
\end{PROOF}

\begin{claim}
\label{1f.53} 
Assume $\mu \in {\bold C}_\kappa, J$ is a
$\kappa$-complete ideal on $\kappa$ and there is no
$(\kappa^{+\omega},J)$-free ${\mat F} \subseteq {}^\kappa
\mu$ of cardinality $\lambda := 2^\mu$. 
\Then \, the set $\Theta = \{\theta:\theta = \cf(\theta) < \mu,
\theta \ne \kappa$ and for some witness $(\chi,I)$ we have $I$ 
a $\theta$-complete ideal
on $\theta,\chi \in (\mu,\lambda)$ of cofinality $\theta$ and 
$\pp_J(\chi) =^+ \lambda$ for some
$\theta$-complete ideal $J$ on $\theta\}$ is empty, or a singleton $>
\kappa$ or of the form $\{\theta,\theta^+\},\theta > \theta$.
\end{claim}

\begin{remark}
\label{1f.54}

This is intended to help in \S4 in dealing with $R$-modules when $R$
has at least three members together with \ref{h18}, \ref{h.9d}, \ref{5e.8}.
\end{remark}

\begin{PROOF}{\ref{1f.53}}
Note
\mn
\begin{enumerate}
\item[$(*)_0$]  \wilog \, $\lambda$ is regular.
\end{enumerate}
\mn
[Why?  By \ref{1f.28}(3).]
\mn
\begin{enumerate}
\item[$(*)_1$]  if $\theta \in \Theta$ then $\theta > \kappa$.
\end{enumerate}
\mn
[Why?  Let $(\chi,J)$ witness $\theta \in \Theta$, now by \ref{1f.21}
 we get a contradiction to the assumption ``there is no $(\kappa^{+
   \omega},J)$-free $\cF  \subseteq {}^\kappa \mu$ of cardinality $\lambda$".] 

Let $(\theta_1,\chi_1,J_1)$ be such that
\mn
\begin{enumerate}
\item[$(*)_2$]   $\theta_1 \in \Theta$ and 
$(\chi_1,J_1)$ is a witness for $\theta_1 \in \Theta$ and $\chi_1$ is
minimal under these conditions (even varying $\theta_1$).
\end{enumerate}
\mn
If $\theta \in \Theta$ by the choice of $\chi_1$ as minimal, by
\cite[Ch.II,5.4]{Sh:g} we have:
\mn
\begin{enumerate}
\item[$(*)_3$]  $\alpha < \chi_1 \Rightarrow
\text{ cov}(|\alpha|,\mu^+,\mu^+,\kappa^+) < \chi_1$.
\end{enumerate}
\mn
If $\Theta = \{\theta_1\}$ or $\Theta = \{\theta_1,\theta^+_1\}$ or
$\theta^+_2 = \theta_1 \wedge \Theta = \{\theta_1,\theta_2\}$, we
are done; otherwise let $(\theta_2,\chi_2,J_2)$ be such that
\mn
\begin{enumerate}
\item[$(*)_4$]   $\theta_2 \in \Theta \backslash
  \{\theta_1,\theta^+_1\} \wedge \theta_1 \ne \theta^+_2$ and $(\chi_2,J_2)$ 
witness that $\theta_2 \in \Theta$, and $\chi_2$ is minimal under these
requirements.
\end{enumerate}
\mn
Now
\mn
\begin{enumerate}
\item[$(*)_5$]  there is a $(\theta^{++}_1 + \theta_2,J_1)$-free set ${\mat F}
\subseteq {}^{\theta_1}(\chi_1)$ of pairwise $J_1$-distinct elements
of cardinality $\lambda$.
\end{enumerate}
\mn
Why?  \underline{Case 1}: $\theta_2 > \theta_1$

So necessarily $\theta_2 > \theta^+_1$ by $(*)_4$, 
hence such an ${\mat F}$ exists by \ref{1f.49} with
$\lambda,\chi_1,\chi_2,\kappa,\theta_1,\theta_2,J_1,J_2$ here standing for
$\lambda,\mu,\chi,\sigma,\kappa,\theta,J_1,J_2$ there. 

Clauses (a),(b),(c) are obvious.  Why 
does clause (d) from \ref{1f.49} hold?  It means $``\kappa <
 \theta_1 < \theta_2 < \chi_1"$ and these inequalities hold  
because, first $\kappa < \theta_1$ holds by
 $(*)_1$, second $\theta_1 < \theta_2$ holds by the present case
 assumption, and third ``$\theta_2 < \mu < \chi_1$" 
holds by $(*)_2$.

Clause (e) of \ref{1f.49} means ``$J_1,J_2$ are $\kappa^+$-complete"
which hold as $\theta_1,\theta_2 > \kappa$ by $(*)_1$ and $J_\ell$ is
$\theta^+_\ell$-complete by the definition of $\Theta$ because
$(\chi_\ell,J_\ell)$ witness $\theta_\ell \in \Theta$ by $(*)_2 + (*)_4$.  

Clause (f) of \ref{1f.23} means here $\alpha < \chi_1 \Rightarrow
\cov(|\alpha|,\theta^+_2,\theta^+_2,\kappa^+) \le \chi_1$ which holds
by $(*)_3$.

Lastly, clause (g) of \ref{1f.49} means ``there is a set of $\lambda$
pairwise $J_2$-distinct members of ${}^{\theta_2}(\chi_2)"$ which
holds as $(J_2,\chi_2)$ witnesses $\theta_2 \in \Theta$.

The conclusion of \ref{1f.49} gives a family 
$\cF \subseteq {}^{\theta_1}(\chi_1)$ of
cardinality $\lambda$ which is $(\theta_2,J_1)$-free, but $\theta_2
\ge \theta_1$ by ``First", and $\theta_2 \ne \theta^+_2$ by $(*)_4$ so
we are done.

\noindent
\underline{Case 2}:  $\theta_2 \le \theta_1$

Again by $(*)_4,\theta^+_2 < \theta_1$.  Hence by \ref{1f.23} with
$\lambda,\chi_1,\chi_2,\kappa,\theta_1,\theta_2,J_1,J_2$ here standing for
$\lambda,\mu,\chi,\sigma,\kappa,\theta,J_1,J_2$ there, we have
finished the proof of $(*)_5$ getting even $(\mu^+,J)$-free.]
\mn
\begin{enumerate}
\item[$(*)_6$]  there is $\cF \subseteq {}^\kappa \mu$ of cardinality
$\lambda$ which is $(\Upsilon,5)$-free letting $\Upsilon = \theta^{+
\kappa}_1$.
\end{enumerate}
\mn
Why?  We apply \ref{1f.10}, case (c) with $\chi_1,\theta_1$ here
standing for $\chi,\theta$ there.  Let $\langle \lambda_i:i <
\kappa\rangle$ be an increasing sequence of regulars with limit $\mu$
such that $(\prod\limits_{i} \lambda_i,<_J)$ has true cofinality
$\lambda,\langle \varrho^1_\alpha:\alpha < \lambda\rangle$ witness it.

We choose $\langle \varrho^2_\alpha:\alpha < \lambda\rangle$ listing
$\cF$ as in $(*)_5$, so it is $\theta^{++}_1$-free.  Let
$\varrho_\alpha = \langle
\pr(\varrho^1_\alpha(i),\varrho^2_\alpha(i)):i < \kappa\rangle$.
Clearly $\partial \in \issp(\{\varrho_\alpha:\alpha < \lambda\})
\Rightarrow \partial \ge \Upsilon(\theta^{++}_1)^{+ \kappa}$.

We choose $\langle \eta^1_\alpha:\alpha < \lambda\rangle$ be
$<_J$-increasing cofinal is some $(\prod\limits_{i < \theta_1}
\lambda^2_i,<_{J_2})$ for some regular $\lambda^2_i < \chi_1$, exist
because $(\chi_1,J_1)$ witness $\theta \in \Theta$.  Hence by
\ref{1.3.15} we have
$\partial \in \issp(\{\eta_\alpha:\alpha < \lambda\}) \Rightarrow
\partial \ge \theta^{+\theta_1} \ge \theta^{+ \kappa}_1$.
\end{PROOF}

\begin{claim}
\label{1f.56}
If (A) then (B) where
\mn
\begin{enumerate}
\item[$(A)$]  $(a) \quad J$ is a $\sigma^+$-complete ideal on $\kappa$
\sn
\item[${{}}$]  $(b) \quad \cF_i \subseteq {}^\kappa \mu$ has
  cardinality $\lambda$ for $i < \sigma$
\sn
\item[${{}}$]  $(c) \quad \mu = \mu^\sigma$ \underline{or}
$(\forall i)([\cF_i \subseteq \prod\limits_{\varepsilon < \kappa}
  \lambda_\varepsilon]$ and $\varepsilon < \kappa
  \Rightarrow (\lambda_\varepsilon)^\sigma < \mu$
\sn
\item[$(B)$]  there is $\cF \subseteq {}^\kappa \mu$ of cardinality
  $\lambda$ such that:
\sn
\item[${{}}$]  $(a) \quad \cF$ is $(\theta_2,\theta_1)-J$-free \when \,
  at least one $\cF_i$ is $(\theta_1,\theta_2)$-free
\sn
\item[${{}}$]  $(b) \quad \cF$ is $(\theta_n,\theta_0)-J$-free \when \,
  $\theta_0 < \ldots < \theta_n$ and for each $\ell < n$ for some 

\hskip25pt  $i < \sigma$ the set 
$\cF_i$ is $(\theta_{\ell+1},\theta_\ell)-J$-free.
\end{enumerate}
\end{claim}

\begin{PROOF}{\ref{1f.56}}
Straightforward.
\end{PROOF}

\noindent
We may note that (related to the beginning of \S3)
\begin{observation}
\label{1f.9d}
Claim \ref{1f.8} implies Claim \ref{1f.7}.
\end{observation}

\begin{PROOF}{\ref{1f.9d}}
We assume $\circledast$ from \ref{1f.7} and let $\theta_1 = \theta_2 =
\theta,\sigma = (2^\theta)^+,J_2 = J^{\bd}_\theta$ and prove that
$\circledast'$ of \ref{1f.8} holds, this suffices.
\smallskip

\noindent
Clause $\circledast'(a)$ holds by clause $\circledast(a)$.
\smallskip

\noindent
Clause $\circledast'(b)$ holds by clause $\circledast(b)$.
\smallskip

\noindent
Clause $\circledast'(c)$ holds as we have chosen $J_2$ as
$J^{\bd}_\theta$.
\smallskip

\noindent
Clause $\circledast'(d)$ holds by clause $\circledast(f)$.
\smallskip

\noindent
Clause $\circledast'(e)$ holds by clause $\circledast(e)$.
\smallskip

\noindent
Clause $\circledast'(f)$ holds, moreover $\circledast'(f)(\alpha)$
holds and we have chosen $\theta_1 = \theta$ and $J_2 =
J^{\bd}_\theta$ and by $\circledast( )$ the cardinal $\theta$ is regular.
\smallskip

\noindent
Clause $\circledast'(g)$ holds by clause $\circledast(g)$,
i.e. letting $\cT$ be as there, \wilog \, $\cT$ is a subtree of
${}^{\theta >}\chi$ and we can find pairwise distinct $\eta_\alpha \in
\lim_\theta(T) \subseteq {}^\theta \chi$ so $\eta_\alpha \in {}^\theta
\chi$ and $\alpha \ne \beta \Rightarrow \{i < \theta:\eta_\alpha(i) =
\eta_\beta(i)\} \subseteq \ell g(\eta_\alpha,m_\beta) \in
J^{\bd}_\theta = J_2$ by the choice of $J_2$.
\smallskip

\noindent
Clause $\circledast'(h)$ holds by the proof of $(*)_2$ inside the
proof of Claim \ref{1f.7}.
\smallskip

\noindent
Clause $\circledast'(i)$ holds, moreover clause
$\circledast'(i)(\alpha)$ holds because $\sigma = (2^\theta)^+ >
\cP(\theta) \ge |\cP(\theta)/J_2|$.
\end{PROOF}

\begin{remark}
In the proof of $(*)_1$ inside the proof of \ref{1f.7} we may wonder.
\end{remark}

\begin{question}
What occurs if we just assume
\mn
\begin{enumerate}
\item[$\odot$]  $\pp^+_J(\mu) > \chi$?
\end{enumerate}
\end{question}

\noindent
Answer:
\begin{claim}
\label{1f.70}
Inside the proof of \ref{1f.7} there is $\bar\rho$ as in 
$(*)_1$ provided that we add to the assumption:
\mn
\begin{enumerate}
\item[$(i)$]   at least one of the following holds:
\sn
\item[${{}}$]  $(\alpha) \quad \chi$ is regular
\sn
\item[${{}}$]  $(\beta) \quad 2^\kappa < \cf(\chi)$
\sn
\item[${{}}$]  $(\gamma) \quad \alpha < \mu \Rightarrow |\alpha|^{<
  \kappa} < \mu$.
\end{enumerate}
\end{claim}

\begin{PROOF}{\ref{1f.70}}
By the assumption $\odot$, for every regular $\chi_1 \in [\mu,\chi]$ we get
$\bar\lambda_{\chi_1}$ and subsequence $\bar \mu_{\chi_1}$ of $\bar
\mu$ and $\bar\sigma_{\chi_1}$ as above.

Now we use clause $\circledast(i)$, so one of the three possibilities
there holds.  The first say $\chi$ is regular, and we choose $\chi_1 =
\chi$ so using $\bar\mu_\chi,\bar\rho_\chi$ we are done; and \wilog \,
we assume $\chi$ is singular.  

The second says $\cf(\chi) > 2^\kappa$
hence for some $\bar\mu'$ the set $\Xi = \{\chi_1 < \chi:\chi_1 \ge
\mu$ is regular and $\bar\mu_{\chi_1} = \bar \mu'\}$ is unbounded in
$\chi$ and using the $\langle \bar\rho_{\chi_1}:\chi_2 \in \Xi\rangle$
by the proof of \cite[Ch.II,1.5A,pg.51]{Sh:g}, i.e. using a pairing
function on each $\mu'_i$ there is $a,\bar\rho$ as required in
$(*)_1$, replacing $\bar\mu$ by $\bar\mu'$, of course.

The third says $\alpha < \mu \Rightarrow |\alpha|^{< \kappa} < \mu$,
so \wilog \, $i < \kappa \Rightarrow \mu_i = \mu^{< \kappa}_i$.  Now
for every regular $\chi_1 \in (\mu,\chi)$ we define
$\bar\rho'_{\chi_1} = \langle \rho'_{\chi_1,\gamma}:\gamma <
\chi_1\rangle$ where $\rho'_{\chi_1,\gamma} \in \prod\limits_{i <
  \kappa} \mu_i$, yes using the original $\bar\mu$, is defined by
$\rho'_{\chi_1,\gamma}(i) = \bold h_i(\langle \rho_{\chi_1,\gamma}(j):j
< h_{\chi_1}(i)\rangle$ where $h_{\chi_1}(i) = \min\{\varepsilon <
\kappa:\lambda_{\chi_1,\varepsilon} > \mu_i\}$ and $\bold h_i$ is a
one-to-one function from $\Pi\{\lambda_{\chi_1,j}:j < h_{\chi_1}(i)\}$
into $\mu_i$.

Recalling $J_1 \supseteq J^{\bd}_\kappa$ clearly $\langle
\rho'_{\chi_1,\gamma}:\gamma < \chi_1\rangle$ is $(\mu^+,J_1)$-free as
a set, and we finish as in ``the second".  So we are done.
\end{PROOF}
\newpage

\section {On the $\mu$-free trivial dual conjecture} \label{On} 

We shall look at the following definition.
\begin{definition}
\label{0f.07}
1) For a ring $R$ and a cardinal
$\mu$, let sp$_\mu(R)$ be the class of
regular cardinals $\kappa$ such that there is a witness $(\bar G,h)$
where ``$(\bar G,h)$ is a witness for $\spp_\mu(R)$" means:
\mn
\begin{enumerate}
\item[$\circledast$]  $(a) \quad \bar G = \langle G_i:i \le \kappa
+1\rangle$
\sn
\item[${{}}$]  $(b)(\alpha) \quad \bar G$ is an increasing continuous
sequence 
\sn
\item[${{}}$]  $\qquad (\beta) \quad G_i$ is a left $R$-module,
free for $i \ne \kappa$
\sn
\item[${{}}$]  $(c) \quad$ if $i < j \le \kappa +1$ and $(i,j) \ne
(\kappa,\kappa +1)$, then $G_j/G_i$ is free,
\sn
\item[${{}}$]  $(d) \quad h$ is a homomorphism from $G_\kappa$ to $R$
as left $R$-modules,
\sn
\item[${{}}$]  $(e) \quad h$ cannot be extended to a homomorphism
from $G_{\kappa +1}$ to $R$,
\sn
\item[${{}}$]  $(f) \quad |G_{\kappa +1}| \le \mu$.
\end{enumerate}
\mn
2) For a ring $R$ and cardinals $\mu \ge \theta$, we define
sp$_{\mu,\theta}(R) = \text{ sp}^1_{\mu,\theta}(R)$ similarly,
replacing ``free" by ``$\theta$-free" in clauses (b) and (c).  Writing sp$_{<
\mu}(R)$ or sp$_{< \mu,\theta}(R)$ means that ``$|G_{\kappa +1}| <
\mu$" in clause (f).
\end{definition}

\begin{definition}
\label{0f.13}
1) Let sp$(R) = \cup\{\text{sp}_\mu(R):\mu$ a cardinal$\} 
= \{\kappa:\kappa$ is a regular
cardinal such that for some $\bar G$ the conditions $\circledast (a)-(e)$ from
\ref{0f.07}(1) hold$\}$.
\newline
2)  Let sp$_1(R) = \cap\{\text{sp}^1_\theta(R):\theta$ a cardinal$\}$
where sp$^1_\theta(R) = \{\kappa:\kappa$ is regular such that 
for some $\mu$, we have $\kappa \in \text{ sp}_{\mu,\theta}(R)\}$.
\end{definition}

\noindent
The next definition is similar to \ref{0f.07} (adding the parameter 
``$r \in R$"), but replacing the cardinal $\kappa$
by a set of ideals on $\kappa$, that is:
\begin{definition}
\label{5e.7}
1) Let sp$^2_{\lambda,\theta}(R)$ be 
the set of cardinals $\kappa$ such that $J^{\text{bd}}_\kappa \in \text{
SP}_{\lambda,\theta}(R)$, see below.

\noindent
2)  SP$_{\lambda,\theta}(R)$ is the set of
ideals $J$ on some $\kappa$ such that for every $r \in R
\backslash \{0\}$, there exists a witness $(\bar G,h)$ (for $r$),
where ``$(r,\bar G,h)$ is a witness for $\SP_{\lambda,\theta}(R)$" and
$(\bar G,h)$ witness $\SP_{\lambda,\theta}(R)$ (for $r$)"
means that $(r,\bar G,h)$ possesses the following properties:
\mn
\begin{enumerate}
\item[$\circledast$]  $(a) \quad \bar G = \langle G_i:i \le
\kappa +1\rangle$ is a sequence of (left) $R$-modules,
\sn
\item[${{}}$]  $(b) \quad G_\kappa = 
\oplus\{G_i : i < \kappa\} \subseteq G_{\kappa +1}$,
\sn
\item[${{}}$]  $(c) \quad$ if $u \in J$, then 
$G_{\kappa +1}/\oplus\{G_i : i \in u\}$ is a $\theta$-free (left) $R$-module,
\sn
\item[${{}}$]  $(d) \quad G_i$ is a $\theta$-free $R$-module and $G_i
  \ne 0$ for simplicity,
\sn
\item[${{}}$]  $(e) \quad |G_{\kappa +1}| \le \lambda$ and $\kappa \le
\lambda$ (follows in non-trivial cases)
\sn
\item[${{}}$]  $(f) \quad h$ is a non-zero homomorphism from $G_\kappa$ to
${}_R R$, i.e. $R$ 

\hskip25pt as a left module,
\sn
 \item[${{}}$]  $(g) \quad$ there is no homomorphism $h^+$ from
$G_{\kappa +1}$ to ${}_R R$ such that

\hskip25pt  $x \in G_\kappa \Rightarrow h^+(x) = h(x)r$.
\end{enumerate}
\mn
3) Omitting $\theta$ means replacing ``$\theta$-free" by ``free"; 
omitting $\theta$ and $\lambda$ means for some $\lambda$; 
writing ``$< \lambda$" has the obvious meaning.
\end{definition}

\begin{observation}
\label{5e.8}
1) If $J_1 \subseteq J_2$ are ideals on $\kappa$ \then \, 
$J_1 \in \SP_{\lambda,\theta}(R)$ implies $J_2 \in \SP_{\lambda,\theta}(R)$.

\noindent
2) If $J_\ell$ is an ideal on $\kappa_\ell$ for $\ell=1,2$ and $J_1
   \le_{\RK} J_2$ \then \, the above holds.

\noindent
3) If $G$ is a left $R$-module, $h$ a homomorphism from $G$ to $R$ (as
   a left $R$-module) and $r \in R$ \then \, the mapping $x \mapsto
   h(x) \cdot r$ is a homomorphism from $G$ to $R$.
\end{observation}

\begin{proof}
Straightforward.
\end{proof}

\begin{remark}
\label{5e.9}
1) Note that if $R$ is a torsion free ring (i.e. $a b=0_R \Rightarrow
a=0_R \vee b=0_k$) \then \, clause (g) of Definition \ref{5e.7} holds
also for $r=1$.  If in addition every left ideal of $R$ is principal then
\wilog \, $r=1$.

\noindent
2) In \ref{5e.7}(2), if $\kappa$ is regular
for $J = J^{\text{bd}}_\kappa$, we may
replace clause (c) by ``$i <\kappa \Rightarrow G_{\kappa
  +1}/(\oplus\{G_j:j < i\})$ is a $\theta$-free $R$-module"; in 
general, we may replace $J$ by a directed subset
of ${\mat P}(\kappa)$ generating it.

\noindent
3) Note that if $J \in \SP_{\lambda,\theta}(R)$ then $\lambda \ge |R|$
   because by clause (c) of \ref{5e.7}(2) we know that $G_{\kappa +1}$
   is $\theta$-free hence is of cardinality $\ge |R|$, (except when
   $G_{\kappa +1}$ is zero contradicting clause (g) there) and
   $\lambda \ge |G_{\kappa +1}|$ by clause (e) there.
\end{remark}

As in \ref{0p.7}
\begin{definition}
\label{e5.12}
Let TDU$_{\lambda,\mu}(R)$ mean that $R$ is a ring and
there is a $\mu$-free left $R$-module $G$ of cardinality 
$\lambda$ with {\rm Hom}$_R(G,R) = \{0\}$, that is, with 
no non-zero homomorphism from $G$ to $R$ as left $R$-modules.
\end{definition}

\begin{claim}
\label{5e.14}
A sufficient condition for {\rm TDU}$_{\lambda,\mu}(R)$ is:
\mn
\begin{enumerate}
\item[$\circledast$]  $(a) \quad R$ is a ring with unit ($1 = 1_R$)
\sn
\item[${{}}$]  $(b) \quad J \in \SP_{\chi,\mu}(R)$ so is an ideal on $\kappa$ 
\sn
\item[${{}}$]  $(c) \quad \bar C = \langle C_\delta:\delta \in
S\rangle$ is such that {\rm otp}$(C_\delta) = \kappa$ and $C_\delta
\subseteq \delta$ where $S$

\hskip25pt  is an unbounded subset of $\lambda$
\sn
\item[${{}}$]  $(d) \quad \lambda > |R| + \chi$ is regular, or at least
{\rm cf}$(\lambda) > |R| + \chi + \kappa$ and $\mu > \kappa$
\sn
\item[${{}}$]  $(e) \quad$ {\rm BB}$(\lambda,\bar C,\Upsilon,J)$ where
  $\Upsilon = 2^{(2^{|R|+\chi})^+},\kappa \le (2^\chi)^+$ 
and $\chi < \lambda$, so 

\hskip25pt $I_* = J^{\text{\rm bd}}_S$ recalling 
$J^{\text{\rm bd}}_S = \{\mat
U:\mat U \subseteq S$ and $\sup(\mat U) < \sup(S))\}$
\sn
\item[${{}}$]  $(f) \quad \bar C$ is $(\mu,J)$-free; recalling
\ref{1.3.14}(1A).
\end{enumerate}
\end{claim}

\begin{remark}
0) See more in Definition \ref{5e.38} on.

\noindent
1) In the present definition of $\SP_{\lambda,\theta}(R)$, 
we need to use $\BB(\lambda,\bar C,\Upsilon,J)$ before
applying SP in \ref{5e.14}.  
But normally it suffices to have a version of BB 
with fewer colours and weaker demands on $|G_i|$,
for example:
\mn
\begin{enumerate}
\item[$(A)$]  Use BB$(\lambda,\bar C,(\chi_*,\theta),J)$ and $\chi_* =
\Pi\{|R|^{\chi_i}:i < \kappa\}$, where $\chi_i = |G_i|
+$ sup$\{|\text{Hom}(G_j,{}_R R)|:j < \kappa\}$
\sn
\item[$(B)$]   We define $\SP_{\lambda,\bar \chi,\sigma,\theta}(R)$ as
in Definition \ref{5e.7}(2) where 
$\bar \chi = \langle \chi_i:i < \kappa\rangle$ and
write $\chi$ if $(\forall i)(\chi_i = \chi$) but instead of clauses
(e) and (f) + (g)
\sn
\begin{enumerate}
\item[$(e)'$]  $|G_{\kappa +1}| \le \lambda$ and $|\text{Hom}(G_i,R)| \le
\chi_i$,
\sn
\item[$(f)'$]  $\bar h = \langle h_i:i < \sigma\rangle,h_i \in
\text{ Hom}(G_\kappa,{}_R R)$ and if $i <j < \sigma$, then $h_i - h_j$
cannot be extended to any $h' \in \text{ Hom}(G_{\kappa +1},{}_R R)$,
\end{enumerate}
\sn
\item[$(C)$]   In Definition \ref{5e.14}, we change
\sn
\begin{enumerate}
\item[$(b)'$]  $\kappa \in \SP_{\lambda,\chi,\sigma,\theta}$
or ($\bar C$ is tree-like, $\kappa \in \SP_{\lambda,\bar
\chi,\sigma,\theta}$ and $J \in \SP_{\lambda,\bar\chi,\sigma,
\theta}$ is an ideal on $\kappa$)
\sn
\item[$(e)'$]  BB$(\lambda,\bar C,(\chi,\sigma),J)$. 
\end{enumerate}
\end{enumerate}
\mn
2) BB$(\lambda,\bar C,(\chi,1/\sigma),J)$ is 
sufficient for the correct version of \ref{5e.7}, see Definition
\ref{0p.15}(2); really we need there to use $\theta = 2^\kappa$ and the
guessing is of an initial segment of the possibilities, i.e. in
\ref{5e.7} we need: without loss of generality 
$|G_i| \le \kappa$ for every $i$,
given $f_\varepsilon \in \Hom(G_\kappa,{}_R R)$
for $\varepsilon < \varepsilon(*) < 2^\kappa$ we can find, e.g. a
permutation $\pi$ of $\kappa$, inducing $G^\pi_\kappa \supseteq
\oplus\{G_i:i < \kappa\}$ such that none of them can be extended to
$f \in \Hom(G^\pi_\kappa,{}_R R)$.
This means we can use ``very few colours" as in \cite[AP,\S1]{Sh:f},
i.e., \ref{0p.15}(2A).

\noindent
3) See $\odot_0$ in \S0.

\noindent
4) We may use only tree-like $\bar C$'s (in \ref{5e.14}(c)) and in
BB$(\lambda,\bar C,(\bar\chi,\sigma),J)$ (in $(C)(e)'$ above).

\noindent
5) In the proof of \ref{5e.14}, if we demand that $G_i (i < \kappa)$
is a free $R$-module, 
then we can save on $\chi$, using free $R$-moduels $G^*_\alpha$'s.

\noindent
6) The beginning of the proof can be stated separately.
\end{remark}

\begin{PROOF}{\ref{5e.14}}  
\Wilog \, $\bar C$ is normal, see \ref{1.3.14}(5).  
By the definitions \ref{0p.14}, \ref{0p.15}
 of BB$(\lambda,\bar C,\Upsilon,J)$, there is a sequence 
$\langle S_\varepsilon:\varepsilon < \lambda \rangle$ 
of $\lambda$ pairwise disjoint subsets of
$S=S(\bar C)$ such that BB$^-(\lambda,\bar C \restriction
S_\varepsilon,\Upsilon,J)$ holds for each $\varepsilon < \lambda$.

Without loss of generality $\delta \in S \Rightarrow C_\delta \cap S =
\emptyset$, moreover $S$ is a set of limit ordinals and each
$C_\delta$ is a set of successor ordinals and we
let $C_* = \cup\{C_\delta:\delta \in S\}$\footnote{Why?  Replace $S$
by $S' = \{\delta \in S:\delta$ a limit ordinal$\}$ and replace
$C_\delta$ by $C'_\delta := \{\alpha +1:\alpha \in S\}$.}.  
We say that $D$ is $\bar C$-closed when 
$D \subseteq C_* \cup S$ and $\delta \in
D \cap S \Rightarrow C_\delta \subseteq D$.  So for every $B'
\subseteq C_* \cup S$ there is a $\bar C$-closed $B'' \subseteq C_*
\cup S$ such that $B' \subseteq B'' \wedge |B''| \le |B'| + \kappa$. 
We can put $\lambda$ of the $S_i$'s together, i.e.
\mn
\begin{enumerate}
\item[$\boxplus_1$]  we can replace $\langle S_i:i < \lambda\rangle$ by 
$\langle \cup\{\{S_i:i \in {\mat U}_\zeta\}:\zeta < \lambda\}\rangle$ 
provided that $\langle {\mat U}_\zeta:\zeta < \lambda\rangle$ is a
partition of $\lambda$ with each $\cU_\zeta$ non-empty).
\end{enumerate}
\mn
Also
\mn
\begin{enumerate}
\item[$\boxplus_2$]  we can replace $\langle C_\delta:\delta \in
S\rangle$ by $\langle C_\delta \backslash h(\delta):\delta \in
S\rangle$ when $h$ is a function satisfying
$\delta \in S \Rightarrow h(\delta) \in C_\delta$,
\end{enumerate}
\mn
hence \wolog \, 
\mn
\begin{enumerate}
\item[$\boxplus_3$]  $(a) \quad \varepsilon < \lambda \wedge
S' \subseteq S_\varepsilon \wedge |S'| < \lambda \Rightarrow 
\text{ BB}^-(\lambda,\bar C \restriction
(S_\varepsilon \backslash S'),\Upsilon,J)$
\sn
\item[${{}}$]  $(b) \quad$ if $\alpha < \lambda$ \then \, 
for $\lambda$ ordinals $\varepsilon < \lambda$ we have

\hskip25pt  $\delta \in S_\varepsilon \Rightarrow \alpha < \min(C_\delta)$.
\end{enumerate}
\mn
Note that we have
\mn
\begin{enumerate}
\item[$\circledast_0$]   $\chi \ge |R| + \kappa$ and $\lambda > 2^\chi$.
\end{enumerate}
\mn
[Why?  We have $\chi \ge |R|$ because $\SP_{\chi,\mu}(R) \ne
  \emptyset$ by clause (b) of the assumption, using \ref{5e.9}(3).
The ``and" holds as $\lambda \ge \Upsilon$ by the first phrase of 
clause (e) of the assumption and $\Upsilon > 2^\chi$ by the second
phrase of clause (e) of the assumption.]
\mn
\begin{enumerate}
\item[$\circledast_1$]  There is a $\mu$-free $R$-module $G_*$ of
cardinality $\chi_* := (2^\chi)^+$ such that
\sn
\begin{enumerate}
\item[$(a)$]  $G_* = \oplus\{G_{*,\varepsilon}:\varepsilon < \chi_*\}$,
\sn
\item[$(b)$]  if $G$ is a $\mu$-free $R$-module of cardinality
$\le \chi$, \underline{then} $G$ is isomorphic to $G_{*,\varepsilon}$ 
for $\chi_*$ ordinals $\varepsilon < \chi_*$, (actually we need just
that for any $r \in R \backslash \{0_R\}$ there is a sequence $\langle
G_i:i \le \kappa +1 \rangle$ satisfying $\circledast$ of
Definition \ref{5e.7}(2) with $\chi,\mu$ here standing for
$\lambda,\theta$ there),
\sn
\item[$(c)$]   $G_{*,\varepsilon}$ is a $\mu$-free $R$-module of 
cardinality $\le \chi$ for each $\varepsilon < \chi_*$.
\end{enumerate}
\end{enumerate}
\mn
[Why?  Because the number of such $G$'s up to isomorphism is $\le
2^{|R|+\chi} = 2^\chi$ and $\kappa \le (2^\chi)^+ = \chi_*$.]

Let $E = \{(\varepsilon,\zeta):\varepsilon,\zeta < \chi_*$ and
$G_{*,\varepsilon} \cong G_{*,\zeta}\}$, so $E$ is an equivalence
relation on $\chi_*$ and $\varepsilon/E := \{\zeta <
\chi_*:\varepsilon E \zeta\}$ is the equivalence class of $\varepsilon
< \chi_*$ under $E$.  
For $\varepsilon < \chi_*$, let $f^1_\varepsilon$ be an isomorphism from
$G_{*,\text{min}(\varepsilon/E)}$ onto $G_{*,\varepsilon}$.
\mn
\begin{enumerate}
\item[$\circledast_2$]  For any $r \in R \backslash \{0\}$ 
let ${\bold x}_r = \{(\bar G,h):(\bar G,h)$ witness $J \in \text{
SP}_{\chi,\theta}(R)$ for $r$, see Definition \ref{5e.7}(2)$\}$,
\sn
\item[$\circledast_3$]  $H_* := \bigoplus\{G^*_\alpha:\alpha \in C_*\}
\oplus \bigoplus\{K^*_\delta:\delta \in S\}$, where 
\sn
\begin{enumerate}
\item[$\bullet_1$]   each $G^*_\alpha$ is isomorphic to $G_*$ under 
$g^1_\alpha$,
\sn
\item[$\bullet_2$]  $K^*_\delta$ isomorphic to $G_*$
for $\delta \in S$ under $g^2_\delta$ and 
\sn
\item[$\bullet_3$]  for $\varepsilon < \chi_*$
let $G_{\alpha,\varepsilon} = g^1_\alpha(G_{*,\varepsilon}),
K_{\delta,\varepsilon} = g^2_\alpha(G_{*,\varepsilon})$ 
\end{enumerate}
\sn
\item[$\circledast_4$]  let $K_{< \delta} = 
\oplus\{G^*_\alpha:\alpha \in C_\delta\}$ for $\delta \in S$, which has
cardinality $\chi_*$ as $\kappa \le \chi_*$ by clause (e) of the assumption
\sn
\item[$\circledast_5$]  for every $B \subseteq C_* \cup S$ 
let $H_B := \bigoplus\{G^*_\alpha:\alpha \in B \cap C_*\}
\oplus \bigoplus\{K^*_\delta:\delta \in S \cap B\}$.
\end{enumerate}
\mn
We easily see that
\mn
\begin{enumerate}
\item[$\circledast_6$]  for every $x \in H_*$ there is a
$\bar C$-closed set $D^*_x \subseteq C_* \cup S$ of cardinality $\le \kappa$
such that $x \in H_{D^*_x}$, in fact there is a minimal one.
\end{enumerate}
\mn
Let
\mn
\begin{enumerate}
\item[$\circledast_7$]  $(a) \quad \langle(x_i,r_i):i < \lambda\rangle$ 
list the pairs $(x,r)$ such that $x \in H_*,r \in R \backslash \{0_R\}$
\sn
\item[${{}}$]  $(b) \quad$ by $\circledast_6 + \boxplus_3$ \wilog \,
$\delta \in S_i \Rightarrow \sup(D^*_{x_i}) < \text{ min}(C_\delta)$.
\mn
\end{enumerate}

Let
\mn
\begin{enumerate}
\item[$\circledast_8$]  $H_{<\alpha} :=
   \oplus\{G^*_\beta,K^*_\delta:\beta \in C_* \cap \alpha$ 
and $\delta \in S \cap \alpha\}$.
\end{enumerate}
\mn
For $\delta \in S$ let
$\beta(\delta,\iota)$ be the $\iota$-th member of $C_\delta$.

As $\delta \in S$, clearly Hom$(K_{<\delta},{}_R R)$ is a set of
cardinality $\le 2^{\chi_*} = \Upsilon$.  Also any $f \in \text{ Hom}(K_{<
\delta},{}_R R)$ is determined by $\langle f \restriction
G^*_\alpha:\alpha \in C_\delta\rangle$.
Hence by clause (e) of the assumption,
for each $i < \lambda$, we can find $\langle
h^1_\delta:\delta \in S_i\rangle$ such that
\mn
\begin{enumerate}
\item[$\circledast_9$]  $(a) \quad$ if $\delta \in S_i$, then
$h^1_\delta \in \text{ Hom}(K_{<\delta},{}_R R)$
\sn
\item[${{}}$]  $(b) \quad$ if $i < \lambda$ and 
$h \in \Hom(H_*,{}_R R)$, then for some (even stationarily 

\hskip25pt  many) $\delta \in S_i$, we have $h^1_\delta \subseteq h$
\sn
\item[$\circledast_{10}$]  for $\delta \in S_i$, let
\sn
\begin{enumerate}
\item[$(a)$]   $x^*_\delta = x_i,r^*_\delta = r_i$
\sn
\item[$(b)$]   let $\bar N^\delta = \langle N^\delta_\iota:
\iota \le \kappa +1\rangle$ and $h^*_\delta$ be, for $r^*_\delta$, 
as guaranteed in Definition \ref{5e.7}(2), with $N^\delta_i$ here
standing for $G_i$ there, so $h^*_\delta \in \Hom(N^\delta_\kappa,{}_R R)$
\sn
\item[$(c)$]  for $\iota < \kappa$, let $\varepsilon(\delta,\iota)
= \text{ Min}\{\varepsilon < \chi_*:G_{*,\varepsilon} \cong
N^\delta_\iota\}$ and let $f^0_{\delta,\iota}$ be an isomorphism 
from $N^\delta_\iota$ onto $G_{*,\varepsilon(\delta,\iota)}$.
\end{enumerate}
\end{enumerate}
\mn
[Why is this possible?  By clause (b) of the assumption.]

Now
\begin{enumerate}
\item[$\circledast_{11}$]   for $\delta \in S_i$ and $\iota < \kappa$ 
we can choose $\varepsilon_{\delta,\iota,1} 
< \varepsilon_{\delta,\iota,2} < \chi_*$
from $Y = Y_{\delta,\iota} = \{\zeta < \chi_*:
G_{*,\varepsilon(\delta,\iota)} \cong G^\delta_{*,\zeta}\}$ 
such that $h^1_\delta \circ g^1_{\beta(\delta,\iota)} 
\circ f^1_{\varepsilon_{\delta,\iota,1}} \circ f^0_{\delta,\iota} =  
h^1_\delta \circ g^1_{\beta(\delta,\iota)}  \circ
f^1_{\varepsilon_{\delta,\iota,2}} \circ f^0_{\delta,i}$.
\end{enumerate}
\mn
[Why?  Note that $\min(Y) = \varepsilon(\delta,\iota)$ and
\mn
\begin{enumerate}
\item[$\bullet$]  $h^1_\delta \in \text{ Hom}(K_{< \delta},{}_R R)$
hence $h^1_\delta \rest G^*_{\beta(\delta,\iota)} \in 
\text{ Hom}(G^*_{\beta(\delta,\iota)},{}_R R)$
\sn
\item[$\bullet$]  $g^1_{\beta(\delta,\iota)}$ is an isomorphism from
$G_*$ onto $G^*_{\beta(\delta,\iota)}$ hence $h^1_\delta \circ
g^1_{\beta(\delta,\iota)} \in \text{ Hom}(G_*,{}_R R)$
\sn
\item[$\bullet$]  $f^1_\varepsilon$, see before $\circledast_2$, 
is an isomorphism from $G_{*,\min(Y)}$ onto 
$G_{*,\varepsilon} \subseteq G_*$ for $\varepsilon \in Y$
\sn
\item[$\bullet$]  $\langle h^1_\delta \circ g^1_{\beta(\delta,\iota)}
\circ f^1_\varepsilon:\varepsilon \in Y \rangle$ is a sequence of
members of $\Hom(G_{*,\min(Y)},{}_R R)$
\sn
\item[$\bullet$]  Hom$(G_{*,\min(Y)},{}_R R)$ has cardinality
$\le |R|^{|G_{*,\min(Y)}|} \le |G_*| \le 2^{\chi +|R|}$, whereas
$|Y| = \chi_* = (2^\chi)^+$.
\end{enumerate}
\mn
Hence we can chose
$\varepsilon_{\delta,\iota,1},\varepsilon_{\delta,\iota,2}$ such that
\mn
\begin{enumerate}
\item[$\bullet$]  $\varepsilon_{\delta,\iota,1} <
  \varepsilon_{\delta,\iota,2}$ are members of $Y$ satisfying
  $h^1_\delta \circ g^1_{\beta(\delta,\iota)} \circ
  f^1_{\varepsilon_{\delta,\iota,2}} = h^1_\delta \circ
g^1_{\beta(\delta,\iota)} \circ f^1_{\varepsilon_{\delta,\iota,1}}$.
\end{enumerate}
\mn
So recalling $\circledast_{10}(c)$ 
the desired conclusion of $\circledast_{11}$ holds.]

Let $g^2_{\delta,\iota}$ be the following
embedding of $N^\delta_\iota$ into $H_*$, in fact, into
$G^*_{\beta(\delta,\iota)}$ (recalling $f^0_{\delta,\iota}$ is an
isomorphism from $N^\delta_\iota$ onto $G_{*,\min(Y)}$):
\mn
\begin{enumerate}
\item[$(*)_0$]  $g^2_{\delta,\iota}(x) = g^1_{\beta(\delta,\iota)} \circ
f^1_{\varepsilon_{\delta,\iota,2}} \circ f^0_{\delta,\iota}(x) 
- g^1_{\beta(\delta,\iota)} \circ f^1_{\varepsilon_{\delta,\iota,1}}
\circ f^0_{\delta,\iota}(x)$ for $x \in G^\delta_\iota$.
\end{enumerate}
\mn
Let $g^3_\delta$ be the embedding of $N^\delta_\kappa$ into $H_*$
extending $g^2_{\delta,\iota}$ for each $\iota < \kappa$, so
\mn
\begin{enumerate}
\item[$(*)_1$]  $(a) \quad g^3_\delta$ is an embedding of
  $N^\delta_\kappa$ into $K_{< \delta} \subseteq H_*$
\sn
\item[${{}}$]  $(b) \quad h^1_\delta \restriction
\Rang(g^3_\delta)$ is zero.
\end{enumerate}
\mn
Let $g^4_\delta$ be the following homomorphism from $N^\delta_\kappa$ 
into $H_*$
\mn
\begin{enumerate}
\item[$(*)_2$]  $g^4_\delta(x) = g^3_\delta(x) + 
h^*_\delta(x) x^*_\delta$ for $x \in N^\delta_\kappa$.
\end{enumerate}
\mn
[Why?  Recalling $x^*_\delta \in H_{< \delta}$ is from
$\circledast_{10}(a), h^*_\delta \in \Hom(N^\delta_\kappa,{}_R R)$
  is from $\circledast_{10}(b)$ so $h^*_\delta(x) \in R$ 
hence $h^*_\delta(x) x^*_\delta \in H_*$ indeed.]

By the choice of $H_{< \delta}$ as $\delta \in S_i \Rightarrow
x^*_\delta = x_i \in H_{D^*_{x_i}} \subseteq H_{< \min(C_\delta)}
\subseteq H_{< \delta}$ using $\circledast_7(b)$ clearly
\mn
\begin{enumerate}
\item[$(*)_3$]  $g^4_\delta$ is an embedding of
$N^\delta_\kappa$ into $H_{< \delta}$.
\end{enumerate}
\mn
So by $(*)_1 + (*)_2$ we have
\mn
\begin{enumerate}
\item[$(*)_4$]  if $h$ is a homomorphism from $H$ into ${}_R R$
 where $K_{< \delta} \subseteq H \subseteq H_*$ such
that $h^1_\delta \subseteq h  \wedge h(x^*_\delta) = r^*_\delta$,
\then \,: $x \in N^\delta_\kappa \Rightarrow h(g^4_\delta(x)) 
= h^*_\delta(x) r^*_\delta$.
\end{enumerate}
\mn
Let $\alpha_{\delta,\kappa} < \chi_*$ be such that
$G_{*,\alpha_{\delta,\kappa}} \cong N^\delta_{\kappa +1}$, and let
$f^0_{\delta,\kappa}$ be an isomorphism from $N^\delta_{\kappa +1}$
onto $G_{*,\alpha_{\delta,\kappa}}$, and recalling
$\circledast_3,\bullet_2$ it follows that $g^2_\delta \circ
f^0_{\delta,\kappa}$ embeds $N^\delta_{\kappa +1}$ into $K^*_\delta
\subseteq H_*$ hence
letting $f^4_{\delta,\kappa} = f^0_{\delta,\kappa} \rest
N^\delta_\kappa$ we have $g^2_\delta \circ f^4_{\delta,\kappa} - 
g^4_\delta$ is a homomorphism from $N^\delta_\kappa$ into $H_*$ (actually an
embedding).   

Let
\mn
\begin{enumerate}
\item[$(*)_5$]  $L_\delta = \{g^2_\delta \circ f^4_{\delta,\kappa}(x) 
- g^4_\delta(x):x \in N^\delta_\kappa\}$.
\end{enumerate}
\mn
Clearly $L_\delta$ is an $R$-submodule of $H_*$.
Now by the choice of $(\bar N^\delta,r^*_\delta,h^*_\delta)$ we shall show:
\mn
\begin{enumerate}
\item[$(*)_6$]  there is no homomorphism $h$ from $H_*$ into ${}_R R$
such that $h^1_\delta \subseteq h$ and $h(x^*_\delta) = r^*_\delta$ and $h
\restriction L_\delta = 0_{L_\delta}$ that is constantly zero.
\end{enumerate}
\mn
[Why?  Toward a contradiction assume $h$ is a counterexample
\mn
\begin{enumerate}
\item[$\oplus_{6.1}$]  if $x \in \Rang(g^3_\delta)$ then $x \in K_{<
  \delta}$ and $h(x) = h^1_\delta(x) = 0$.
\end{enumerate}
\mn
[Why?  Note $\Rang(g^3_\delta) \subseteq K_{< \delta}$ hence $x \in
  K_{< \delta}$ by $(*)_1(a),h \supseteq h^1_\delta$ by the choice
  of $h$ and $\Dom(h^1_\delta) = K_{< \delta}$ by $\circledast_9(a)$
  hence $h \rest \Rang(g^3_\delta) = h^1_\delta \rest
  \Rang(g^3_\delta)$.  So as $x \in \Rang(g^3_\delta)$ by the
  assumption of $\oplus_{6.1}$, clearly we have
$h(x) = h^1_\delta(x)$.  But $h^1_\delta \rest \Rang(g^3_\delta)$
  is constantly zero by $(*)_1(b)$ and $x \in \Rang(g^3_\delta)$ so
  $h^1_\delta(x) = 0$, so we are done.]
\mn
\begin{enumerate}
\item[$\oplus_{6.2}$]   $x \in N^\delta_\kappa \Rightarrow
h(g^4_\delta(x)) = h^*_\delta(x) \cdot r^*_\delta$.
\end{enumerate}
\mn
[Why?  The assumptions of $(*)_4$ say that $h^1_\delta \subseteq h^+ \wedge
  h(x^*_\delta) = r^*_\delta$ which hold by the assumption of
  $(*)_6$, but the conclusion of $(*)_4$ is what we claim in $\oplus_{6.2}$.]
\mn
\begin{enumerate}
\item[$\oplus_{6.3}$]  if $x \in N^\delta_\kappa$ then $h((g^2_\delta
 \circ f^4_{\delta,\kappa})(x)) = h(g^4_\delta(x))$.
\end{enumerate}
\mn
[Why?  As (in $(*)_6$) we 
are assuming $h \rest L_\delta$ is constantly zero and by
  the choice of $L_\delta$ in $(*)_5$.]
\mn
\begin{enumerate}
\item[$\oplus_{6.4}$]  if $x \in N^\delta_\kappa$ then $h((g^2_\delta
\circ f^0_{\delta,\kappa})(x)) = h(g^4_\delta(x))$.
\end{enumerate}
\mn
[Why?  As $f^4_{\delta,\kappa} \subseteq f^0_{\delta,\kappa}$, see
after $(*)_4$, and $\oplus_{6.3}$.] 
\mn
\begin{enumerate}
\item[$\oplus_{6.5}$]  if $x \in N^\delta_\kappa$ then $h((g^2_\delta
\circ f^0_{\delta,\kappa})(x)) = h^*_\delta(x)r_\delta$.
\end{enumerate}
\mn
[Why?  By $\oplus_{6.2} + \oplus_{6.4}$.]

Recalling $g^2_\delta$ is from $\circledast_3$ and
$f^0_{\delta,\kappa}$ is from after $(*)_4$
\mn
\begin{enumerate}
\item[$\oplus_{6.6}$]  define $h':N^\delta_{\kappa +1} \rightarrow {}_R R$
  by $h'(x) = h((g^2_\delta \circ f^0_{\delta,\kappa})(x))$.
\sn
\item[$\oplus_{6.7}$]   $(a) \quad h'$ is indeed a function from 
$N^\delta_{\kappa +1}$ to ${}_R R$
\sn
\item[${{}}$]  $(b) \quad$ moreover it is an $R$-module homomorphism.
\end{enumerate}
\mn
[Why?  As $f^0_{\delta,\kappa}$ is a homomorphism from
  $N^\delta_{\kappa +1}$ into $G_{*,\alpha_{\delta,\kappa}}$ and
  $g^2_\delta$ is a homomorphism from $G_* \supseteq
  G_{*,\alpha_{\delta,\kappa}}$ into $H_*$ and $h$ is a homomorphism
  from $H_*$ to ${}_R R$.]
\mn
\begin{enumerate}
\item[$\oplus_{6.8}$]   $h'$ extends the mapping $x \mapsto
 h^*_\delta(x) \cdot r_\delta$ for $x \in N^\delta_\kappa$.
\end{enumerate}
\mn
[Why?  By $\oplus_{6.5}$.]

Now $\oplus_{6.7} + \oplus_{6.8}$ contradicts the choice of
$h^*_\delta,r^*_\delta$ in $\circledast_{10}$.  So $(*)_6$ indeed holds.]

Lastly, let
\mn
\begin{enumerate}
\item[$(*)_7$]  $(a) \quad L := \Sigma\{L_\delta:\delta \in S\}$, a
sub-module of $H_*$
\sn
\item[${{}}$]  $(b) \quad H := H_*/L$, a module of cardinality $\lambda$.
\end{enumerate}
\mn
So
\mn
\begin{enumerate}
\item[$(*)_8$]  Hom$(H,{}_R R) = 0$.
\end{enumerate}
\mn
[Why?  Assume $h \in \text{ Hom}(H,{}_R R)$ is not constantly zero, 
so we can define
$h^+ \in \text{ Hom}(H_*,{}_R R)$ by $h^+(x) = h(x+L)$ hence also
  $h^+$ is not constantly zero.  Let $x \in H_*$
be such that $h^+(x) \ne 0$, so for some $i < \lambda$ we have
$(x_i,r_i) = (x,h^+(x))$.  By the choice of $\langle
h^1_\delta:\delta \in S_i\rangle$ the set $\{\delta \in S_i: 
h \restriction K_{<\delta} = h^1_\delta\}$ is unbounded in $\lambda$, 
so for some $\delta \in S_i$ we have:
\mn
\begin{enumerate}
\item[$\oplus_{8.1}$]  $h \rest K_{<\delta} = h^1_\delta$, 
\end{enumerate}
\mn
and by $(*)_6$ we are done as $h^+ \restriction L_\delta$ is zero.]
\mn
\begin{enumerate}
\item[$(*)_9$]  $H$ is a $\mu$-free $R$-module.
\end{enumerate}
\mn
[Why?  Let $H^1 \subseteq H$ be of cardinality $<\mu$. So for some
$H^2 \subseteq H_*$ of cardinality $< \mu$, we have 
$H^1 = \{x +L:x \in H^2\}$.

So $H^1 \subseteq (H^2 + L)/L$, and clearly for some $\bar C$-closed set
$B \subseteq C_* \cup S$ of cardinality $< \mu$ (see before
$\oplus_1$) we have $H^2 \subseteq H^3 := H_B$, see $\circledast_5$.
So because $\{H_B:B \subseteq C_* \cup S,|B| < \mu\}$ is cofinal, and it
is $\bar C$-closed (inside $[C_* \cup S]^{< \mu}$, clearly it suffices to
prove that $(H_B + L)/L$ for $\bar C$-closed $B \in [C_* \cup S]^{< \mu}$.

By clause (f) of the claim's assumption there is $\bar u = \langle
u_\delta:\delta \in B \cap S\rangle$ such that $u_\delta \in J$ and
$\delta_1 \ne \delta_2 \in B \cap \delta \wedge \iota_1 \in (\kappa
\backslash u_{\delta_1}) \wedge \iota_2 \in (\kappa \backslash
u_{\delta_2}) \Rightarrow \beta(\delta_i,\iota_1) \ne
\beta(\delta_2,\iota_2)$ recalling $\bar C$ is normal.
The rest should be clear.] 

By $(*)_7 + (*)_8 + (*)_9$ we are done.
\end{PROOF}

\begin{claim}
\label{5e.21}
1) In \ref{5e.14} if $\mu = \lambda$,
(i.e., for $\bar C$ the cardinality and degree of freeness coincide,
naturally in clause (b) we have $J \in \SP_\chi(R)$)
we can also deduce $\lambda \in \text{\rm sp}_\lambda(R)$.
\newline
2) In \ref{5e.14}, it suffices to assume
\mn
\begin{enumerate}
\item[$\circledast'$]   as in $\circledast$ of
\ref{5e.14} omitting (d) and strengthening clause (b) to
\sn
\item[${{}}$]  $(b)' \quad \kappa \in \text{\rm sp}_{\le \lambda,\mu}(R)$,
see Definition \ref{0f.07}
\sn
\item[${{}}$]  $(c)' \quad$ like (c) but $\bar C$ is tree-like, that
  is, $\alpha \in C_{\delta_1} \cap C_{\delta_2} \Rightarrow
  C_{\delta_1} \cap \alpha = C_{\delta_2} \cap \alpha$.
\end{enumerate}
\end{claim}

\begin{PROOF}{\ref{5e.21}}
This should be clear.  
\end{PROOF}

\begin{claim}
\label{5e.28}
1) For $R = \Bbb Z$, we have
\mn
\begin{enumerate}
\item[$(a)$]  $J^{\text{\rm bd}}_{\aleph_0}$ belongs to 
{\rm SP}$_{\aleph_0}(R)$
\sn
\item[$(b)$]  $J^{\text{\rm bd}}_{\aleph_1}$ belongs to 
{\rm SP}$_{\aleph_1}(R)$ 
\sn
\item[$(c)$]  $J^{\text{\rm bd}}_{\aleph_1 * \aleph_0}$ belongs to
{\rm SP}$_{\aleph_1}(R)$
\sn
\item[$(d)$]  if $2^{\aleph_0} = \aleph_1$ or $2^{\aleph_1} <
2^{\aleph_2}$ \then \, $J^{\bd}_{\aleph_2}$ belongs to {\rm SP}$_{\aleph_2}(R)$
\sn
\item[$(e)$]  if $2^{\aleph_0} = \aleph_1$ or $2^{\aleph_1} <
2^{\aleph_2}$ \then \, $J^{\bd}_{\aleph_2 * \aleph_1}$ belongs to
{\rm SP}$_{\aleph_2}(R)$.
\end{enumerate}
\mn
2) Similarly for $R$ a proper subring of $\bbQ$.
\end{claim}
\bigskip

\begin{remark}
\label{5e.30}
1) If we want the proof of TDU$_\mu$ to be more direct,
we have to add Hom$(G_{\kappa +1}/G_\kappa) = 0$, otherwise we have to
``iterate". 

\noindent
2) Claim \ref{5e.28} does not seem new but we could not find a direct quote.  
Clauses (b),(c) follows essentially from \cite{Sh:98} and clauses 
(d),(e) are the parallel for $\aleph_2$ instead of $\aleph_1$; 
we can continue for higher $\aleph_i$'s inductively.

\noindent
3) This is closely related to ``$G$ is derived from ${\mat F}$", see
   \ref{1f.15}.

\noindent
4) Can we use this to prove TDU$_{\lambda,\aleph_{\omega+1}}(\bbZ)$
for some $\lambda$?  Can we do it assuming CH?  Can we do it assuming
there $k < \omega$ such that $2^{\aleph_\ell} = 
\aleph_{\ell +1}$ for $\ell < k$?
\end{remark}

\begin{PROOF}{\ref{5e.30}}
\underline{Proof of \ref{5e.30}}

For part (1) let $R = \bbZ$ and $a \in \bbZ$ be a prime, 
$a_n = a$ (or we can use, e.g. $a_n = n!$), for part (2) let $a \in R$ be a
prime such that $\frac{1}{a} \notin R$ and $a_n=a$; but we could use any
$\langle a_n:n < \omega\rangle$ such that $a_n R \subset R$.  We have
to check Definition \ref{5e.7}.  Note that here the $r$ in Definition
\ref{5e.7} is \wilog \,, see Remark \ref{5e.9}(1).
\bigskip

\noindent
\underline{Clause (a)}:

Let $G_{\omega + 1}$ be the abelian group generated by $\{x_n,y_n:n <
\omega\}$ freely except for the equations
\mn
\begin{enumerate}
\item[$\boxplus_1$]  $a_n y_{n+1} = y_n - x_n \text{ for } n < \omega$.
\end{enumerate}
\mn 
Let $G_n = R x_n$ and $G_\omega = \oplus\{R x_k : k < \omega\}$.

Letting $a_{< n} = \prod\limits_{\ell <n} a_\ell$ so that $a_0 = 1$, we have
$G_{\omega +1} \models a_{<(n+1)} y_{n+1} = y_0 + \sum\limits_{\ell
\le n} a_{< \ell} x_\ell$.  We now define $h \in \Hom(G_\omega,R)$ 
by choosing $h(x_n)$ by induction on $n$ so that:
if $b \in \bbZ$ and $r \in R \backslash \{0\}$ 
then for some $n$, computing in $\bbQ$, the sum $r(b +
\sum\limits_{\ell \le n} a_{< \ell} h(x_\ell))$ is not in $a_{<(n+1)}
R$, i.e. not divisible by $a_{<(n+1)}$ in $R$.  In fact the set of
sequences $\langle h|x_n|:n < \omega\rangle \in {}^\omega \bbZ$ for
which this fails is meagre.
\medskip

\noindent
\underline{Clause (b)}:   Let $\eta_\alpha \in {}^\omega 2$ for 
$\alpha < \omega_1$  be pairwise distinct.  Let $G_{\omega_1+1}$
be the abelian group freely generated by $\{x_i:i < \omega_1\} \cup
\{y_\eta:\eta \in {}^{\omega >} 2\} \cup \{z_{\alpha,n}:\alpha <
\omega_1,n < \omega\}$ freely except for the equations
\mn
\begin{enumerate}
\item[$\boxplus_2$]  $a_n z_{\alpha,n+1} = z_{\alpha,n} - 
y_{\eta_\alpha \restriction n} - x_{\omega \alpha +n}$ for $\alpha <
\omega_1,n < \omega$.
\end{enumerate}
\mn
For $\alpha < \omega_1$ let $G_\alpha := R x_\alpha$ and
$G_{\omega_1} = \oplus\{R x_\beta:\beta < \omega_1\}$.
\medskip

\noindent
\underline{Clause (c)}:  As in clause (b) note that
for $A \in J$ we let $G_A = 
\bigoplus\{Rx_{\omega \alpha + n}:(\alpha,n) \in A\}$.
\medskip

\noindent
\underline{Clause (d)}:

For each $\alpha < \omega_2$ let 
$\langle \varrho_{\alpha,\varepsilon}:\varepsilon <
\omega_1\rangle$ be a sequence of pairwise distinct members of
${}^\omega 2$.
Let $\langle \nu_\alpha:\alpha < \omega_2\rangle$ be a sequence of
increasing functions from $\omega_1$ to $\omega_1$ of length
$\omega_1$ such that for all $\alpha < \beta < \omega_2$ for some 
$\varepsilon < \omega_1$ we have
$\{\nu_\alpha(\zeta):\zeta \in [\varepsilon,\omega_1)\} \cap
\{\nu_\beta(\zeta):\zeta \in [\varepsilon,\omega_1)\} = \emptyset$.  

Let $G_{\omega_2+1}$ be the $R$-module generated by

\begin{equation*}
\begin{array}{clcr}
X = \{z_{\alpha,\varepsilon,n}:\alpha < \omega_2,\varepsilon <
\omega_1,n < \omega\} &\cup \, \{y_\zeta:\zeta < \omega_1\} \\
  &\cup \, \{x_{\alpha,\varrho}:\alpha < \omega_2,\varrho \in 
{}^{\omega >}2\} \cup \{t_\alpha:\alpha < \omega_2\}
\end{array}
\end{equation*}

\mn
freely except for the equations, 
\mn
\begin{enumerate}
\item[$\boxplus_3$]  $a_n z_{\alpha,\varepsilon,n+1} = 
z_{\alpha,\varepsilon,n} - y_{\nu_\alpha(\omega \varepsilon +n)} -
x_{\alpha,\varrho_{\alpha,\varepsilon} \rest n} - t_{\omega_1 \alpha +
\omega \varepsilon +n}$ for 
$\alpha < \omega_2,\varepsilon < \omega_1,n < \omega_0$.
\end{enumerate}
\mn
For $\alpha < \omega_2$ let $G_\alpha = \oplus\{R t_\beta:\beta \in
[\omega_1 \alpha,\omega_1 \alpha + \omega_1)\}$ and $G_{\omega_2} =
\oplus\{G_\alpha:\alpha < \omega_2\}$.
\mn
\begin{enumerate}
\item[$\boxplus_4$]  $G_{\omega_2 +1}/G_{\omega_2}$ is $\aleph_2$-free.
\end{enumerate}
\mn
Why?  Let $H_* = \oplus\{R y_\varepsilon:\varepsilon < \omega_1\}$
and for $\alpha < \omega_2$ we let $H_\alpha$ be the subgroup of
$G_{\omega_2+1}$ generated by $G_{\omega_2} \cup H_* \cup
\{z_{\alpha,\varepsilon,n}:\varepsilon < \omega_1,n < \omega\} \cup
\{x_{\alpha,\varrho}:\varrho \in {}^{\omega >}2\}$.

For $\alpha \le \omega_2$ let 
$H_{< \alpha} = \Sigma\{H_\beta:\beta < \alpha\}$.  Then clearly
$G_{\omega_2 +1} = H_{< \omega_2}$ and $\langle H_{<\alpha}:\alpha \le
\omega_2\rangle$ is $\subseteq$-increasing continuous.  
Hence it suffices to prove for $\alpha < \aleph_2$
\mn
\begin{enumerate}
\item[$\boxplus^{4.1}_\alpha$]  $H_{< \alpha}/G_{\omega_2}$ is free.
\end{enumerate}
\mn
Why?  Without loss of generality $\alpha \ge \omega_1$, let $\langle
\beta(\xi):\xi < \omega_1 \rangle$ list $\{\beta:\beta < \alpha\}$ with
no repetitions.  We can easily find a sequence $\bar\zeta = \langle
\zeta_\beta:\beta < \alpha\rangle$ such that the sets $\cU_\beta :=
\{\nu_\beta(\varepsilon):\varepsilon \in [\zeta_\beta,\omega_1)\}$ for
$\beta < \alpha$ are pairwise disjoint.  Without loss of generality
the ordinal power
$\omega^\omega$ divide $\zeta_\beta$ for every $\beta < \omega_1$ 
and we let $\cU = \omega_1
\backslash \cup\{\cU_\beta:\beta < \alpha\}$.  Moreover, \wilog \,
$\xi_1 < \xi_2 \Rightarrow \text{ Rang}(\nu_{\beta(\xi_1)}) \cap
\{\nu_{\beta(\xi_2)}(\varepsilon):\varepsilon \in
[\zeta_{\beta(\xi_2)},\omega_1)\} = \emptyset$.

For $\xi \le \omega_1$ let $H_{\alpha,\xi}$ be the subgroup of $H_{<
\alpha}$ generated by

\begin{equation*}
\begin{array}{clcr}
G_{\omega_2} &\cup \{z_{\gamma,\varepsilon,n}:\gamma \in 
\{\beta(\zeta):\zeta < \xi\} \text{ and } \varepsilon < \omega_1,
n < \omega\} \\
  &\cup \{y_\gamma:\gamma \in \cU\} \\
  &\cup \{y_{\nu_\gamma(\varepsilon)}:\varepsilon \in
[\zeta_\gamma,\aleph_1) \text{ for some } \gamma \in \{\beta(\zeta):\zeta
< \xi\}\} \\
  &\cup \{x_{\gamma,\varrho}:\gamma \in \{\beta(\zeta):\zeta <
\xi\}$ and $\varrho \in {}^{\omega >}2\}.
\end{array}
\end{equation*}

\mn
So $G_{\omega_2} \subseteq 
H_{\alpha,0} = \bigoplus\{Ry_\zeta:\zeta \in \cU\} \oplus
G_{\omega_2}$ hence $H_{\alpha,0}/G_{\omega_2}$ is free; also
$H_{\alpha,\omega_1} = H_{< \alpha}$ and $\langle
H_{\alpha,\xi}:\xi \le \omega_1\rangle$ is $\subseteq$-increasing continuous.
Hence it suffices to prove, for each $\xi < \omega_1$, that
$H_{\alpha,\xi+1}/H_{\alpha,\xi}$ is free.  Let $H'_{\alpha,\xi}$ be
the subgroup of $H_{\alpha,\xi +1}$ generated by $H_{\alpha,\xi} \cup
\{x_{\beta(\xi),\varrho}:\varrho \in {}^{\omega >}2\}$.  Now
$H_{\alpha,\xi} \subseteq H'_{\alpha,\xi} \subseteq
H_{\alpha,\xi+1}$.  It is easy to see that $H'_{\alpha,\xi}/H_{\alpha,\xi}$ is
countable and free.

Also $H_{\alpha,\xi +1}/H'_{\alpha,\xi}$ is free, in fact
$\{z_{\beta(\xi),\varepsilon,n} + H_{\alpha,\xi}:\varepsilon \in
[\zeta_{\beta(\xi)},\omega_1),n < \omega\}$ is a free basis.  Putting
  those together $\boxplus^{4.1}_\alpha$ holds hence $\boxplus_4$ is true.
\mn
\begin{enumerate}
\item[$\boxplus_5$]   some $h_0 \in \Hom(G_{\omega_2},R,R)$
has no extension $h_2 \in \Hom(G_{\omega_2+1},{}_R R)$.
\end{enumerate}
\mn

Why?  For $\alpha < \omega_2$ let $\omega_\alpha = \{t_{\omega_1 \alpha +
\varepsilon}:\varepsilon < \omega_1\}$ and $Y_\alpha =
\{y_{\nu_\alpha(\varepsilon)}:\varepsilon < \omega_1\}$.

For $\ell=1,2$ let $K^\ell_\alpha$ be the subgroup of 
$G_{\omega_2 +1}$ generated by:
\mn
\begin{enumerate}
\item[$\bullet$]  $\{y'_{\alpha,\varepsilon}:\varepsilon < \omega_1\}$
  when $\ell=1$ and $y'_{\alpha,\varepsilon} =
  y_{\nu_\alpha(\varepsilon)} + t_{\omega_1 \cdot \alpha + \varepsilon}$
\sn
\item[$\bullet$]  $\{x_{\alpha,\rho}:\rho \in {}^{\omega >}2\} \cup
\{y'_{\alpha,\varepsilon}:\varepsilon < \omega_1\}$ for $\ell=2$
\sn
\item[$\bullet$]  $\{z_{\alpha,\varepsilon,n}:\varepsilon < 
\omega_1,n < \omega\} \cup \{x_{\alpha,\rho}:\rho \in {}^{\omega >}2\}
\cup \{y'_{\alpha,\varepsilon}:\varepsilon < \omega_1\}$ 
when $\ell=3$ so
\sn
\item[$\bullet$]   $K^1_\alpha \subseteq K^2_\alpha \subseteq
  K^3_\alpha \subseteq G_{\omega_2+1}$.  
\end{enumerate}
\mn
Let $L^\ell_\alpha = \text{ Hom}(K^\ell_\alpha,\bbZ)$
for $\ell=1,2,3$.

Let $L_\alpha = \{f \rest K^1_\alpha:f \in L^3_\alpha\}$.  Clearly 
$L_\alpha$ is a submodule of $L^1_\alpha$.
As in the proof of clause (b), $L_\alpha \nsubseteq L^1_\alpha$, see
\cite{Sh:98}, \cite{EM02}.
Let $u_\alpha = u(\alpha) = \Rang(\nu_\alpha)$.  We now define a function 
$\bold F_\alpha:{}^{u(\alpha)}R \rightarrow L^1_\alpha/L_\alpha$ 
as follows: for $f \in {}^{u(\alpha)}R$ let 
$g_f \in \text{ Hom}(K^1_\alpha,{}_R R)$ be
defined by $g_f(y'_{\alpha,\varepsilon}) =
f(\nu_\alpha(\varepsilon))$ and then 
$\bold F_\alpha(f) = g_f + L_\alpha \in L^1_\alpha/L_\alpha$.
Obviously
\mn
\begin{enumerate}
\item[$(*)_{5.1}$]  $\bold F_\alpha$ is a homomorphism from
${}^{u(\alpha)}R$ onto $L^1_\alpha/L_\alpha$.
\end{enumerate}
\mn
Now consider
\mn
\begin{enumerate}
\item[$(*)_{5.2}$]  it suffices to find $\bar g^* = \langle g^*_\alpha:\alpha <
\omega_2\rangle$ such that $g^*_\alpha \in L^1_\alpha$ and for every
$f \in {}^{\omega_1} R$ for some $\alpha < \omega_2$ we have $\bold
F_\alpha(f \rest u_\alpha) \ne g^*_\alpha + L_\alpha$.
\end{enumerate}
\mn
Why is $(*)_{5.2}$ enough?    Let $f_\alpha \in {}^{u(\alpha)}R$
be such that $\bold F_\alpha(f_\alpha) = g^*_\alpha + L_\alpha$.  
We define $h_0 \in \text{ Hom}(G_{\omega_2},R)$ by:
\mn
\begin{enumerate}
\item[$(*)_{5.3}$]  $h_0(t_{\omega_1 \alpha + \varepsilon}) =
-f_\alpha(\nu_\alpha(\varepsilon))$ for $\alpha < \omega_2,\varepsilon
< \omega_1$.
\end{enumerate}
\mn
Toward contradiction assume $h_2 \in \text{ Hom}(G_{\omega_2+1},R)$
extends $h_0$.  Define the function $f:\omega_1 \rightarrow R$ by
$f(\varepsilon) = h(y_\varepsilon)$.  
Now for each $\alpha < \omega_2$, clearly $h_2 \rest K^1_\alpha \in
\Hom(K^1_\alpha,{}_R R) = L^1_\alpha$ but $K^1_\alpha \subseteq
K^3_\alpha \subseteq G_{\omega_2 +1}$ hence by the choice of
$L_\alpha$ we have $h_2 \rest K^1_\alpha \in L_\alpha$.

Now let $f'_\alpha = f \rest u_\alpha \in {}^{u(\alpha)}R$ so
$f'_\alpha(\nu_\alpha(\varepsilon)) =
h_2(y_{\nu_\alpha(\varepsilon)})$ for $\varepsilon < \omega_1$.
Recall that $\varepsilon < \omega_1 \Rightarrow 
-f_\alpha(\nu_\alpha(\varepsilon)) = h_0(t_{\omega_1 \alpha +
\varepsilon}) = h_2(t_{\omega_1 \alpha + \varepsilon})$ by
$(*)_{5.3}$ and by $h_2 \supseteq h_0$ respectively.  
So $f''_\alpha := f'_\alpha - f_\alpha
\in {}^{u(\alpha)}R$ satisfies $f''_\alpha(\nu_\alpha(\varepsilon)) =
f'_\alpha(\nu_\alpha(\varepsilon)) + f_\alpha(\nu_\alpha(\varepsilon))
= h_2(y_{\nu_\alpha(\varepsilon)}) - h_2(t_{\omega_1 \alpha +
\varepsilon}) = h_2(y'_{\alpha,\varepsilon})$, henace $g_{f''_\alpha} =
h_2 \rest K^1_\alpha$ which (as we said above) belongs to $L_\alpha$.  It
follows that $g_{f'_\alpha} - g_{f_\alpha} \in L_\alpha$ that is
$\bold F_\alpha(f'_\alpha) = \bold F_\alpha(f_\alpha) \in 
L^3_\alpha/L^1_\alpha$, hence by the choice of
$f_\alpha$ above, $\bold F_\alpha(f'_\alpha) = g^*_\alpha + L_\alpha$,
but $f'_\alpha = f \rest u_\alpha$.

As this holds for every $\alpha < \omega_2$, the function 
$f$ contradicts the present
assumption that $\langle g^*_\alpha:\alpha < \omega_2\rangle$ 
are as in $(*)_{5.2}$, so there is no $h_2$ as above, hence 
indeed it suffices to find
\mn
\begin{enumerate}
\item[$\bullet$]  $\bar g^*$ as in $(*)_{5.2}$.
\end{enumerate}
\mn
Why does such $\bar g^*$ exists?  The proof splits into cases.
\bigskip

\noindent
\underline{Case 1}:  $2^{\aleph_1} < 2^{\aleph_2}$

By renaming \wilog \, :
\mn
\begin{enumerate}
\item[$\odot$]  $\cup\{u_\alpha:\alpha < \omega_2\} = \omega_1$.
\end{enumerate}
\mn
We note that $\{\langle \bold F_\alpha(f \rest u_\alpha):\alpha <
\omega_2\rangle:f \in {}^{\omega_1}R\}$ is a subset of
$\prod\limits_{\alpha < \omega_2} L^1_\alpha/L_\alpha$ but the former has
cardinality $\le |R|^{\aleph_1} \le 2^{\aleph_1}$ and the latter 
has cardinality $\ge 2^{\aleph_2}$ (actually equal) but we are
assuming $2^{\aleph_1} < 2^{\aleph_2}$ in the present case,
so indeed we can find $\langle g_\alpha:\alpha <
\omega_2\rangle \in \prod\limits_{\alpha <\omega_2} L_\alpha$ which is
$\ne \langle \bold F_\alpha(f \rest u_\alpha):\alpha <
\omega_2\rangle$ for every $f \in {}^{\omega_1}R$.
\bigskip

\noindent
\underline{Case 2}:  $2^{\aleph_0} = \aleph_1$

\Wilog \, $\rho_{\alpha,\varepsilon} = \rho_\varepsilon$ for $\alpha <
\omega_2,\varepsilon < \omega_1$.

Now choose $\bar\nu$ such that:
\mn
\begin{enumerate}
\item[$\odot_1$]  $(a) \quad \bar\nu = \langle
\nu_\alpha:\alpha < \omega_2\rangle$
\sn
\item[${{}}$]  $(b) \quad \nu_\alpha:\omega_1 \rightarrow
\omega_1$ is increasing
\sn
\item[${{}}$]  $(c) \quad$ if $\beta < \alpha < \omega_2$ then for
some $\varepsilon < \omega_1$ we have $\nu_\alpha \rest 
(\omega \varepsilon + \omega) =$

\hskip25pt $\nu_\beta \rest (\omega \varepsilon + \omega)$ but
$\nu_\alpha(\omega \varepsilon + \omega) \ne
\nu_\beta(\omega \varepsilon + \omega)$
\sn
\item[${{}}$]  $(d) \quad$ if $\alpha \ne \beta$ then
$\Rang(\nu_\alpha) \cap \Rang(\nu_\beta)$ is countable.
\end{enumerate}
\mn
[Why?  E.g. choose $\nu_\alpha$ by induction on $\alpha <
\omega_2$ so that $\Rang(\nu_\alpha)$ is a non-stationary subset of $\omega_1$
and the relevant parts of (a)-(d) hold.]

Now choose $h_*$ such that
\mn
\begin{enumerate}
\item[$\odot_2$]  $(a) \quad h_*:{}^{\omega_1>}2 \rightarrow 
{}^\omega R$
\sn
\item[${{}}$]  $(b) \quad$ let $h^*_n:{}^{\omega_1 >}2 \rightarrow
 R$ for $n < \omega$ be such that $h_*(\nu) = \langle
h^*_n(\nu):n < \omega\rangle$
\sn
\item[${{}}$]  $(c) \quad$ if $\varepsilon < \omega_1,\varrho \in
{}^{\omega \cdot \varepsilon + \omega}2,\varrho_\ell = \nu \char 94 \langle \ell \rangle$ for
$\ell=0,1$ then the following set of 

\hskip25pt  equations is not solvable in $R$
\sn
\begin{enumerate}
\item[$\bullet$]  $a_n z_{n+1} = z_n - (h^*_n(\varrho_1) - h^*_n(\varrho_0))$
for $n < \omega$.
\end{enumerate}
\end{enumerate}
\mn
This is as in the proof of case (b).

Now we choose $h_0$ satisfying:
\mn
\begin{enumerate}
\item[$\odot_3$]  $h_0$ is the homomorphism from $G_{\omega_2}$ to
${}_R R$ such that
\sn
\item[${{}}$]  $\bullet \quad h_0(t_{\omega_1 \alpha + \omega
\varepsilon +n}) = h^*_n(\nu_\alpha \rest (\omega \cdot \varepsilon + 
\omega +1))$.
\end{enumerate}
\mn
Now toward a contradiction assume that $h \in \Hom(G_{\omega_2 +1},{}_R R)$
extends $h_0$.  We define a two-place relation $E$ on $\omega_2$ by:
\mn
\begin{enumerate}
\item[$\odot_4$]  $\alpha E \beta$ \Iff \, 
\sn
\item[${{}}$]  $(a) \quad \nu_\alpha \rest \omega = \nu_\beta \rest
\omega$
\sn
\item[${{}}$]  $(b) \quad h_2(x_{\alpha,\varrho}) =
h_2(x_{\beta,\varrho})$ for $\varrho \in {}^{\omega >}2$.
\end{enumerate}
\mn
Clearly $E$ is an equivalence relation with $\le 2^{\aleph_0}$
equivalence classes, so in our case $\aleph_1$ equivalence classes hence
there are $\alpha \ne \beta$ such that $\alpha E \beta$.  By
$\odot_1(c)$ there is $\varepsilon$ such that $\nu_\alpha \rest
(\omega \cdot \varepsilon + \omega) = \nu_\beta \rest (\omega \cdot
\varepsilon + \omega)$ and
$\nu_\alpha(\omega \cdot \varepsilon + \omega) \ne 
\nu_\beta(\omega \cdot \varepsilon + \omega)$.  \Wilog \, 
$\nu_\alpha(\omega \cdot \varepsilon + \omega)=1$ and $\nu_\beta(\omega
\cdot \varepsilon + \omega) = 0$.

For each $n$, consider the equations in $\boxplus_3$ for
$(\alpha,\varepsilon,n),(\beta,\varepsilon,n)$; apply $h_2$ and
subtract them.  The $h(y_{\nu_\alpha(\omega \varepsilon +n)}) -
h(y_{\nu_\beta(\omega \varepsilon +n)})$'s cancel by the choice of
$\varepsilon$.  Also the $h_2(x_{\alpha,\varrho_\varepsilon \rest n}) -
h_2(x_{\beta,\varrho_\varepsilon \rest n})$ cancel because $\alpha E \beta$.  

Lastly, by the choice of $h_0$ recalling $h_0 \subseteq h_2$ we have
$h_2(t_{\omega_1 \cdot \alpha + \omega \cdot \varepsilon
 +n}) - h_2(t_{\omega_1 \cdot \beta + \omega \cdot \varepsilon +1}) 
= h^*_n(\nu_\alpha \rest (\omega \cdot \varepsilon +
 \omega +1) - h^*_n(\nu_\beta \rest (\omega \cdot \varepsilon + \omega
 +1))$.   Hence the substitution
 $z_n \mapsto h_2(z_{\alpha,\varepsilon,n}) -
 h_2(z_{\beta,\varepsilon,n})$ solves the equations in $\odot_2(c)$ for
\mn
\begin{enumerate}
\item[$\bullet$]  $\varrho_1 = \nu_\alpha \rest (\omega \cdot
  \varepsilon + \omega +1),\varrho_0 = \nu_\beta \rest (\omega \cdot
  \varepsilon + \omega +1)$.
\end{enumerate}
\mn
So we get a contradiction to $\odot_2(c)$
\bigskip

\noindent
\underline{Clause (e)}:

As in clause (d).
\end{PROOF}

\begin{conclusion}
\label{5e.32}
1) TDU$_\lambda$ holds, \when \, 
BB$(\lambda,\mu,2^{(2^{\aleph_1})^+},J)$, where $J \in
\{J^{\bd}_{\aleph_0},J^{\bd}_{\aleph_1 * \aleph_0}\}$ and
$\cf(\lambda) > \aleph_1$.

\noindent
2) Similarly for $\BB(\lambda,\mu,(2^{\Dom(J)},2^{\Dom(J)}),J)$.
\end{conclusion}

\begin{PROOF}{\ref{5e.32}}
1) By \ref{5e.14} and \ref{5e.28}.

\noindent
2) Similarly by \ref{5e.43} below and \ref{5e.28}.
\end{PROOF}

\begin{remark}
1) The number, $2^{(2^{\aleph_1})^+}$ of colours 
is an artifact of the proof.  Actually 2 and even 
the so-called ``$1/\theta$ colours" (as in \cite[Ap,\S1]{Sh:f},
\ref{0p.15}(2)) should suffice, see \ref{5e.9}.

\noindent
2) See \ref{1f.13}.  But we can quote in \S0 cases of BB
with 2 instead of $\beth_4$ or just $2^{(2^{\aleph_1})^+}$ colours.
\end{remark}

\noindent
We can get more than in \ref{5e.14}.
\begin{definition}
\label{5e.38}
For cardinals $\lambda,\theta,\sigma$ for $\iota \in \{0,1\}$ let 
$\SP^{3 + \iota}_{\lambda,\theta,\sigma}(R)$ be the set
of ideals $J$ on some $\kappa$ such that for every $r \in R \backslash
\{0\}$ some pair $(\bar G,\bar h)$ witnesses it for $r$ where ``$(\bar
G,\bar h)$ witness $\SP^{3 + \iota}_{\lambda,\theta,\sigma}(R)$ for $r$" means:
\mn
\begin{enumerate}
\item[$\oplus$]  $(a) \quad \bar G = \langle G_i:i < \kappa +1 +
\sigma\rangle$ is a sequence of $R$-modules each of 

\hskip25pt cardinality $\le \lambda$
\sn
\item[${{}}$]  $(b) \quad G_\kappa = \oplus\{G_i : i < \kappa\}$ 
and $\zeta < \sigma \Rightarrow G_\kappa \oplus {}_R R \subseteq 
G_{\kappa +1+ \zeta}$
\sn
\item[${{}}$]  $(c) \quad$ if $u \in J$ and $\zeta < \sigma$, then
$G_{\kappa +1 + \zeta}/\oplus\{G_i : i \in u\}$ is a $\theta$-free 

\hskip25pt left $R$-module
\sn
\item[${{}}$]  $(d) \quad G_i$ is a $\theta$-free left $R$-module (for
$i < \kappa$ hence for $i < \kappa + 1 + \sigma$)
\sn
\item[${{}}$]  $(e) \quad \bar h = \langle h_\zeta:\zeta <
\sigma\rangle$ and $h_\zeta$ is a homomorphism from $G_\kappa$ to ${}_R R$ 

\hskip35pt for $\zeta < \sigma$
\sn
\item[${{}}$]  $(f)_0 \quad$ if $\iota=0$ for every homomorphism $h$
  from $G_\kappa$ to ${}_R R$ there is $\zeta < \sigma$ 

\hskip25pt  such that
\sn
\item[${{}}$]  \hskip20pt $\bullet \quad$ no homomorphism 
$h^+$ from $G_{\kappa +1+\zeta}$ to ${}_R R$ satisfies\footnote{the
computation ``$h(x) + h_\zeta(x) \cdot r$" is in the ring $R$.}

\hskip40pt  $x \in G_\kappa \Rightarrow h^+(x) = h(x) + h_\zeta(x) r$
\sn
\item[${{}}$]  $(f)_1 \quad$ if $\iota=1$ then for every homomorphism $h$
  from $G_\kappa$ to ${}_R R$ there is 

\hskip25pt  $\varepsilon < \sigma$ such
  that for every $\zeta < \sigma,\zeta \ne \varepsilon$ we have
\sn
\item[${{}}$]  \hskip20pt $\bullet \quad$ the same as $\bullet$ from above.
\end{enumerate}
\end{definition}

\begin{claim}
\label{5e.43}
A sufficient condition for $\TDU_{\lambda,\mu}(R)$ 
(i.e., there is a $\mu$-free left $R$-module $G$ of
cardinality $\lambda$ with $\Hom_R(G,R) = \{0\})$ is $\circledast_0$
and also $\circledast_1$ where:
\mn
\begin{enumerate}
\item[$\circledast_0$]  $(a) \quad R$ is a ring with unit ($1 = 1_R$)
\sn
\item[${{}}$]  $(b) \quad J \in \SP^3_{\chi,\theta,\sigma}(R)$ 
is an ideal on $\kappa$
\sn
\item[${{}}$]  $(c) \quad \bar C = \langle C_\delta:\delta \in
S\rangle$ is such that $\otp(C_\delta) = \kappa$ and $C_\delta
\subseteq \delta$
\sn
\item[${{}}$]  $(d) \quad \lambda > |R| + \chi$ is regular or at least
$\cf(\lambda) > |R| + \chi$ and $\mu > \kappa$
\sn
\item[${{}}$]  $(e) \quad \BB(\lambda,\bar C,
(2^{|R| + \chi},\sigma),J)$, see Definition \ref{0p.14}(1)
\sn
\item[${{}}$]  $(f) \quad \bar C$ is $(\mu,J)$-free (but see
  \ref{1f.13})
\sn
\item[$\circledast_1$]  similarly replacing clauses $(b),(e)$ by
  $(b)',(e)'$ where
\sn
\item[${{}}$]  $(b)' \quad J \in \SP^4_{\chi,\theta,\sigma}(R)$
\sn
\item[${{}}$]  $(e)' \quad \BB(\lambda,\bar C,\langle
  2^{|R|+\chi},\sigma\rangle,J)$, see \ref{0p.14}(2).
\end{enumerate}
\end{claim}

\begin{PROOF}{\ref{5e.43}}
Assuming $\circledast_\iota$, the proof is similar to 
the proof of \ref{5e.14} with some
changes.  First of all, instead of $\circledast_1$ we use
\mn
\begin{enumerate}
\item[$\circledast'_0$]  let $(\bar G^r,\bar h^r)$ witness
Definition \ref{5e.38} for $r \in R \backslash \{0\}$
\sn
\item[$\circledast'_1$]  $G_*$ is a $\mu$-free $R$-module 
and for some ordinal $\varepsilon(*) \le |R| + \kappa$
\sn
\begin{enumerate}
\item[$(a)$]  $G_* = \bigoplus\{G_{*,\varepsilon}:\varepsilon <
\varepsilon(*)\}$ is a $\mu$-free $R$-module $G_{*,\varepsilon}$ of
cardinality $\le \chi$ for $\varepsilon < \varepsilon(*)$
\sn
\item[$(b)$]  if $r \in R \backslash \{0\}$, then for some
sequence  $\bar G^r = \langle G^r_j:j < \kappa + 1 + \sigma\rangle$ 
as in \ref{5e.38} we have: if $j < \kappa$ then $\varepsilon(*) =
\otp\{\varepsilon < \varepsilon(*):G^r_j \cong_{f^*_{r,j}} 
G_{*,\varepsilon}\}$ hence
\sn
\item[$(c)$]  $|G_*| \le \chi + \kappa + |R|$.
\end{enumerate}
\end{enumerate}
\mn
Secondly, after $\circledast_8$ we choose $\langle
\eta_\delta:\delta \in S_i\rangle$ such that $\eta_\delta \in
{}^\kappa \varepsilon(*)$ and $j < \kappa \Rightarrow
G_{*,\eta_\delta(j)} \cong G^{r_i}_j$.

Thirdly, we choose $\langle \zeta^1_\delta:\delta \in S_i\rangle$ such
that:
\mn
\begin{enumerate}
\item[$\circledast'_{9.1}$]  $(a) \quad \zeta^1_\delta < \sigma$
\sn
\item[${{}}$]  $(b) \quad$ if $h \in \Hom(H_*,{}_R R)$,
\underline{then} for unboundedly many $\delta \in S_i$ we 

\hskip25pt have:
$\zeta^1_\delta \ne \bar{{\bold c}}^1_\delta(h \restriction
\bigcup\limits_{\alpha \in C_\delta} G^*_\alpha)$ - see below
\sn
\item[$\circledast_{9.2}$]  for $\delta \in S_i$ and $h \in 
\Hom(K_{< \delta},{}_R R)$, we define ${\bold c}^1_\delta(h)$ 
to be the minimal $\zeta < \sigma$ satisfying $\odot^i_{\delta,\zeta}$ 
below, and zero if there is no such $\zeta$
\sn
\item[$\odot^i_{\delta,\zeta}$]   there is $f \in 
\Hom(G^{r_i}_{\kappa +1 +\zeta},{}_R R)$ such that:
\sn
\begin{enumerate}
\item[$(\alpha)$]   $f(z) = r_i$,
\sn
\item[$(\beta)$]   if $j < \kappa$, then $x \in G^{r_i}_j
\Rightarrow f(x) = h(f^*_{r_i,j}(x))$.
\end{enumerate}
\end{enumerate}
\mn
The rest is similar.  
\end{PROOF}

\begin{conclusion}
\label{5e.47}
Assume that $J^{\bd}_{\kappa_n \times \omega} \in 
\SP_{\lambda_n,\theta_n}(R)$ and $\kappa_n < \kappa_{n+1}$ for 
$n < \omega$.  \Then , for some $\lambda$, for every large enough $n$, 
$\TDU_{\lambda,\theta_n^{+ \omega +1}}$ holds.
\end{conclusion}

\begin{remark}
If we use \cite{Sh:460}, then we need ``$\sum\limits_{n} \kappa_n$ is
strong limit" but instead we use \cite{Sh:829}.
\end{remark}

\begin{PROOF}{\ref{5e.47}}
We shall use \ref{5e.14} freely.

Let $\mu \in {\bold C}_{\aleph_0}$ be greater than $\lambda_n$ for
each $n$, and let $\sigma_n < \mu$ be large enough.

\bn
\underline{Case 1}:  There is $\lambda'$ such that 
$\lambda' < 2^\mu < 2^{\lambda'}$.

Then we can apply \ref{d.11} getting even a $\mu^+$-free Abelian group.
\bn
\newline
\underline{Case 2}:  $2^\mu$ is singular \underline{or} 
just there is a $\mu^+$-free
${\mat F} \subseteq {}^\omega \mu$ of cardinality $2^\lambda$.

By \ref{2b.98}(2).
\bn
\newline
\underline{Case 3}:  Neither Case 1 nor Case 2.

By Theorem \ref{h.7} $\lambda = 2^\mu = 
\lambda^{< \lambda}$ and $\lambda = 
\tcf(\prod\limits_{m < \omega} \lambda_m,<_{J^{\bd}_\omega})$ for some regular
$\lambda_m < \mu$ increasing with $m < \omega$ and let $\langle f_\alpha:\alpha
< \lambda\rangle$ exemplify this.  Let $S_{\ged} = 
S^{\ged}_{\bar f}$ - see \ref{2b.113}
and $S'_{\ged} =  \{\delta \in S_{\ged}:\cf(\delta) > \aleph_0$ and $\delta$
is divisible by $\mu\}$.

For each $n < \omega,\delta \in S_* = S'_{\ged} \cap
S^\lambda_{\kappa_n}$, let $C_{\delta,n}$ be a club of $\delta$ of
order type $\kappa_n$ and let

\[
C^n_\delta = \{\mu^\alpha + \eta_\delta(n):\alpha \in C_\delta \text{
and } n < \omega\}
\]

\mn
So $\langle C^n_\delta:\delta \in S^n_\delta\rangle$ is a strict
$(\lambda,\kappa_n)$-ladder system, i.e. $\otp(C^n_\delta) = 
\kappa,C^n_\delta \subseteq \delta = \sup(C^n_\delta)$.
By \ref{2b.111} we know that
$\bar C^n$ is $(\kappa^{+ \kappa_n}_n,J^\kappa_{\kappa_n \times
\omega})$-free (see Definitions \ref{0p.6}(3) and \ref{1.3.14}).
Now by \cite[1.10]{Sh:775}, \cite[3.1]{Sh:829} it follows that 
for every $n$ large enough,
we have $\BB(\lambda,\bar C^n,(\lambda,\theta_*),\kappa_n)$, where $\theta_* <
\mu$ is large enough.  
\end{PROOF}

\begin{conclusion}
\label{5e.53}
If the ideal $J = J^{\bd}_\kappa$ belongs to 
$\SP_{\lambda,\mu}(R)$ \then \, $\TDU_\mu$ holds.
\end{conclusion}

\begin{PROOF}{\ref{5e.53}}
Left to the reader.  
\end{PROOF}

\begin{remark}
Now we can check all the promises from \S0.
\end{remark}
\newpage

\bibliographystyle{alphacolon}
\bibliography{lista,listb,listx,listf,liste,listy}

\end{document}